\documentclass[10pt]{amsart}

\usepackage{cite}
\usepackage{amsmath}
\usepackage{latexsym}
\usepackage{amscd}
\usepackage{mathrsfs}
\usepackage{graphicx}
\usepackage{url}
\usepackage{enumitem}
\usepackage{mathrsfs}
\usepackage{amssymb}
\usepackage{amsmath}
\usepackage{amsthm}
\usepackage{listings}
\usepackage{version}
\usepackage{subfigure}
\usepackage{yhmath}
\usepackage{hyperref}

\usepackage[latin1]{inputenc}

\newtheorem{thm}{Theorem}[section]
\newtheorem{lemma}[thm]{Lemma}
\newtheorem{prop}[thm]{Proposition}

\newtheorem{rem}[thm]{\bf Remark}
\newtheorem{defn}[thm]{Definition}

\newcommand{\rank}{{\rm rank}}

\newcommand{\CC}{\mathbb{C}}
\newcommand{\RR}{\mathbb{R}}
\newcommand{\PP}{\mathbb{P}}

\newcommand{\TT}{\mathcal{T}}

\newcommand{\mb}[1][]{\mathbf}
\newcommand{\m}[1]{\mb{m_{#1}}}
\newcommand{\x}{\mb{x}}

\newcommand{\X}{\mb{X}}

\newcommand{\bs}[1][]{\boldsymbol}
\newcommand{\bt}{\boldsymbol{\tau}}
\newcommand{\eqn}{\begin{eqnarray}}
\newcommand{\feqn}{\end{eqnarray}}
\newcommand{\Sing}{\operatorname{Sing}}

\title[The algebro-geometric study of TOA maps]{The algebro-geometric study of range maps}

\author[M. Compagnoni]{Marco Compagnoni}
\address[M. Compagnoni]{Dipartimento di Matematica, Politecnico di Milano, Via Bonardi 9, I-20133 Milan, Italy}
\email{marco.compagnoni@polimi.it}

\author[R. Notari]{Roberto Notari}
\address[R. Notari]{Dipartimento di Matematica, Politecnico di Milano, Via Bonardi 9, I-20133 Milan, Italy}
\email{roberto.notari@polimi.it}

\author[A.A. Ruggiu]{Andrea Alessandro Ruggiu}
\address[A.A. Ruggiu]{Department of Mathematics, Link\"oping University, SE-58183 Link\"oping, Sweden}
\email{andrea.ruggiu@liu.se}

\author[F. Antonacci]{Fabio Antonacci}
\address[F. Antonacci]{Dipartimento di Elettronica, Informazione e Bioingegneria, Politecnico di Milano, Piazza L. Da Vinci 32, I-20133 Milan, Italy}
\email{fabio.antonacci@polimi.it}

\author[A. Sarti]{Augusto Sarti}
\address[A. Sarti]{Dipartimento di Elettronica, Informazione e Bioingegneria, Politecnico di Milano, Piazza L. Da Vinci 32, I-20133 Milan, Italy}
\email{augusto.sarti@polimi.it}

\begin{document}

\begin{abstract}
Localizing a radiant source is a widespread problem to many scientific and technological research areas. E.g. localization based on range measurements stays at the core of technologies like radar, sonar and wireless sensors networks. In this manuscript we study in depth the model for source localization based on range measurements obtained from the source signal, from the point of view of algebraic geometry. In the case of three receivers, we find unexpected connections between this problem and the geometry of Kummer's and Cayley's surfaces. Our work gives new insights also on the localization based on range differences.
\end{abstract}

\maketitle


\section{Introduction}\label{sec:intro}

Numerous problems in science and engineering are formulated in terms of distances (or ranges) between pairs of points in a given set. Without any claim to exhaustiveness, we could list a number of research fields where this problem plays a key role: radar and sonar technology, wireless sensor networks, statics, robotics, molecular conformation and dimensionality reduction in statistics and machine learning. We refer to \cite{Liberti2014} for a very useful introduction to these subjects.

In this manuscript, we focus on the problem of localizing a radiant source using range measurements. This is the prototypical problem in active localization technologies such as radar and active sonar \cite{Quazi1981}. In these situations, the measurements are time delays between the transmission of a pulse signal and the reception of its echo. Assuming the speed of propagation as known and costant, the Time Of Arrival (TOA) of the signal is directly related to the range between the source and the corresponding emitter/receiver. The goal of localization is to find the source position from the TOAs.

Localization is a fundamental problem for wireless sensor networks as well \cite{Weiser1993,Forman1994}. Indeed, the network routers must be updated of the positions of the sensors (e.g. smartphones), in order to adapt routes, frequencies, and network ID data accordingly.
It has been showed in \cite{Savvides2001} that the distance between any pair of sufficiently close sensors is strongly correlated to the battery charge used in their communications. Furthermore, by the fact that the positions of the fixed elements of the network (e.g routers and repeaters) are known, it follows that wireless sensor networks localization has many similarities to a multi--source localization problem based on TOA measurements.

In the mathematical literature, the problems involving ranges measurements have been intensively studied in the context of Euclidean Distance Geometry (DG) \cite{Liberti2014}. The fundamental problem in DG can be formulated in terms of the embeddability of a weighted graph $G=(V,E)$ into a suitable $k$--dimensional Euclidean space. Roughly speaking, one has to understand when the set $V$ of vertices $p_i,\ i=1,\dots,n$ actually corresponds to a set of points $\phi(p_i),\ i=1,\dots,n$ in $\RR^k,$ where the Euclidean distance $\Vert\phi(p_i)-\phi(p_j)\Vert$ is equal to the weight of the edge $e_{ij}\in E.$
The most important investigation tool in DG are the so called Cayley--Menger (CM) matrices \cite{Menger1928,Menger1931,Blumenthal1953}. A CM matrix is defined in terms of the squared distances between the points of $G$ and it has been proved that the embeddability of $G$ in $\RR^k$ is strictly related to the values of the determinants of certain CM matrices \cite{Schoenberg1935,Sippl1986}.

In localization problems the vertices $p_i$ of $G$ correspond to the sensors and sources, while the weighted edges $e_{ij}$ are the available range measurements. In this situation, the first important question is to determine the conditions for the graph $G$ (in particular the amount of data) that are necessary to have a unique realization in $\RR^k,$ where $k=2,3.$ To this respect, the DG approach proved its usefulness \cite{Aspnes2006,Eren2004}.

However, in real world applications range measurements are affected by noise. The engineering literature concerning the localization based on range measurements in such scenario is very wide (see for example \cite{Navidi1998,Caffery2000,Cheung2004b,Cheung2004a,Cheung2005,So2007,Wei2008,Beck2008}). The source estimation, in particular the Maximum Likelihood Estimation (MLE), is a non linear and non convex problem, therefore it is difficult to globally solve it. This is why researchers have studied many different approaches and algorithms that give rise to robust estimations, but that are suboptimal from a statistical point of view. The intend to offer different perspective on this important problem has been one of the main motivation of our work.

In this manuscript we propose the analysis of the range--based localization from the point of view of differential and algebraic geometry. Although these are not usual mathematical tools of DG and space--time signal processing, they are not unheard of in these fields.
In particular, in the current case of interest, we can mention the analysis of the performance of sensor networks localization proposed in \cite{Cheng2013} that is based on Information Geometry. On the other hand, the study of Cayley--Menger varieties in terms of algebraic geometry has been used to characterize certain properties of the embedding of graphs \cite{Borcea2002,Borcea2004}.

Here, we focus on the range--based source localization with two and three calibrated and synchronous sensors. Firstly, we define the stochastic model for range measurements, which encodes the range--based localization into a map from the Euclidean space containing source and receivers to the space of range measurements. We then offer a complete characterization of such a map, along the same line we adopted in \cite{Compagnoni2013a} for the analysis of the localization based on range differences, which are measured through Time Differences Of Arrival (TDOAs). First of all, we address the identifiability problem \cite{Bellman1970}, which is the statistical equivalent of the unique embeddability problem of DG.

Even more important, we describe in great detail the geometry of the sets of feasible range measurements. As already discussed in \cite{Compagnoni2013a}, we consider such analysis as a crucial step toward a deeper understanding of the localization issue.
Indeed, it is well known from Information Geometry \cite{Amari2000,Cheng2013} that the properties of the range statistical model are strictly related to the shape of its measurements set.
This approach proved successful in many similar situations, for example in characterizing the accuracy of an estimator and in designing optimal estimation algorithms \cite{Kobayashi2013}.
Furthermore, the description of the range deterministic model in algebraic geometry terms paves the way to the analysis of the corresponding statistical model using Algebraic Statistics \cite{Pistone2001,Drton2009}. For example, our study becomes instrumental for understanding and computing the MLE. As far as that is concerned, we can cite the recent definition of Euclidean Distance Degree \cite{Draisma2013,Friedland2014} for algebraic statistical models, which is an indicator of the complexity of the MLE, and the continuous development of techniques for polynomial optimization over semialgebraic varieties \cite{Blekherman2012}.

Finally, the last goal we achieve in the manuscript is the comparison between the range and range difference localization models. The localization methods based on range difference are very widespread and popular in the literature of space--time signal processing. The comparison between the performances of the range and range differences localization algorithms has been proposed in works like \cite{Shin2002,Shen2008}.
Here we show how one can obtain the same characterization of the deterministic range difference model given in \cite{Compagnoni2013a} starting from the results on the range model. This analysis could be very useful in the cases where the receivers are actually measuring ranges, but for some reason such as the lack of synchronization between the sources and the receivers, the available information is completely contained in the associated range differences. A relevant example is the Global Positioning System (GPS), where the direct use of the ranges is impossible due to the bias between the clocks of the receiver and the satellites.

The manuscript is organized as follows. In Section \ref{sec:Geometry of TOA maps} we define the deterministic and statistical models for source localization based on range measurements.
We take $r$ receivers and a source in a Euclidean plane, that we identify with $\RR^2$. We encode the deterministic model in the TOA map $\bs{\TT_r}$ from $\RR^2$ to the space of range measurements $\RR^r.$ We then rewrite the most relevant problems of source localization in terms of $\bs{\TT_r}.$

Section \ref{sec:r2d2} is devoted to the study of $\bs{\TT_2}.$ We prove that this map is locally injective at every point outside the line $r$ containing the two receivers, but it is not injective due to the symmetry of the problem with respect to $r$. Furthermore, we prove that the image of $\bs{\TT_2}$ is an unbounded convex polyhedron. This Section can be considered as a warm up for the rest of the paper, wherein we adopt investigation techniques that are as simple as possible. Unfortunately, such mathematical tools are not as suitable when $r>2.$

Sections from \ref{sec:r3d2local} to \ref{sec:collinear-microphones} focuses on the study of $\bs{\TT_3}$ and they are the core of the manuscript. In Section \ref{sec:r3d2local} we complete the local analysis of the range map through the study of its Jacobian matrix. In Section \ref{sec:MK}, we translate the localization problem in terms of the exterior algebra formalism over the Minkowski space $\RR^{2,1}$. On one hand, this allows us to shorten the computations, secondly it permits the comparison between the range and range differences models that we give in Sections \ref{sec:TDOA-summary} and \ref{sec:TDOAvsTDOA}. Section \ref{sec:2Dr3} offers a complete description of the map $\bs{\TT_3}$ for the case of non-aligned receivers. As main results, in Subsection \ref{sec:ImsubsetKum} we prove that $\text{Im}(\bs{\TT_3})$ is contained in a Kummer's quartic surface and in \ref{sec:geompropIm} we give a detailed description of its geometric properties, both from an algebraic and a differential point of view. In particular, we investigate the link between the properties of the Kummer's and some distinguished sets in the physical Euclidean plane. Similar results are derived for the case of aligned receivers in Section \ref{sec:collinear-microphones}.

From Section \ref{sec:TOA-3D} to Section \ref{sec:TDOAvsTDOA} we give some applications of the previous results. In particular, in Section \ref{sec:TOA-3D} we glance at the source localization problem in the Euclidean space $\RR^3.$
This can be seen as the first step into the extension of our results to the cases of general spatial configurations of source and receivers. Sections \ref{sec:TDOA-summary} and \ref{sec:TDOAvsTDOA} are devoted to the study of the connection between the range and range difference models. In order to improve the readiness of the manuscript, Section \ref{sec:TDOA-summary} provides a brief analysis of the range difference model following \cite{Compagnoni2013a}. In Section \ref{sec:TDOAvsTDOA}, we compare the two models and we show how to derive the results on the range difference localization from the ones about the range case.

Finally, in Section \ref{sec:conclusion} we briefly discuss the impact of this work and draw some conclusions. Moreover, we describe possible future research directions that can take advantage of the analysis presented in this manuscript.

The mathematical techniques involved in the article mainly come from differential and algebraic geometry and from multilinear algebra. We invite the reader to Appendices A and B of \cite{Compagnoni2013a} for a concise introduction to the tools used in the paper, while we suggest \cite{Abraham1988,Beltrametti2009,doCarmo} as possible references.


\section{Physical model and mathematical description}\label{sec:Geometry of TOA maps}

In this manuscript, with the exception of Section \ref{sec:TOA-3D}, we consider range--based localization under the assumption that the source and the receivers are coplanar. This choice is in line with previous works \cite{Compagnoni2012,Bestagini2013,Compagnoni2013a,Compagnoni2013b,Compagnoni2016a} and it allows us to approach the problem with more progression and visualization effectiveness. Therefore, we take synchronized source and receivers, we assume that the receivers are placed at known locations and the signals propagate through a homogeneous medium in anechoic conditions. Under these hypotheses we can identify the physical space with the Euclidean plane, here referred to as the $x$--plane.

After choosing an orthogonal Cartesian coordinate system, the x--plane can be identified with $\mathbb{R}^{2}$. On this plane, $\m{i} = \left(x_{i},y_{i}\right), \ i=1,\dots,r$ are the positions of the receivers and $\mb{x} = \left(x,y\right)$ is the position of the source $S$. The corresponding displacement vectors are
\begin{equation}
\mb{d_i} \left(\mb{x}\right) = \mb{x} - \m{i}, \ \quad \ \mb{d_{ji}} = \m{j} - \m{i}, \ \quad \ i,j = 1, \dots, r, \ \ i \neq j,
\end{equation}
whose norms are $d_{i}\left(\mb{x}\right)$ and $d_{ji}$, respectively. For the sake of notation simplicity, in the following we will denote the norm of the vector $\mb{v}$ with $v$ and the corresponding unit vector is  $\mb{\tilde{v}} = \frac{\mb{v}}{v}$.

Without loss of generality, we assume the propagation speed of the signal in the medium to be equal to 1. This way, the measured range $\widehat{\mathcal{T}}_{i}\left(\mb{x}\right)$ for the sensor $\m{i}$ turns out to be equal to the range $d_{i}\left(\mb{x}\right)$ plus a measurement error $\varepsilon_{i}$:
\begin{align}
\widehat{\mathcal{T}}_{i}\left(\mb{x}\right) = d_{i}\left(\mb{x}\right) + \varepsilon_{i}, \quad i = 1,\dots,r.
\end{align}
Regardless of the measurement procedure, a wavefront originating from a source $S$ at $\x$ will yield a set of range measurements $\left(\widehat{\mathcal{T}}_{1}\left(\mb{x}\right), \dots, \widehat{\mathcal{T}}_{r}\left(\mb{x}\right)\right)$. As the measurement noise is a random variable, we are concerned with a stochastic model.
\begin{defn}\label{def:TOAmodel}
The range model is
\begin{align}
\bs{\widehat{\TT}_r}(\x) = \left(\widehat{\mathcal{T}}_{1}\left(\mb{x}\right), \dots, \widehat{\mathcal{T}}_{r}\left(\mb{x}\right)\right).
\end{align}
The deterministic part of this model is obtained by setting $\varepsilon_{i} = 0$ in $\bs{\widehat{\TT}_r}$, which gives us the range map
\begin{center}
$\begin{matrix} \bs{\TT_r}&: \mathbb{R}^{2}& \rightarrow & \mathbb{R}^{r}
\\ &\mb{x} &\mapsto& \left(d_{1}\left(\mb{x}\right), \dots, d_{r}\left(\mb{x}\right)\right).
\end{matrix}$
\end{center}
The target set is referred to as the $\mathcal{T}$--space and we denote its points with $\bs{\mathcal{T}}=\left(\mathcal{T}_{1}, \dots,\mathcal{T}_{r}\right)$.
\end{defn}

We observe that the range maps have been recently defined also in \cite{Ichiki2013}, where they have been studied from the point of view of differential geometry. As already done in the case of range difference--based localization \cite{Compagnoni2013a}, the scope of the manuscript is the study of the properties of the range map that are crucial for a deeper understanding of the range model. First of all, we are concerned in the identifiability problem, which can be formulated as follows.
\begin{itemize}
\item
Given $\bs{\mathcal{T}}$ in the $\mathcal{T}$--space, does there exist a source in the $x$--plane such that $\bs{\TT_r}(\x)  = \bs{\mathcal{T}}$, i.e. $ \bs{\mathcal{T}} \in \mbox{Im}(\bs{\TT_r})$?
\item
If $\mb{x}$ exists, is it unique, i.e. $ \vert \bs{\TT_r}^{-1}(\bs{\mathcal{T}}) \vert = 1$?
\end{itemize}
Subsequently, we want to obtain the coordinates of $ \mb{x}$ as a function of $\bs\TT$, or equivalently, the explicit inverse map $\bs{\TT_r}^{-1}.$
Finally, as we said in the Introduction, the properties of a statistical model depend on the geometry of the set of feasible measurements \cite{Amari2000}. We are therefore interested in describing in detail the geometry of $\mbox{Im}(\bs{\TT_r}).$
Our work takes into account the cases $r = 2,3$, which are the smallest values of $r$ for which $\bs{\TT_r}$ can be locally injective and injective, respectively, or equivalently the source position is uniquely discoverable (at least locally). Hence, in the following Sections we study the maps $\bs{\TT_2}$ and $\bs{\TT_3}$.


\section{The localization problem with two receivers}\label{sec:r2d2}
We begin our analysis by focusing on the simplest case of range--based source localization. We take a source in $\x$ and two sensors placed in $\m{1},\m{2}$ on the Euclidean plane $\RR^2.$ As shown in Figure \ref{sez2-fig1}, this problem is symmetric with respect to the line $r$ containing the sensors, therefore we immediately notice that it is not possible to uniquely locate the source if $x\notin r,$ unless we have some a--priori information on the half plane containing $\x.$ From a mathematical standpoint, the analysis of the range map $\bs{\TT_2}$ can be completed through the use of elementary instruments. On the other hand, the usefulness of such study goes beyond the scope of this section. Indeed, this allows us to gradually introduce several concepts that we will widely use throughout the entire manuscript. Notice that many of the following ideas are already known in the signal processing literature \cite{Caffery2000}.

By Definition \ref{def:TOAmodel}, two real numbers $\mathcal{T}_{1},\mathcal{T}_{2}$ give a point $\bs{\mathcal{T}}=\left(\mathcal{T}_{1},\mathcal{T}_{2}\right)\in\text{Im}(\bs{\TT_2})$ if, and only if, there exists a solution $\x\in\RR^2$ of the system
\begin{equation}\label{eq:sistTOA2}
\left\{
\begin{array}{l}
d_{1}\left(\mb{x}\right)=\TT_1\\
d_{2}\left(\mb{x}\right)=\TT_2
\end{array}\right..
\end{equation}
The two equations in \eqref{eq:sistTOA2} have a neat geometrical interpretation.
\begin{defn}\label{levelsets}
Let $\mathcal{T} \in \mathbb{R}$. The set
\begin{align}
A_{i}\left(\mathcal{T}\right) = \left\{\mb{x} \in \mathbb{R}^{2} \vert d_{i}(\x) = \mathcal{T} \right\}
\end{align}
is the level set of $d_{i}(\x)$ in the $x$--plane.
\end{defn}
\begin{rem}\label{rem:levelsets}
$A_{i}\left(\mathcal{T}\right)$ is a circle centered at $\m{i}$ and with radius $\mathcal{T}$, when $\mathcal{T} > 0$. If $\ \mathcal{T} < 0$ then $A_{i}\left(\mathcal{T}\right) = \varnothing$. Finally, if $\mathcal{T} = 0$ we have $A_{i}\left(0\right) = \left\{\m{i}\right\}$.
\end{rem}
A straightforward consequence of Remark \ref{rem:levelsets} is that $\TT_1,\TT_2$ are feasible only if
$
\mathcal{T}_{1}\geq 0\ \text{and}\
\mathcal{T}_{2}\geq 0,
$
which implies that $\text{Im}(\bs{\TT_2})$ is a subset of the first quadrant of $\RR^2.$ Under this assumption, we have that the localization of the source is geometrically equivalent to find the intersection of two given real circles. In particular, we have that $\bs{\mathcal{T}}\in\text{Im}(\bs{\TT_2})$ if, and only if, the intersection $A_1(\TT_1)\cap A_2(\TT_2)$ is not empty (see Figure \ref{sez2-fig1}).
\begin{figure}[htb]
\centering
\includegraphics[width=6cm,bb=206 342 405 449]{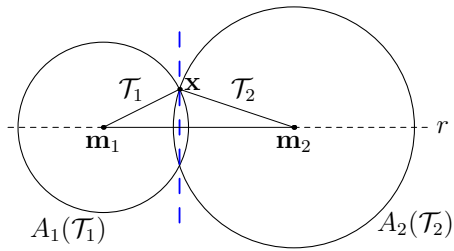}\\
\caption{Two ranges $\mathcal{T}_{1}$ and $\mathcal{T}_{2}$ generate two circles and their intersections are the solutions for the localization problem with two receivers. It is also possible to locate the source by intersecting the radical axis (the dashed blue line) and one of the two circles.}\label{sez2-fig1}
\end{figure}

Let us take the orthonormal basis of $\RR^2$ given by $\left(\mb{\tilde{d}_{21}},\ast\mb{\tilde{d}_{21}}\right),$ where we remind that the two dimensional Hodge operator $\ast$ is simply the counterclockwise rotation of $\frac{\pi}{2}.$ Then
$$
\mb{d_1}\left(\mb{x}\right) = a\, \mb{\tilde{d}_{21}} + b \ast\mb{\tilde{d}_{21}} 	
\qquad\text{and}\qquad
\mb{d_2}(\x) = - \mb{d_{21}} + \mb{d_1}(\x) = (a-d_{21})\, \mb{\tilde{d}_{21}} + b \ast\mb{\tilde{d}_{21}}.
$$
In order to work with algebraic equations, we square both sides of \eqref{eq:sistTOA2} and we obtain
\begin{equation}\label{eq:sistTOA22}
\left\{
\begin{array}{l}
a^{2} + b^{2}=\TT_1^2\\
d_{21}^{2} - 2a\ d_{21} + a^{2} + b^{2}=\TT_2^2
\end{array}\right.
\qquad\Rightarrow\qquad
\left\{
\begin{array}{l}
a^{2} + b^{2}=\TT_1^2\\
- 2a\ d_{21} =\TT_1^2-\TT_2^2+d_{21}^{2}
\end{array}\right..
\end{equation}
It is important to observe that the second equation of the last system is linear in the coordinates $a,b.$ Indeed, it defines the radical axis of the two circles $A_1(\TT_1),A_2(\TT_2)$ (see Figure \ref{sez2-fig1}). Similar reduction of the degree of the range equations holds for a general configuration of $r$ receivers and it is a key fact for the resolution of the problem. From \eqref{eq:sistTOA22} it follows
\begin{align*}
a = \frac{d_{21}^{2}+\mathcal{T}_{1}^{2}-\mathcal{T}_{2}^{2}}{2d_{21}}\,,
\qquad\qquad
b^{2} = \mathcal{T}_{1}^{2} - a^{2}\,.
\end{align*}

Since we are looking for the real solutions of the localization problem, it is necessary that $\mathcal{T}_{1}^{2} - a^{2} \geq 0$:
\begin{alignat*}{2}
& \mathcal{T}_{1}^{2} - \left(\frac{d_{21}^{2}+\mathcal{T}_{1}^{2}-\mathcal{T}_{2}^{2}}{2d_{21}} \right )^{2} = \\ &= \frac{\left(2d_{21}\mathcal{T}_{1}-d_{21}^{2}-\mathcal{T}_{1}^{2}+\mathcal{T}_{2}^{2} \right )\left(2d_{21}\mathcal{T}_{1}+d_{21}^{2}+\mathcal{T}_{1}^{2} - \mathcal{T}_{2}^{2}\right )}{4d_{21}^{2}} = \\ &= \frac{\left[\mathcal{T}_{2}^{2} - \left(\mathcal{T}_{1} - d_{21} \right )^{2} \right] \left[\left(\mathcal{T}_{1} + d_{21} \right )^{2} -\mathcal{T}_{2}^{2}\right ]}{4d_{21}^{2}} = \\ &= \frac{\left(\mathcal{T}_{2} + \mathcal{T}_{1} - d_{21} \right)\left(\mathcal{T}_{2} - \mathcal{T}_{1} + d_{21} \right )\left(\mathcal{T}_{1}+\mathcal{T}_{2}+d_{21} \right )\left(\mathcal{T}_{1} - \mathcal{T}_{2} + d_{21} \right)}{4d_{21}^{2}}\,.
\end{alignat*}
In the positive quadrant of the $\TT$--plane, it is not difficult to check that the last row in the above formula is non negative if, and only if, the following three inequalities are satisfied:
\begin{equation}\label{eq:TI}
\mathcal{T}_{1} - \mathcal{T}_{2} \leq d_{21}\,,\qquad
-\mathcal{T}_{1} + \mathcal{T}_{2} \leq d_{21}\qquad\text{and}\qquad
\mathcal{T}_{1} + \mathcal{T}_{2} \geq d_{21}\,.
\end{equation}

These are nothing else but the triangular inequalities. We can say that two real numbers $\mathcal{T}_{1},\mathcal{T}_{2}$ are the ranges between the source $\mb{x}$ and the sensors $\m{1},\m{2},$ respectively, if and only if $\mathcal{T}_{1}$, $\mathcal{T}_{2}$ and $d_{21}$ are the lengths of the sides of a triangle. When one of \eqref{eq:TI} holds as equality, the localization problem with two receivers has got only one solution. Clearly, this happens when $\m{1},\m{2}$ and $\x$ lie on the same line. According to Theorem \ref{rank} and Proposition \ref{geom}, in this case $\bs{\TT_2}$ is not locally injective because its Jacobian matrix has rank equal to 1 and the two circles intersects tangentially at $x$.

We summarize the above discussion in the following Theorem and we draw $\text{Im}(\bs{\TT_2})$ in Figure \ref{sez2-fig2}.
\begin{thm}\label{prop:caso2mic}
The image of $\bs{\TT_2}$ is the unbounded polyhedron $Q_{2}$ in the $\mathcal{T}$--plane defined by the triangular inequalities ~\eqref{eq:TI}. If $\bs{\mathcal{T}}= \left(\mathcal{T}_{1}, \mathcal{T}_{2}\right) \in Q_{2}$, then
\begin{align}
\left| \bs{\TT_2}^{-1}\left(\bs{\mathcal{T}}\right) \right| = \begin{cases}
2 \quad\text{if}\ \,\bs{\mathcal{T}} \in \mathring{Q}_{2}, \\
1 \quad\text{if}\ \,\bs{\mathcal{T}} \in \partial Q_{2}, \end{cases}
\end{align}
where $\mathring{Q}_{2}$  and $\partial Q_2$ are the interior and the boundary of $Q_2$, respectively.
Moreover, the solutions $\mb{x_\pm}$ of the localization problem satisfy
\begin{align}\label{eq:solT2}
\mb{d_1} (\mb{x_\pm}) = \frac{d_{21}^{2}+\mathcal{T}_{1}^{2}-\mathcal{T}_{2}^{2}}{2d_{21}}\, \mb{\tilde{d}_{21}} \pm \sqrt{\mathcal{T}_{1}^{2} - \left(\frac{d_{21}^{2}+\mathcal{T}_{1}^{2}-\mathcal{T}_{2}^{2}}{2d_{21}} \right )^{2}} \ast \mb{\tilde{d}_{21}}.
\end{align}
\end{thm}

\begin{figure}[htb]
\centering
\includegraphics[width=6cm,bb=206 303 405 488]{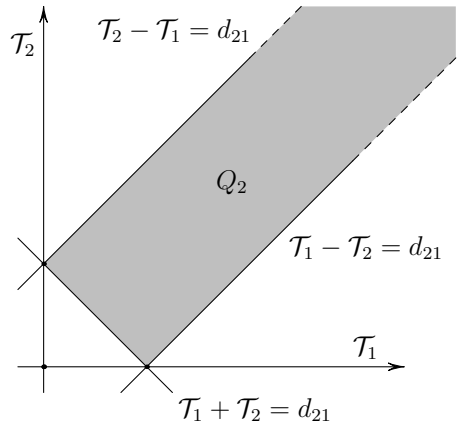}\\
\caption{According to Theorem \ref{prop:caso2mic}, the image of $\bs{\TT_2}$ is the unbounded polyhedron $Q_2$ defined by the triangular inequalities. In the grey region the map $\bs{\TT_2}$ is 2--to--1, while on its boundary it is 1--to--1.}\label{sez2-fig2}
\end{figure}

\begin{rem}\label{solform}
As we said at the beginning of the Section, the localization problem is symmetrical with respect to the line $r$ through $\m{1},\m{2}.$ This implies that there is uniqueness of localization if, and only if, the source $\x$ lies on that line. In such a case, however, the problem is ill conditioned. Indeed, in a noisy scenario, small errors on the measurements $\bs{\widehat{\TT}}$ could drastically modify the relative position of the two circles $A_1(\TT_1)$ and $A_2(\TT_2)$, e.g. their intersection could be either empty or two distinct points (see Figure \ref{sez2-fig1}). In the $\TT$--plane, this situation corresponds to $\bs{\mathcal{T}} \in \partial Q_{2}:$ a small perturbation can push $\bs{\mathcal{T}}$ either outside $Q_2$ or in its interior (see Figure \ref{sez2-fig2}).
\end{rem}


\section{Local analysis of $\bs{\TT_3}$}\label{sec:r3d2local}
In this section we study the local invertibility of the range map $\bs{\TT_3}$.
As extensively explained in \cite{Compagnoni2013a}, this is the first step toward the resolution of the problems stated at the end of Section \ref{sec:Geometry of TOA maps}. The main tool we use is the Inverse Function Theorem: if the Jacobian matrix of $\bs{\TT_3}$ is invertible at $\x$, then $\bs{\TT_3}$ is invertible in a neighborhood of $\x.$ For this reason we look for the \emph{degeneracy locus} of $\bs{\TT_3}$, i.e. the locus where the Jacobian matrix $J(\x)$ of $\bs{\TT_3}$ drops rank.

A first fact is that the component function $d_{i} (\x)$ of $\bs{\TT_3}$ is differentiable in $\RR^2 \setminus \m{i},$ for $i=1,2,3.$ Therefore, $\bs{\TT_3}$ is differentiable in $D=\RR^{2} \setminus \{\m{1},\m{2},\m{3}\},$ which becomes the domain of $J(\x).$ Moreover, the $i$--th row of $J(\x)$ is equal to
\begin{align}
\nabla d_{i} (\x)= \left(\frac{x-x_{i}}{d_{i}(\x)}, \frac{y-y_{i}}{d_{i}(\x)}\right) = \bs{\tilde{d}_i}(\x), \quad \ i = 1,2,3.
\end{align}
\begin{thm}\label{rank}
Let $J(\x)$ be the Jacobian matrix of $\bs{\TT_3}$ at $\mb{x} \in D$. We have the following cases:
\begin{enumerate}[label=(\roman{*}), ref=(\roman{*})]
\item
if $\m{1}$, $\m{2}$ and $\m{3}$ are not collinear, then
$$\rank \ J(\x) = 2, \quad \forall \mb{x} \in D;$$
\item
if $\m{1}$, $\m{2}$ and $\m{3}$ lie on the line $r$, then
\begin{equation*}
\rank \ J(\x) = \begin{cases}
1\quad \text{if}\ \,\mb{x} \in r \cap D, \cr
2\quad \text{otherwise}. \end{cases}
\end{equation*}
\end{enumerate}
\end{thm}
\begin{proof}
Assume $\mb{x} \neq \m{i}$, for $i = 1,2,3$. The rank of $J(\x)$ drops when its three rows are proportional each other. This happens if, and only if, the vectors $\bs{\tilde{d}_1}(\x),\bs{\tilde{d}_2}(\x),\bs{\tilde{d}_3}(\x)$ have the same direction. This is possible if, and only if, the three receivers $\m{1},\m{2},\m{3}$ and the point $\x$ are aligned.
\end{proof}

Theorem ~\ref{rank} has a nice geometric interpretation in terms of the intersection of the level sets $A_i(\TT_i),\ i=1,2,3$ given in Definition \ref{levelsets}. First of all we observe that, similarly to the case of $\bs{\TT_2}$ in Section \ref{sec:r2d2}, the source lies on the intersection $A_1(\TT_1)\cap A_2(\TT_2)\cap A_3(\TT_3)$. Moreover, we should not forget that two generic curves $\mathcal{C}_{1}$ and $\mathcal{C}_{2}$ meet transversally at a smooth point $P$ if their tangent lines at $P$ are different.
\begin{prop}\label{geom}
Let $\mb{x} \in A_{1}\left(\mathcal{T}_{1}\right) \cap A_{2}\left(\mathcal{T}_{2}\right) \cap A_{3}\left(\mathcal{T}_{3}\right)$. Then $A_{1}\left(\mathcal{T}_{1}\right)$, $A_{2}\left(\mathcal{T}_{2}\right)$ and $A_{3}\left(\mathcal{T}_{3}\right)$ do not meet transversally at $\mb{x}$ if, and only if, $\m{1}$, $\m{2}$, $\m{3}$ and $\mb{x}$ lie on the same line.
\end{prop}
\begin{proof}
Let $\mb{x} \in A_{i}\left(\mathcal{T}_{i}\right)$. Then, the vector $\bs{\tilde{d}_i}(\x)$ is orthogonal to the tangent line to $A_{i}\left(\TT_{i}\right)$ at $\mb{x}$. Hence, the level set $A_{i}\left(\TT_{i}\right)$, $i = 1,2,3$, do not meet transversally at $\mb{x}$, if and only if, the vectors $\bs{\tilde{d}_i}(\x)$ are proportional each other, that is to say $\m{1},\m{2},\m{3}$ and $\mb{x}$ lie on the same line.
\end{proof}
In Figure \ref{fig:circlesintersections} we depict two cases of transversal and non transversal intersection of the level sets. The full discussion about the existence and uniqueness of the source localization for $\bs{\TT_3}$ is contained in Sections \ref{sec:2Dr3} and \ref{sec:collinear-microphones}. However, from the two pictures it should be clear that we can uniquely locate a source using three non collinear receivers, while in the aligned case the localization behaves very similarly to the case of sources aligned with the two receivers $\mathbf{m}_1$ and $\mathbf{m}_2$ described in Section \ref{sec:r2d2}. In particular, the neighborhood of the degeneracy locus $r$ is a critical region for the localization problem, where small errors on the range measurements could give large errors on the source position. More generally, the study of $J(\x)$ plays an important role in the characterization of the accuracy of any statistical estimator of the source position, because of its strict relation with the Fisher Information Matrix \cite{Amari2000}. At this respect, we refer to the discussion contained in Section $9$ of \cite{Compagnoni2013a} regarding the range difference--based localization.
\begin{figure}[htb]
 \centering
 \subfigure[A case of unique source localization with transversal intersection.]
   {\includegraphics[width=6cm,bb=206 296 405 495]{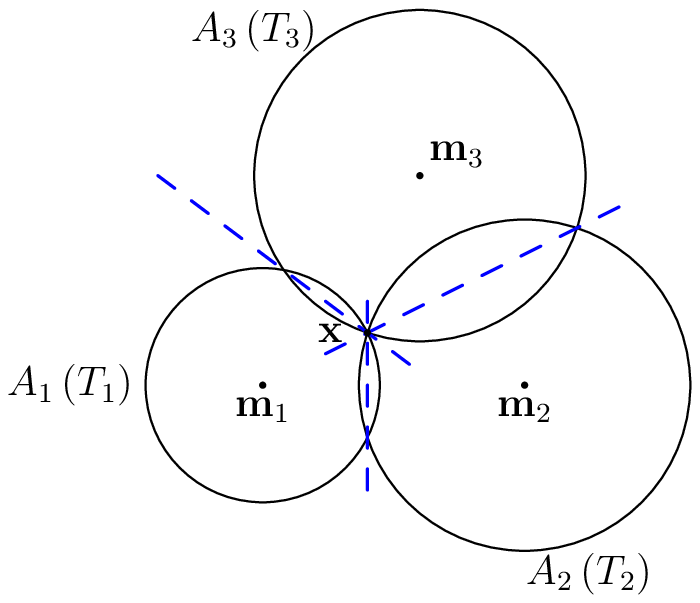}}
 \hspace{10mm}
 \subfigure[A case of unique source localization with non transversal intersection.]
   {\includegraphics[width=6cm,bb=206 319 405 472]{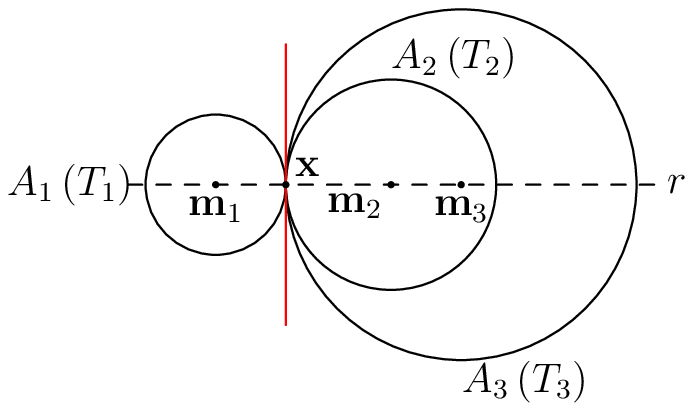}}
 \caption{The two cases described by both Theorem ~\ref{rank} and Proposition ~\ref{geom}. If the level sets do not meet transversally, then $\m{1},\m{2},\m{3}$ and $\mb{x}$ lie on the same line $r,$ which is the degeneracy locus of the range map $\bs{\TT_3}$. From Theorem \ref{rank}, in this case the map is not locally invertible.}\label{fig:circlesintersections}
 \end{figure}

Finally, Theorem \ref{rank} allows us to give the first result on the range measurements set $\text{Im}(\bs{\TT_3}).$
\begin{thm}
The image of $\bs{\TT_3}$ is locally a surface in the $\TT$--space.
\end{thm}
\begin{proof}
Let us assume that $\x$ is a point where the Jacobian matrix $J(\x)$
has rank $2$. Without loss of generality, we assume that
\begin{align*}
\ast \left(\nabla \mathcal{T}_{1}(\mb{\bar{x}}) \wedge \nabla \mathcal{T}_{2}(\mb{\bar{x}})\right) \neq 0.
\end{align*}
The map $\bs{\TT_3}$ can be written as
\begin{equation}\label{eq:TOAsystem}
\left\{\begin{matrix} d_{1}\left(\mb{x}\right) = \mathcal{T}_{1}
\\ d_{2}\left(\mb{x}\right) = \mathcal{T}_{2}
\\ d_{3}\left(\mb{x}\right) = \mathcal{T}_{3}
\end{matrix}\right.
\end{equation}
and $\bs{\bar{\TT}} = \left(\bar{\TT}_1,\bar{\TT}_2,\bar{\TT}_3\right) = \bs{\TT_3}(\mb{\bar{x}})$ is a solution of ~\eqref{eq:TOAsystem}.
By the Implicit Function Theorem, there exist three functions $x = x\left(\TT_1,\TT_2\right)$, $y = y\left(\TT_1,\TT_2\right)$ and $\TT_3 = \TT_3\left(\TT_1,\TT_2\right),$ defined in a neighborhood of $\left(\bar{\TT_1},\bar{\TT}_2\right)$ and taking values in a neighborhood of $\mb{\bar{x}}$ and $\TT_3,$ such that \eqref{eq:TOAsystem} is locally equivalent to
\begin{equation*}\label{eq:TOAsystem2}
\left\{\begin{matrix} x = x\left(\mathcal{T}_{1},\mathcal{T}_{2}\right)
\\ y = y\left(\mathcal{T}_{1},\mathcal{T}_{2}\right)
\\ \mathcal{T}_{3} = \mathcal{T}_{3}\left(\mathcal{T}_{1},\mathcal{T}_{2}\right)
\end{matrix}\right. \ .
\end{equation*}
Hence, the image of $\bs{\TT_3}$ is locally the graph of $\mathcal{T}_{3}\left(\mathcal{T}_{1},\mathcal{T}_{2}\right)$ and the claim follows.
\end{proof}
The surface $\text{Im}(\bs{\TT_3})$ will be carefully studied in Section \ref{sec:2Dr3} for the case of non collinear receivers and in Section \ref{sec:collinear-microphones} for the aligned sensors scenario.


\section{The range model in the Minkowski formalism}\label{sec:MK}
Minkoswki space and exterior algebra proved to be key tools in the study of range difference--based source localization \cite{Compagnoni2013a,Coll2009,Coll2012}. As explained in Section $4$ of \cite{Compagnoni2013a}, the introduction of a third variable (the time) in addition to the spatial ones has allowed the authors to  reformulate the 2D localization with range differences in terms of intersection of semialgebraic surfaces in the 3D spacetime. This way, through nice simplifications, one can partially linearize the involved equations and finally set the problem as the intersection of one null cone and several planes in $\RR^{2,1}$.

In the case of range measurements, we have already seen similar arguments in Section \ref{sec:r2d2}, where we showed that the range localization with two sensors is mathematically equivalent to the intersection of a line and a circle. In this section we derive similar results for the map $\bs{\TT_3},$ by adopting the same mathematical formalism used in \cite{Compagnoni2013a}. This choice has two main motivations. On the one side, exterior algebra is a powerful tool that allows us to work with compact expressions and that can be used in more general scenarios with a greater number of receivers and sources. On the other hand, as one of the main goals of the manuscript is to compare the range and range differences models and to obtain the properties of the latter from the ones of the former (see next Section \ref{sec:TDOAvsTDOA}), the comparison is much easier if the two models are studied by adopting the same formalism.

We start from system \eqref{eq:sys1}, whose solutions are the points in $A_1(\TT_1) \cap A_2(\TT_2) \cap A_3(\TT_3)$:
\begin{equation}\label{eq:sys1}
\left\{\begin{array}{l}
d_1(\x) = \TT_1 \\
d_2(\x) = \TT_2 \\
d_3(\x) = \TT_3
\end{array}\right. \ .
\end{equation}
As observed in Section \ref{sec:r2d2}, the three ranges are feasible only if $\TT_1,\TT_2,\TT_3\geq 0,$ hence from now on we work under this assumption.
We introduce an auxiliary variable $\mathcal{T}$, which represents the time of emission of the signal. Then, system ~\eqref{eq:sys1} is equivalent to
\begin{equation*}\label{eq:sys2}
\left\{\begin{array}{l}
d_1(\x) = \TT_1 - \TT \\
d_2(\x) = \TT_2 - \TT \\
d_3(\x) = \TT_3 - \TT \\
\TT = 0
\end{array}\right. \ ,
\end{equation*}
as we assume that the source emits the signal at $\TT = 0$. This is not an algebraic system, because of the presence of Euclidean distances. However, since the hypothesis $\TT_i\geq 0,\ i = 1,2,3.$ we can square both sides of the first three equations and get the equivalent polynomial system
\begin{equation}\label{eq:sys3}
\left\{\begin{array}{l}
d_1(\x)^2 - (\TT_1-\TT)^2 = 0 \\
d_2(\x)^2 - (\TT_2-\TT)^2 = 0 \\
d_3(\x)^2 - (\TT_3-\TT)^2 = 0 \\
\mathcal{T} = 0
\end{array}\right. \ .
\end{equation}
In geometric terms, we are considering the $3$--dimensional space described by the triple $\left(x,y,\mathcal{T}\right)$, and in particular we are studying the intersection of the three cones having $\mb{M_i} = \left(\m{i}, \mathcal{T}_{i}\right)$ as vertex, for $i = 1,2,3$ (first three equations of the system \ref{eq:sys3}) and the plane $\mathcal{T}=0$. The (spacetime) source position is among the solutions of this system, it has coordinates $\left(\bar{x},\bar{y},0\right)$ and it lies on the half--cone contained in the half--space $\mathcal{T} \leq \mathcal{T}_{i}$ for each $i$.

The equations in the last system involve expressions very similar to the Euclidean $3$--dimensional norms, up to the minus sign at each term related to time. This fact motivates us to adopt the $3$--dimensional Minkowski space formalism. In the following we use the conventions and the results exposed in \cite{Compagnoni2013a}, Appendix A.
We choose $\mb{e_1}$ and $\mb{e_2}$ as the unit vectors of the $x$- and $y$--axes, respectively, while we choose $\mb{e_3}$ as the unit vector of the $\mathcal{T}$--axis. Then, the inner product $\left \langle \mb{u}, \mb{v}\right \rangle$ of the vectors $\mb{u} = u_{1} \mb{e_1} + u_{2} \mb{e_2} + u_{3} \mb{e_3}$ and $\mb{v} = v_{1} \mb{e_1} + v_{2} \mb{e_2} + v_{3} \mb{e_3}$ is computed as
\begin{align*}
\left \langle \mb{u}, \mb{v} \right \rangle = u_{1}v_{1} + u_{2}v_{2} - u_{3}v_{3}
\end{align*}
and $\left\|\mb{u}\right\|^{2} = \left \langle \mb{u}, \mb{u} \right \rangle = u_{1}^{2} + u_{2}^{2} - u_{3}^{2}$ is the associated norm.  As the inner product for vectors having null component along $ \mb{e_3} $ is the standard Euclidean scalar product, we denote it with the symbol $ \cdot $ according to the standard notation used in literature.

We set the following notation in the $3$--dimensional Minkowski space $\mathbb{R}^{2,1}$. Given $\bs{\mathcal{T}}=(\mathcal{T}_1,\mathcal{T}_2,\mathcal{T}_3)$ in the $\mathcal{T}$--space, the receivers have coordinates
$
\mb{M_i}\left(\bs{\mathcal{T}}\right) = \left(x_{i},y_{i},\mathcal{T}_{i}\right),\ i = 1,2,3,
$
while $\mb{X} = \left(x,y,\mathcal{T}\right)$ is the generic point of $\mathbb{R}^{2,1}.$ The displacement vectors are then
$$
\mb{D_i}(\X,\bs{\TT}) = \X - \mb{M_i}(\bs{\TT})
\qquad\text{and}\qquad
\mb{D_{ji}}(\X,\bs{\TT}) = \mb{M_j}(\bs{\TT}) - \mb{M_i}(\bs{\TT}),
\qquad\text{for}\ i,j = 1,2,3\ \text{and}\ i \neq j.
$$
In this setting, we rewrite system ~\eqref{eq:sys3} as
\begin{equation}\label{eq:sys4}
\left\{\begin{array}{l}
\Vert\mb{D_1}(\X,\bs{\TT})\Vert^2 = 0 \\
\Vert\mb{D_2}(\X,\bs{\TT})\Vert^2 = 0 \\
\Vert\mb{D_3}(\X,\bs{\TT})\Vert^2 = 0 \\
\langle\mb{D_i}(\X,\bs{\TT}), \mb{e_3}\rangle = \TT_i
\end{array}\right. \ ,
\end{equation}
where in the last equation we can indifferently choose $i = 1,2,3$.

Since $\mb{D_j}(\X,\bs{\TT}) = \mb{D_i}(\X,\bs{\TT}) -\mb{D_{ji}}(\bs{\TT}) $
with $i,j = 1,2,3$ and $i \neq j$, it holds
\begin{align*}
\Vert\mb{D_j}(\X,\bs{\TT}) \Vert^2 = \Vert\mb{D_i}(\X,\bs{\TT}) \Vert^2 -
2\langle\mb{D_i}(\X,\bs{\TT}), \mb{D_{ji}}(\X,\bs{\TT})\rangle +
\Vert\mb{D_{ji}}(\X,\bs{\TT})\Vert^2.
\end{align*}
From \eqref{eq:sys4} we know that $\Vert\mb{D_i}(\X,\bs{\TT}) \Vert^2 = \Vert\mb{D_j}(\X,\bs{\TT})\Vert^2 = 0.$ Therefore, we obtain
\begin{equation}\label{eq:planefromcone}
\langle\mb{D_i}(\X,\bs{\TT}),\mb{D_{ji}}(\X,\bs{\TT})\rangle =
\frac{1}{2}\Vert\mb{D_{ji}}(\X,\bs{\TT})\Vert^2,
\end{equation}
or, equivalently,
\begin{equation}\label{eq:planefromcone2}
i_{\mb{D_i}(\X,\bs{\TT})} \mb{D_{ji}}(\X,\bs{\TT})^\flat =
\frac{1}{2} \Vert\mb{D_{ji}}(\X,\bs{\TT})\Vert^2.
\end{equation}
This simplification is the analogous in $\RR^{2,1}$ to the one made in passing between the two systems \eqref{eq:sistTOA22} of Section \ref{sec:r2d2}.
Thus, we can rewrite system \eqref{eq:sys4} and obtain the \emph{$i$--th Formulation} of the localization problem:
\begin{equation}\label{eq:Formulation_i}
\left\{
\begin{array}{l}
i_{\mb{D_i}(\X,\bs{\TT})}\mb{D_{ji}}(\X,\bs{\TT})^\flat =
\frac{1}{2}\Vert\mb{D_{ji}}(\X,\bs{\TT})\Vert^2 \\
i_{\mb{D_i}(\X,\bs{\TT})}\mb{D_{ki}}(\X,\bs{\TT})^\flat =
\frac{1}{2}\Vert\mb{D_{ki}}(\X,\bs{\TT})\Vert^2 \\
\Vert\mb{D_i}(\X,\bs{\TT})\Vert^2 = 0 \\
i_{\mb{D_i}(\X,\bs{\TT})}\mb{e_3}^\flat = \TT_i
\end{array}\right. \ ,
\end{equation}
where $i,j,k = 1,2,3$ are distinct each other. As we anticipated at the beginning of the section, we get an algebraic system, almost linear in $\X.$

Every equation in \eqref{eq:Formulation_i} defines a surface in $\RR^{2,1}.$
\begin{defn}
For each $i=1,2,3,$ we have:
\begin{enumerate}[label=(\roman{*}), ref=(\roman{*})]
\item
$C_i(\bs{\TT}) = \left\{\X\in\RR^{2,1} \mid
\Vert \mb{D_i}(\X,\bs{\TT})\Vert^2 = 0\right\}$ is the null cone with vertex $\mb{M_i}\left(\bs{\mathcal{T}}\right);$
\item
$\Pi^i_j(\bs{\TT}) = \left\{\X \in \RR^{2,1} \mid
i_{\mb{D_i}(\X,\bs{\TT)}} \mb{D_{ji}} (\bs{\TT})^\flat =
\frac{1}{2} \Vert \mb{D_{ji}}(\bs{\TT})\Vert^2\right\},\ i\neq j$ is the plane orthogonal to the displacement vector $\mb{D_{ji}}\left(\bs{\mathcal{T}}\right)$ through the middle point of the segment $\overline{\mb{M_i}\left(\bs{\mathcal{T}}\right)\mb{M_j}\left(\bs{\mathcal{T}}\right)}$;
\item
$\Pi = \left\{\mb{X} \in \RR^{2,1} \mid
i_{\mb{D_i}(\X,\bs{\TT})} \mb{e_3}^\flat = \TT_{i} \right\} =
\left\{\X \in \RR^{2,1} \mid \TT = 0 \right\}$ is the space plane.
\end{enumerate}
\end{defn}
\noindent Let $ i, j, k $ be three distinct integers between $ 1 $ and $ 3.$ The solution of ~\eqref{eq:Formulation_i} is the intersection of  $\Lambda(\bs{\TT}) = \Pi(\bs{\TT}) \cap \Pi^i_j(\bs{\TT}) \cap \Pi^i_k(\bs{\TT})$ with the cone $C_i(\bs{\TT})$. We spend the remaining part of this Section to study $\Lambda(\bs{\TT})$, while the analysis of $\Lambda(\bs{\TT})\cap C_i(\bs{\TT})$ stays at the core of Section \ref{sec:2Dr3}.

In the next Lemma we start by considering the intersection $\Pi(\bs{\TT}) \cap \Pi^i_j(\bs{\TT}).$
\begin{lemma}\label{4am}
For every $\bs{\TT} \in \RR^3$ and $i,j = 1,2,3$ with $i \neq j$, the intersection $L_{ji}(\bs{\TT}) = \Pi(\bs{\TT}) \cap \Pi^i_j(\bs{\TT})$ is a line. A parametric representation of $L_{ji}(\bs{\TT})$ is $\X(\lambda;\bs{\TT}) = \mb{L_0}(\bs{\TT}) + \lambda \mb{v}(\bs{\TT})$, where
\begin{align}
\mb{v}(\bs{\TT})= \ast (\mb{d_{ji}} \wedge \mb{e_3})
\end{align}
and the displacement vector of $\mb{L_0}(\bs{\TT})$ referred to $\m{i}$ is
\begin{align}\label{eq:displvect}
\mb{D_i}(\mb{L_0}(\bs{\TT})) =
\frac{d_{ji}^2-\TT_j^2+\TT_i^2}{2d_{ji}^2}\, \mb{d_{ji}} - \TT_i\,\mb{e_3}.
\end{align}
\end{lemma}
\begin{proof}
Since $\m{i} \neq \m{j},$ it follows that $\mb{D_{ji}}(\bs{\TT})$ and $\mb{e_3}$ are linearly independent, hence $L_{ji}(\bs{\TT})$ is a line. The equation of $L_{ji}(\bs{\TT})$ is:
\begin{align*}
i_{\mb{D_i}(\mb{X},\bs{\TT})}(\mb{D_{ji}}(\bs{\TT})^\flat\wedge\mb{e_3}^\flat)
= \frac{1}{2} \Vert \mb{D_{ji}}(\bs{\TT})\Vert^2\mb{e_3}^\flat -
\TT_i\, \mb{D_{ji}}(\bs{\TT})^\flat .
\end{align*}
A vector $\mb{v}(\bs{\TT})$ is parallel to the line $L_{ji}(\bs{\TT})$ if it is a solution of the homogeneous equation
\begin{align*}
i_{\mb{v}(\bs{\TT})}(\mb{D_{ji}}(\bs{\TT})^\flat \wedge \mb{e_3}^\flat) = 0.
\end{align*}
This is equivalent to (see Corollary A.6 of \cite{Compagnoni2013a})
\begin{align*}
\mb{v}(\bs{\TT}) = t \ast(\mb{D_{ji}}(\bs{\TT})^\flat\wedge\mb{e_3}^\flat)^\sharp
= t \ast (\mb{D_{ji}}(\bs{\TT}) \wedge \mb{e_3})
= t \ast (\mb{d_{ji}} \wedge \mb{e_3})
\end{align*}
for $t \in \RR$, $t \neq 0$. The first claim of the statement easily follows by setting $t = 1.$

Next, we need a point $\mb{L_0}\left(\bs{\mathcal{T}}\right)\in L_{ji}(\bs{\TT}),$ thus we take the intersection between $L_{ji}(\bs{\TT})$ and the plane $i_{\mb{D_i}(\X,\bs{\TT})}(\ast(\mb{D_{ji}}(\bs{\TT}) \wedge \mb{e_3})) = 0$. If we set
$$
\bs{\Omega_{ji}} =\mb{D_{ji}}(\bs{\TT}) \wedge \mb{e_3} \wedge
\ast ( \mb{D_{ji}}(\bs{\TT}) \wedge \mb{e_3})=
\mb{d_{ji}} \wedge \mb{e_3} \wedge \ast ( \mb{d_{ji}} \wedge \mb{e_3}),
$$
then the equation for $\mb{D_i}(\mb{L_0}(\bs{\TT}))$ is
\begin{align*}
i_{\mb{D_i}(\mb{L_0}(\bs{\TT}))} \bs{\Omega_{ji}}^\flat =
\left(\frac{1}{2} \left\|\mb{D_{ji}}(\bs{\TT})\right\|^{2} \mb{e_3}^\flat -
T_{i} \mb{D_{ji}}(\bs{\TT})^\flat \right) \wedge
\ast \left(\mb{D_{ji}}(\bs{\TT})^\flat \wedge \mb{e_3}^\flat\right).
\end{align*}
Hence (see Lemma A.7 of \cite{Compagnoni2013a})
\begin{align*}
\mb{D_i}(\mb{L_0}(\bs{\TT})) = \frac{1}{2\ast \bs{\Omega_{ji}}}
\ast \left[ \left( \Vert\mb{D_{ji}}(\bs{\TT})\Vert^{2} \mb{e_3} -
2\mathcal{T}_{i}\mb{D_{ji}}(\bs{\TT}) \right) \wedge
\ast ( \mb{D_{ji}}(\bs{\TT}) \wedge \mb{e_3}) \right],
\end{align*}
that can be simplified with straightforward computations to obtain formula \eqref{eq:displvect}.
\end{proof}
Now we can give the explicit expressions for $\Lambda(\bs{\TT}).$
\begin{prop}\label{Sol}
Let $ 1 \leq i, j, k \leq 3,$ different each other.

If $\m{1}$, $\m{2}$ and $\m{3}$ are not on a line, then $\Lambda(\bs{\TT})= \Pi(\bs{\TT}) \cap \Pi^i_j(\bs{\TT}) \cap \Pi^i_k(\bs{\TT})$ is the point $\mb{L}(\bs{\TT})$ given by
\begin{equation}\label{eq:point}
\mb{D_i}(\mb{L}(\bs{\TT})) =
\frac{\ast\left[\left((d_{ji}^2-\TT_j^2+\TT_i^2)\mb{d_{ki}}-
(d_{ki}^2-\TT_k^2+\TT_i^2)\mb{d_{ji}}\right)\wedge\mb{e_3}\right]}
{2\ast\left(\mb{d_{ji}}\wedge\mb{d_{ki}}\wedge\mb{e_3}\right)}-
\TT_{i}\,\mb{e_3}.
\end{equation}

If $\m{1}$, $\m{2}$ and $\m{3}$ lie on a line, let $\rho \in \mathbb{R}$ such that $\mb{d_{31}} = \rho\,\mb{d_{21}}$. If it holds
\begin{equation}\label{eq:comp}
\TT_3^2 = (1 - \rho) \TT_1^2 + \rho \TT_2^2 - \rho (1-\rho) d_{21}^2,
\end{equation}
then $\Lambda(\bs{\TT}) = L_{21}(\bs{\TT}) = L_{31}(\bs{\TT}) = L_{32}(\bs{\TT})$, otherwise $\Lambda(\bs{\TT}) = \varnothing$.
\end{prop}
\begin{proof}
If $\mb{d_{ji}}$ and $\mb{d_{ki}}$ are linearly independent, then $\mb{d_{ji}}$, $\mb{d_{ki}}$ and $\mb{e_3}$ are linearly independent too because the subspaces spanned by $\mb{d_{ji}},\mb{d_{ki}}$ and by $\mb{e_1},\mb{e_2}$ are equal. Hence, $\mb{D_{ji}}(\bs{\TT})$, $\mb{D_{ki}}(\bs{\TT})$ and $\mb{e_3}$ are also linearly independent and the three corresponding planes meet at a single point $\mb{L}(\bs{\TT})$.

As above, we set the three form
$$
\bs{\Omega_i} = \mb{D_{ji}}(\bs{\TT})\wedge\mb{D_{ki}}(\bs{\TT})\wedge\mb{e_3}
= \mb{d_{ji}}\wedge\mb{d_{ki}}\wedge\mb{e_3}.
$$
Then the equation for $\mb{D_i}(\mb{L}(\bs{\TT}))$ is
\begin{align*}
i_{\mb{D_i}(\mb{L}(\bs{\TT}))} \bs{\Omega_i}^\flat =
\frac{1}{2}\left(\Vert\mb{D_{ji}}(\bs{\TT})\Vert^2\mb{D_{ki}}(\bs{\TT})-
\Vert\mb{D_{ki}}(\bs{\TT})\Vert^2\mb{D_{ji}}(\bs{\TT})\right)\wedge\mb{e_3}+
\TT_i\mb{D_{ji}}(\bs{\TT})\wedge\mb{D_{ki}}(\bs{\TT}).
\end{align*}
It follows from Lemma A.7 in \cite{Compagnoni2013a} that $\mb{D_i}(\mb{L}(\bs{\TT}))$ is given by
\begin{equation*}
\mb{D_i}(\mb{L}(\bs{\TT})) = \frac{1}{2\ast\bs{\Omega_i}}
\ast \left[\left(\Vert\mb{D_{ji}}(\bs{\TT})\Vert^2\mb{D_{ki}}(\bs{\TT})-
\Vert\mb{D_{ki}}(\bs{\TT})\Vert^2\mb{D_{ji}}(\bs{\TT})\right)\wedge\mb{e_3}+
2\TT_i\mb{D_{ji}}(\bs{\TT})\wedge\mb{D_{ki}}(\bs{\TT})\right].
\end{equation*}
Thus, we get formula \eqref{eq:point} with straightforward computations.

Assume now that the receivers lie on a line. It follows that $\mb{d_{21}}$ and $\mb{d_{31}}$ are linearly dependent, i.e. there exists $\rho \in \RR$ such that $\mb{d_{31}} = \rho \mb{d_{21}}$. As usual, we assume that $\m{1}$, $\m{2}$ and $\m{3}$ are all distinct, which means $\rho \neq 0,1$.
The vectors $\mb{D_{21}}(\bs{\TT})$ and $\mb{e_3}$ are linearly independent and so $\Pi(\bs{\TT}) \cap \Pi_2^1(\bs{\TT})$ is the line $L_{21}(\bs{\TT})$.

By Lemma \ref{4am}, $L_{21}(\bs{\TT})$ is parallel to vector $ \ast\left(\mb{d_{21}} \wedge \mb{e_3}\right).$ On the other hand, it is rather straightforward to verify that $ \langle \mb{D_{31}}(\bs{\TT}), \ast( \left(\mb{d_{21}} \wedge \mb{e_3}\right) \rangle = 0,$ thus $L_{21}(\bs{\TT})$ is orthogonal to $\mb{D_{31}}(\bs{\TT}).$ Since also $ \Pi^1_3(\bs{\TT})$ is orthogonal to $ \mb{D_{31}}(\bs{\TT}),$ it follows that $L_{21}(\bs{\TT})$ is parallel to $ \Pi^1_3(\bs{\TT}).$
This implies that $\Lambda(\bs{\TT})\neq\varnothing$ if, and only if, $L_{21}(\bs{\TT})$ is contained in $ \Pi^1_3(\bs{\TT}).$ By direct substitution of the parametric equation of $L_{21}(\bs{\TT})$ from Lemma \ref{4am} in the equation
$$
\langle \mb{D_1}(\mb{X},\bs{\TT}), \mb{D_{31}}(\bs{\TT})\rangle = \frac{1}{2} \Vert\mb{D_{31}}(\bs{\TT})\Vert^{2}
$$
of $ \Pi^1_3(\bs{\TT}),$ we obtain the compatibility condition \eqref{eq:comp}.
\end{proof}
\begin{rem}
The point $\mb{L}\left(\bs{\mathcal{T}}\right)$ given by \eqref{eq:point} can be explicitly rewritten in the variables $\left(x,y,\mathcal{T}\right)$ as
\begin{equation}\label{eq:invmap}
\begin{pmatrix} x - x_{i}
\\ y - y_{i}
\end{pmatrix} = \frac{1}{2}
\begin{pmatrix}
x_{j} - x_{i} & y_{j} - y_{i}\\
x_{k} - x_{i} & y_{k} - y_{i}
\end{pmatrix}^{-1} \begin{pmatrix} d_{ji}^{2} + \mathcal{T}_{i}^{2} - \mathcal{T}_{j}^{2}
\\ d_{ki}^{2} + \mathcal{T}_{i}^{2} - \mathcal{T}_{k}^{2}
\end{pmatrix}
\end{equation}
and $\mathcal{T} = 0$, where, as usual, $ 1 \leq i,j,k \leq 3 $ are different from each other. Notice that, if $\bs\TT=(\TT_1,\TT_2,\TT_3)$ lies on the image of $\bs{\TT_3},$ formula \eqref{eq:invmap} solves the localization problem because it gives the source position $\x=\bs{\TT_3}^{-1}(\bs{\TT}).$
As observed in \cite{Compagnoni2013a}, formula \eqref{eq:invmap} can be used as the starting point and building block for a local error propagation analysis in the case of noisy measurements or even with sensor calibration uncertainty. Furthermore, it is interesting to notice that there exists a distinct inverse formula for each reference sensor. Although they are completely equivalent from a theoretical point of view, they behave differently if we account for numerical approximations. Indeed, an ongoing analysis based on the condition numbers of the matrices \eqref{eq:invmap} shows that there exists an optimal choice of the reference sensor for each given configuration of the receivers.
\end{rem}


\section{The image of $\bs{\TT_3}$ when the receivers are not collinear}\label{sec:2Dr3}
In this section we study of the set of feasible range measurements, i.e. the image $\textrm{Im}\left(\bs{\TT_3}\right)$ of the range map, under the assumption that $\m{1}$, $\m{2}$ and $\m{3}$ are not collinear. In Subsection \ref{sec:ImsubsetKum} we prove that $Im\left(\bs{\TT_3}\right)$ is contained in a Kummer's quartic surface. This kind of surfaces was discovered in the nineteenth century and their geometric properties have been investigated in the framework of projective algebraic geometry (we invite the reader to \cite{hudson} as a general reference on Kummer's surfaces). The relation between the space of range measurements and the Kummer's was first discussed in \cite{Berry1992}, where the author was interested in characterizing the points at rational distance from a given triangle. Here we directly link the geometry of Kummer's surfaces to the source localization problem. In Subsection \ref{sec:geompropIm} we use the results on such surfaces in order to describe in depth $\text{Im}\left(\bs{\TT_3}\right).$ In particular, we find an unbounded convex polyhedron $Q_3$ that contains $\text{Im}(\bs{\TT_3}).$ We prove that the intersection $\text{Im}(\bs{\TT_3})\cap Q_3$ is given by the only arcs of conics contained into $\text{Im}(\bs{\TT_3}),$ that are also the only points having zero Gaussian curvature. Such polyhedron is a good approximation of the convex hull $\mathcal{E}$ of $\text{Im}(\bs{\TT_3}).$ Finally, in Subsection \ref{sec:Cayley} we study the parameter space of the TOA--based localization, i.e. the set describing every configuration of three receivers (up to rigid movements).

\subsection{$\text{Im}\left(\bs{\TT_3}\right)$ is contained in a Kummer's surface}\label{sec:ImsubsetKum}
Although the map $\bs{\TT_3}$ is not defined in terms of polynomials, in this section we are going to prove that its image is actually a real semialgebraic surface.
We remind that the solution of the localization problem is contained in the intersection of $\Lambda(\bs{\TT})=\Pi(\bs\TT)\cap\Pi_i^j(\bs\TT)\cap\Pi_i^k(\bs\TT),$ with $i,j,k$ any permutation of $1,2,3,$ and the cone $C_i(\bs\TT).$ By Proposition \ref{Sol}, $\Lambda(\bs{\TT})$ is the point $\mb{L}(\bs{\mathcal{T}})$. By substituting $\mb{D_i}(\mb{L}(\bs{\TT}))$, given by ~\eqref{eq:point}, into the equation of the cone $C_i(\bs\TT)$ in ~\eqref{eq:Formulation_i}, we get a compatibility equation in $\bs\TT$ that defines an algebraic surface $S$ in the $\TT$--space:
\begin{equation}\label{eq:imm1}
\left\|\left[\frac{1}{2}\left\|\mb{D_{ji}}\left(\bs{\mathcal{T}}\right)\right\|^{2} - \mathcal{T}_{i}\left(\mathcal{T}_{j}-\mathcal{T}_{i}\right)\right] \mb{d_{ki}} - \left[\frac{1}{2}\left\|\mb{D_{ki}}\left(\bs{\mathcal{T}}\right)\right\|^{2} - \mathcal{T}_{i}\left(\mathcal{T}_{k}-\mathcal{T}_{i}\right) \right] \mb{d_{ji}}\right\|^{2} - \mathcal{T}_{i}^{2}\left\|\mb{d_{ji}}\wedge\mb{d_{ki}}\right\|^{2} = 0.
\end{equation}
With further simplifications, we arrive to the equivalent equation
\begin{equation}\label{eq:imm2}
\left\|\mathcal{T}_{1}^{2}\mb{d_{32}} - \mathcal{T}_{2}^{2}\mb{d_{31}} + \mathcal{T}_{3}^{2}\mb{d_{21}}\right\|^{2} - 2\mathcal{T}_{1}^{2}d_{32}^{2}\mb{d_{21}}\cdot\mb{d_{31}} + 2\mathcal{T}_{2}^{2}d_{31}^{2}\mb{d_{21}}\cdot\mb{d_{32}} - 2\mathcal{T}_{3}^{2}d_{21}^{2}\mb{d_{31}}\cdot\mb{d_{32}} + d_{21}^{2}d_{31}^{2}d_{32}^{2} = 0.
\end{equation}
An example of the surfaces $S$ is given in Figure
\ref{fig:kummer}.
\begin{figure}[htb!]
\centering
 \includegraphics[width=9cm,bb=163 268 448 523]{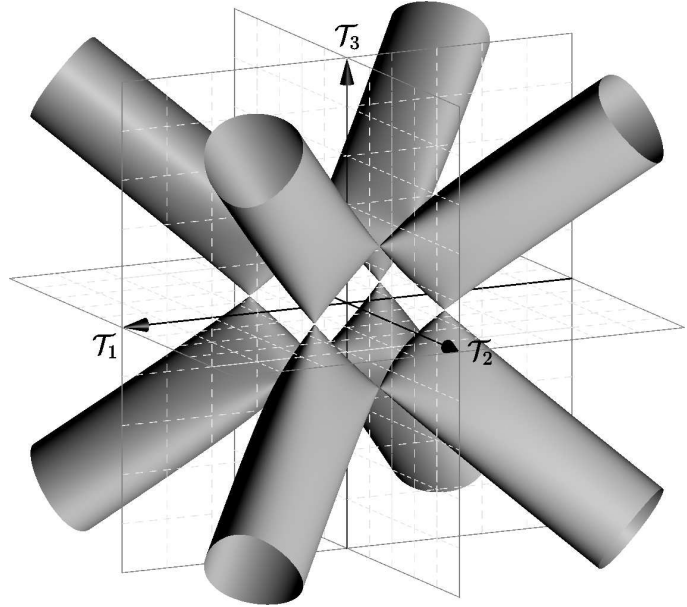}
 \caption{The real part of the Kummer's surface $S$ defined by equation \eqref{eq:imm2}, for the sensors configuration depicted in Figure \ref{fig:rettecerchio}. By symmetry, $S$ can be obtained from that portion of the surface that is contained in the first octant. To go from one octant to the others along the surface, one has to pass through the singular points of $S.$}\label{fig:kummer}
 \end{figure}

\begin{thm}\label{th:1oct}
The image of $\bs{\TT_3}$ is the part of $S$ contained in the first octant $W$ of the $\TT$--space.
\end{thm}
\begin{proof}
The previous argument shows that $Im(\bs{\TT_3}) \subseteq S.$ As observed in Section \ref{sec:MK}, it is obvious that $Im(\bs{\TT_3})$ is contained in $W$, as $d_i(\x)\geq 0.$

Conversely, let $\bs{\bar{\TT}} = \left(\bar\TT_1, \bar\TT_2, \bar\TT_3\right) \in S\cap W$ and let $\mb{L}\left(\bs{\bar\TT}\right)$ be the point given by ~\eqref{eq:invmap}. With abuse of notation we denote with $\mb{L}\left(\bs{\bar\TT}\right)$ also the projection of $\mb{L}\left(\bs{\bar\TT}\right)$ on the x--plane, because the two points differ only for the last coordinate $\mathcal{T} = 0$.\\ By using ~\eqref{eq:imm2}, it is straightforward to check that
\begin{align*}
\bs{\TT_3}\left(\mb{L}\left(\bs{\bar\TT}\right) \right) = \left(\bar{\mathcal{T}}_{1}, \bar{\mathcal{T}}_{2}, \bar{\mathcal{T}}_{3}\right)
\end{align*}
and so $\left(\bar{\mathcal{T}}_{1}, \bar{\mathcal{T}}_{2}, \bar{\mathcal{T}}_{3}\right) \in Im\left(\bs{\TT_3}\right)$.
\end{proof}

If we expand equation \eqref{eq:imm2}, with easy computations we get
 \begin{equation*} \begin{split}
 & d_{32}^2 \mathcal{T}_1^4 + d_{31}^2 \mathcal{T}_2^4 + d_{21}^2 \mathcal{T}_3^4- \\
 & -2 \mb{d_{32}} \cdot \mb{d_{31}} \mathcal{T}_1^2 \mathcal{T}_2^2
+2 \mb{d_{32}} \cdot \mb{d_{21}} \mathcal{T}_1^2 \mathcal{T}_3^2 -2 \mb{d_{31}} \cdot \mb{d_{21}} \mathcal{T}_2^2 \mathcal{T}_3^2 -\\
& -2 \mb{d_{21}} \cdot \mb{d_{31}} d_{32}^2 \mathcal{T}_1^2 + 2 \mb{d_{32}} \cdot \mb{d_{21}} d_{31}^2\mathcal{T}_2^2 - 2 \mb{d_{32}} \cdot \mb{d_{31}} d_{21}^2\mathcal{T}_3^2+\\
& + d_{21}^2 d_{31}^2 d_{32}^2 = 0.
\end{split} \end{equation*}
As the polynomial contains only even degree monomials, the surface $S$ is clearly symmetric with respect to the coordinate planes, the coordinate axes and the origin. Furthermore there is a symmetry with respect to a relabelling of the receivers.

Let us set the rescaled coordinates
\begin{equation}\label{eq:rescaledcoord}
t_1 = \frac{\mathcal{T}_1}{\sqrt{d_{21} d_{31}}}\,,\qquad
t_2 = \frac{\mathcal{T}_2}{\sqrt{d_{21} d_{32}}}\,,\qquad
t_3 = \frac{\mathcal{T}_3}{\sqrt{d_{31} d_{32}}}\,.
\end{equation}
By dividing the above equation times $ d_{21}^2 d_{31}^2 d_{32}^2,$ that is non zero because the sensors are pairwise distinct, and by homogenizing it with respect to the variable $ t_0,$ we get
\begin{equation*}
t_0^4 + t_1^4 + t_2^4 + t_3^4 -2 \mb{\tilde{d}_{32}} \cdot \mb{\tilde{d}_{31}} (t_1^2 t_2^2 + t_0^2 t_3^2)
+2 \mb{\tilde{d}_{32}} \cdot \mb{\tilde{d}_{21}} (t_1^2 t_3^2 + t_0^2t_2^2)
-2 \mb{\tilde{d}_{31}} \cdot \mb{\tilde{d}_{21}} (t_2^2 t_3^2 + t_0^2 t_1^2) = 0.\end{equation*}
We set $ a = \mb{\tilde{d}_{32}} \cdot \mb{\tilde{d}_{31}}, b = \mb{\tilde{d}_{32}} \cdot \mb{\tilde{d}_{21}}, c = \mb{\tilde{d}_{31}} \cdot \mb{\tilde{d}_{21}},$ and so we finally arrive to
\begin{equation} \label{Kum-proj}
t_0^4 + t_1^4 + t_2^4 + t_3^4 -2 a (t_1^2 t_2^2 + t_0^2 t_3^2) +2 b (t_1^2 t_3^2 + t_0^2t_2^2) -2 c (t_2^2 t_3^2 + t_0^2 t_1^2) = 0.
\end{equation}
Equation (\ref{Kum-proj}) defines the projective closure $ \bar{S}\subset\PP^3 $ of the surface $ S $ defined by equation \eqref{eq:imm2}, when we identify the affine space $ \mathbb{A}^3 $ with $ \mathbb{P}^3 $ minus the plane $t_0=0.$

The singular locus $ \Sing(\bar{S}) $ of $ \bar{S} $ is the projective algebraic set defined by $ \nabla F = 0 $ where $$ F(t_0,t_1,t_2,t_3) = t_0^4 + t_1^4 + t_2^4 + t_3^4 -2 a (t_1^2 t_2^2 + t_0^2 t_3^2) +2 b (t_1^2 t_3^2 + t_0^2t_2^2) -2 c (t_2^2 t_3^2 + t_0^2 t_1^2).$$ With straightforward computations, we get that $ \Sing(\bar{S}) $ contains the $ 16 $ points having the following homogeneous coordinates (the signs of each coordinate can be chosen  independently from the others):
\begin{equation}\label{eq:puntisingolari}
\begin{array}{c}
(0 : \pm \sqrt{d_{32}} : \pm \sqrt{d_{31}} : \pm \sqrt{d_{21}})\\
(\pm \sqrt{d_{32}} : 0 : \pm \sqrt{d_{21}} : \pm \sqrt{d_{31}})\\
(\pm \sqrt{d_{31}} : \pm \sqrt{d_{21}} : 0 : \pm \sqrt{d_{32}})\\
(\pm \sqrt{d_{21}} : \pm \sqrt{d_{31}} : \pm \sqrt{d_{32}} : 0).
\end{array}
\end{equation}
The previous computation proves the following
\begin{thm}
$ \bar{S} $ is a Kummer's surface.
\end{thm}
\begin{proof}
A Kummer's surface is a surface in $ \mathbb{P}^3$ defined by a degree--$4$ homogeneous polynomial and having $ 16 $ isolated singular points.
\end{proof}
\begin{rem}\label{rm:singmodels}
The presence of singular points in $\textrm{Im}(\bs{\TT_3})$ has significant consequences for the TOA statistical model. For example, it is known that for a non--regular model the maximum likelihood estimator could not exist or it could not be subject to the asymptotically normal distribution (see \cite{Watanabe1999} and the references therein contained). This fact is particularly relevant in the case of near field source localization, since in the $x$--plane the singularities correspond to the position of the sensors. Although the analysis of the range statistical model is beyond the scopes of this manuscript, we observe that the large sample asymptotics at a singular point depends on the local geometry of $\textrm{Im}(\bs{\TT_3}),$ which can be described using the tangent cone \cite{Drton2007,Drton2009}. This is the semi-algebraic set that approximates the limiting behavior of the secant lines passing through the point of interest.
In our case, the singular points are nodes, i.e. not all the second derivatives are zero at each singular point (see \cite{hudson}). Hence, around each singular point $\bar{S}$ can be approximated by a quadratic cone. For example, for the point $ (\sqrt{d_{21}} : \sqrt{d_{31}} : \sqrt{d_{32}} : 0 )$ the tangent cone has equation
\begin{equation}\label{eq:tgcone}
d_{21}t_0^2+d_{31}t_1^2+d_{32}t_2^2-
\frac{2\Vert\mb{d_{21}}\wedge\mb{d_{31}}\Vert^2}{d_{21}d_{31}d_{32}}\,t_3^2-2c\sqrt{d_{21}d_{31}}\,t_0t_1+
2b\sqrt{d_{21}d_{32}}\,t_0t_2-2a\sqrt{d_{31}d_{32}}\,t_1t_2=0.
\end{equation}
\end{rem}
\begin{rem}\label{rm:tetra}
The general equation of a Kummer's surface is
\begin{equation} \label{eq:generalKum}
t_0^4 + t_1^4 + t_2^4 + t_3^4 -2 a (t_1^2 t_2^2 + t_0^2 t_3^2) +2 b (t_1^2 t_3^2 + t_0^2t_2^2) -2 c (t_2^2 t_3^2 + t_0^2 t_1^2) + 2 d t_0t_1t_2t_3 = 0,
\end{equation}
under the condition that
\begin{equation}\label{eq:paramgK}
8abc-4a^2-4b^2-4c^2+d^2+4=0.
\end{equation}
The surface $\bar{S}$ is obtained by setting $d=0.$ Geometrically, this corresponds to the fact that the nodes are in special position with respect to the planes of $ \mathbb{P}^3.$ In fact, there exist exactly $ 4 $ planes that contain $ 4 $ nodes each, no singular point on two of such planes. As it can be easily checked, such planes are exactly the coordinate planes. In this situation, we say that $ \bar{S} $ is a tetrahedroid (see Chapter 9 of \cite{hudson}.)
\end{rem}

\subsection{Geometric properties of $\text{Im}(\bs{\TT_3})$}\label{sec:geompropIm}

A subtle property of the nodes of a Kummer's surface is the existence of $ 16 $ planes, each containing $ 6 $ nodes, each of which is contained in $ 6 $ such planes. For this reason, the position of the $ 16 $ nodes is referred to as a $ 16_6 $ configuration. It is easy to obtain the equations of the planes, known in the literature as \emph{tropes}, from the coordinates of the nodes. In fact, each trope has coordinates in the dual projective space $ {\check{\mathbb P}}^3 $ that are equal to the coordinates of a node. This way, the tropes have equations
\begin{equation}\label{eq:tropes}
\begin{array}{l}
\pm t_0 \sqrt{d_{21}} \pm t_1 \sqrt{d_{31}} \pm t_2 \sqrt{d_{32}} = 0,\\
\pm t_0 \sqrt{d_{31}} \pm t_1 \sqrt{d_{21}} \pm t_3 \sqrt{d_{32}} = 0,\\
\pm t_0 \sqrt{d_{32}} \pm t_2 \sqrt{d_{21}} \pm t_3 \sqrt{d_{31}} = 0,\\
\pm t_1 \sqrt{d_{32}} \pm t_2 \sqrt{d_{31}} \pm t_3 \sqrt{d_{21}} = 0,
\end{array}
\end{equation}
where the sign of each summand can be chosen independently. Each trope is tangent to $\bar{S}$ along a conic curve, and the only conics contained in $\bar{S}$ are these $16$ ones. For future reference, we also observe that such conics are the smallest degree algebraic curves on $\bar{S},$ since a Kummer's surface contains no line.

Now we focus on the tropes having a role in the description of $\text{Im}(\bs{\TT_3}).$ To this aim, we describe some relevant loci in the $x$--plane (see Figure \ref{fig:rettecerchio}).
\begin{defn}\label{def:rettecerchio}
Let $r_{1},r_{2},r_{3}$ be the lines through two of the three receivers $\m{1},\m{2},\m{3}$, in compliance with the notation $\m{i} \notin r_{i}, \ i = 1,2,3$. Let us split each line in three as $r_{1} = r_{1}^{-} \cup r_{1}^{0} \cup r_{1}^{+}$, where $r_{1}^{0}$ is the segment with endpoints $\m{2}$ and $\m{3}$, $r_{1}^{-}$ is the half-line originating from $\m{3}$ and not containing $\m{2}$, and $r_{1}^{+}$ is the half-line originating from $\m{2}$ and not containing $\m{3}$. Similar splittings are done for lines $r_{2}$ and $r_{3}$, with $r_{2}^{+}$, $r_{3}^{+}$ having $\m{1}$ as endpoint.

Moreover, let $\Gamma$ be the unique circle through the points $\m{1},\m{2},\m{3}.$ We split $\Gamma$ in three arcs as $\Gamma=\Gamma_1\cup\Gamma_2\cup\Gamma_3,$ where $\Gamma_1$ has endpoints $\m{2},\m{3}$ and does not contain $\m{1}.$ Analogous conventions hold for the arcs $\Gamma_2$ and $\Gamma_3$.
\end{defn}
\begin{figure}[htb!]
\centering
 \includegraphics[width=7.5cm,bb=177 280 434 511]{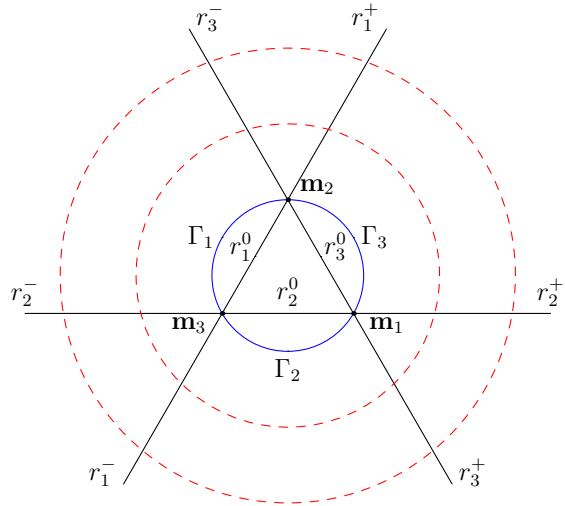}
 \caption{A general configuration of the receivers $\m{i}$, $i = 1,2,3$ and the subsets described in Definition \ref{def:rettecerchio}. E.g. the line $r_1$ does not contain $\m{1}$ and $r_1^0$ (resp. $\Gamma_1$) is the line segment (resp. arc of circle) from $\m{2}$ to $\m{3}.$}\label{fig:rettecerchio}
\end{figure}
\begin{prop}\label{th:retttecerchio}
The image of each locus considered in Definition \ref{def:rettecerchio} is an arc of conic. Moreover, $\text{Im}(\bs{\TT_3})$ does not contain any other arc of conic.
\end{prop}
\begin{proof}
We first prove that the image of each locus is contained in a trope, therefore it is an arc of conic. As there are three different kinds of loci, namely the half--lines $r_i^\pm$, the segments $r_i^0$ and the arcs $\Gamma_i,$ we consider only $r_1^+,\ r_1^0$ and $\Gamma_1,$ as the other cases are similar.

Let $\x\in r_1^+.$ Then $d_3(\x)-d_2(\x)=d_{32}$ and $d_1(\x)\geq d_{21}$ as can be easily seen from Figure \ref{fig:rettecerchio}. This means that $\bs{\TT_3}(\x)$ belongs to the plane $\TT_3-\TT_2=d_{32}.$ By applying the change of coordinates \eqref{eq:rescaledcoord}, we get the homogeneous equation $t_0\sqrt{d_{32}}+t_2\sqrt{d_{21}}-t_3\sqrt{d_{31}}=0,$ thus the plane is a trope. For completeness, the conic containing the image of $r_1^+$ is:
\begin{equation*}
\left\{ \begin{array}{l}
\TT_3 - \TT_2 = d_{32} \\
\TT_1^2 - \TT_3^2 + 2 a d_{31} \TT_3 - d_{31}^2 = 0
\end{array} \right.\ .
\end{equation*}
The arc $\bs{\TT_3}(r_1^+)$ is obtained for $\TT_2\geq 0$ and $\TT_1\geq d_{21}.$

Let $\x\in r_1^0.$ Then $d_3(\x)+d_2(\x)=d_{32}$ and $d_1(\x)\leq\max\{d_{21},d_{31}\}.$ It follows that $\bs{\TT_3}(\x)$ belongs to the plane $\TT_3+\TT_2=d_{32}.$ By applying again the previous change of coordinates, we obtain the homogeneous equation $t_0\sqrt{d_{32}}-t_2\sqrt{d_{21}}-t_3\sqrt{d_{31}}=0,$ thus the plane is a trope. The conic containing the image of $r_1^0$ is:
\begin{equation*}
\left\{ \begin{array}{l}
\TT_3 + \TT_2 = d_{32} \\
\TT_1^2 - \TT_2^2 - 2 b d_{21} \TT_2 - d_{21}^2 = 0
\end{array} \right.\ .
\end{equation*}
The arc $\bs{\TT_3}(r_1^0)$ is obtained for $\TT_1\geq 0$ and $d_{32}\geq\TT_2\geq 0.$

Let $\x\in\Gamma_1$ and let us set $\alpha=\widehat{\m{1}\m{3}\m{2}},$ $\beta=\widehat{\m{3}\m{2}r_3^-},$ $\gamma=\widehat{\m{3}\m{1}\m{2}},$ $\delta=\widehat{\m{3}\m{1}\x},$  $\epsilon=\widehat{\m{2}\m{1}\x}=\widehat{\m{2}\m{3}\x}$ (see Figure \ref{fig:angoli}).
\begin{figure}[htb!]
\centering
 \includegraphics[width=5cm,bb=217 296 394 495]{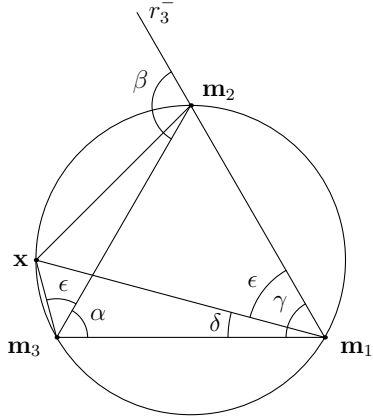}
 \caption{The angles involved in the proof of Proposition \ref{th:retttecerchio}.}\label{fig:angoli}
 \end{figure}
It is apparent that $\gamma=\delta+\epsilon,$ while $\beta=\alpha+\gamma$ because of the properties of triangles. Moreover, it is straightforward to check that
\begin{equation}\label{eq:trigid}
-\sin(\gamma)\sin(\alpha+\epsilon)+\sin(\beta)\sin(\epsilon)+
\sin(\alpha)\sin(\delta)=0
\end{equation}
is a trigonometric identity. Now, let $R$ be the radius of $\Gamma.$ For the law of sines, we have the following equalities:
$$
\frac{d_1(\x)}{\sin(\alpha+\epsilon)}=\frac{d_2(\x)}{\sin(\epsilon)}=\frac{d_3(\x)}{\sin(\delta)}=\frac{d_{32}}{\sin(\gamma)}=\frac{d_{31}}{\sin(\beta)}=\frac{d_{21}}{\sin(\alpha)}=2R.
$$
By substituting into identity \eqref{eq:trigid}, we obtain
$$
-d_{32}d_1(\x)+d_{31}d_2(\x)+d_{21}d_3(\x)=0.
$$
Hence, $\bs{\TT_3}(\x)$ belongs to the plane $-d_{32}\TT_1+d_{31}\TT_2+d_{21}\TT_3=0.$ This corresponds to the homogeneous equation $-t_1\sqrt{d_{32}}+t_2\sqrt{d_{31}}+t_3\sqrt{d_{21}}=0,$ therefore this plane is a trope as well. The conic containing the image of $\Gamma_1$ is:
\begin{equation*}
\left\{ \begin{array}{l}
-d_{32}\TT_1+d_{31}\TT_2+d_{21}\TT_3=0 \\
\TT_1^2 + \TT_2^2 - 2 a \TT_1 \TT_2 - d_{21}^2 = 0
\end{array} \right.\ .
\end{equation*}
The arc $\bs{\TT_3}(\Gamma_1)$ is obtained for $\TT_1\geq 0$ and $d_{32}\,\TT_1\geq d_{31}\,\TT_2\geq 0.$

For the second statement, we remind that the only singular points in $\text{Im}(\bs{\TT_3})$ are the images of $\m{1},\m{2},\m{3},$ whose homogeneous coordinates are
$ (\sqrt{d_{32}} : 0 : \sqrt{d_{21}} : \sqrt{d_{31}} ),
 (\sqrt{d_{31}} : \sqrt{d_{21}} : 0 : \sqrt{d_{32}}),
 (\sqrt{d_{21}} : \sqrt{d_{31}} : \sqrt{d_{32}} : 0 ).$
Among the $16$ tropes, only $4$ contain no one of such points, namely the planes obtained by choosing always the plus sign in equations \eqref{eq:tropes}. This implies that these tropes do not contain any point in the first octant of $\RR^3$ and so they cannot intersect $\text{Im}(\bs{\TT_3})$.
\end{proof}

\begin{figure}[htb!]
\centering
 \includegraphics[width=8cm,bb=168 261 443 530]{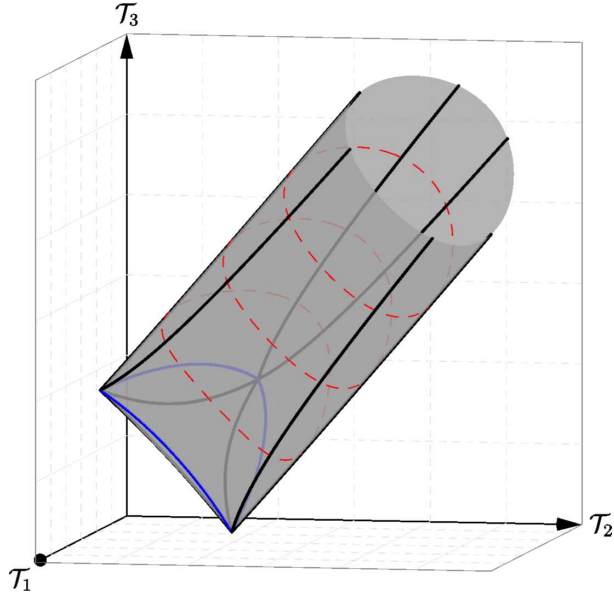}
 \caption{The TOA surface for the configuration of receivers in Figure \ref{fig:rettecerchio}. The blue arcs of ellipses are the images of $\Gamma_i, i=1,2,3,$ the black arcs of hyperbolas are the images of the lines $r_i,\ i=1,2,3$ and they meet at the singular points of the surface, which are the images of the receivers. The red curves, that are not planar, are the images of the red circles in Figure \ref{fig:rettecerchio} concentric to $\Gamma.$}\label{fig:TOAsurf}
 \end{figure}
\begin{rem}\label{rm:tropi}
We can classify the conics of Proposition \ref{th:retttecerchio} according to their preimages in the $x$--plane.

The projective closure of $\text{Im}(\bs{\TT_3})$ contains one ideal singular point having coordinates $(0:\sqrt{d_{32}}:\sqrt{d_{31}}:\sqrt{d_{21}}).$ The six tropes through this point cut on $\text{Im}(\bs{\TT_3})$ six unbounded arcs of hyperbolas, whose preimages are the half--lines $r_i^\pm,\ i=1,2,3.$ In the projective space, the arcs $\bs{\TT_3}(r_i^\pm)$ go from $\bs{\TT_3}(\m{i})$ to the ideal singular point. Their affine equations are:
\begin{equation}\label{eq:ip12}
\bs{\TT_3}(r_3^\mp):\ \left\{ \begin{array}{l}
\TT_1 - \TT_2 = \pm d_{21} \\
\TT_2^2 - \TT_3^2 \pm 2 b d_{32} \TT_2 + d_{32}^2 = 0
\end{array} \right.\ ,
\end{equation}
\begin{equation}\label{eq:ip34}
\bs{\TT_3}(r_2^\mp):\ \left\{ \begin{array}{l}
\TT_1 - \TT_3 = \pm d_{31} \\
-\TT_2^2 + \TT_3^2 \pm 2 a d_{32} \TT_3 + d_{32}^2 = 0
\end{array} \right.\ ,
\end{equation}
\begin{equation}\label{eq:ip56}
\bs{\TT_3}(r_1^\mp):\ \left\{ \begin{array}{l}
\TT_2 - \TT_3 = \pm d_{32} \\
-\TT_1^2 + \TT_3^2 \pm 2 a d_{31} \TT_3 + d_{31}^2 = 0
\end{array} \right.\ .
\end{equation}

The images of the three sides of the triangle in Figure \ref{fig:rettecerchio} are bounded arcs of hyperbolas. In particular, for each $i\neq j\neq k,$ the arc $\bs{\TT_3}(r_i^0)$ has ending points $\bs{\TT_3}(\m{j})$ and $\bs{\TT_3}(\m{k}).$ Their affine equations are:
\begin{equation}\label{eq:ip7}
\bs{\TT_3}(r_3^0):\ \left\{ \begin{array}{l}
\TT_1 + \TT_2 = d_{21} \\
\TT_1^2 - \TT_3^2 - 2 c d_{31} \TT_1 + d_{31}^2 = 0
\end{array} \right.\ ,
\end{equation}
\begin{equation}\label{eq:ip8}
\bs{\TT_3}(r_2^0):\ \left\{ \begin{array}{l}
\TT_1 + \TT_3 = d_{31} \\
\TT_1^2 - \TT_2^2 - 2 c d_{21} \TT_1 + d_{21}^2 = 0
\end{array} \right.\ ,
\end{equation}
\begin{equation}\label{eq:ip9}
\bs{\TT_3}(r_1^0):\ \left\{ \begin{array}{l}
\TT_2 + \TT_3 = d_{32} \\
\TT_1^2 - \TT_2^2 - 2 b d_{21} \TT_2 - d_{21}^2 = 0
\end{array} \right.\ .
\end{equation}

Finally, the images of the three arcs of $\Gamma$ are arcs of ellipses. As above, for each $i\neq j\neq k,$ the arc $\bs{\TT_3}(\Gamma_i)$ has ending points $\bs{\TT_3}(\m{j})$ and $\bs{\TT_3}(\m{k}).$ Their affine equations are:
\begin{equation}\label{eq:el1}
\bs{\TT_3}(\Gamma_3):\ \left\{ \begin{array}{l}
d_{32}\TT_1 + d_{31}\TT_2 - d_{21}\TT_3 = 0 \\
\TT_1^2 + \TT_2^2 + 2 a \TT_1\TT_2 - d_{21}^2 = 0
\end{array} \right.\ ,
\end{equation}
\begin{equation}\label{eq:el2}
\bs{\TT_3}(\Gamma_2):\ \left\{ \begin{array}{l}
d_{32}\TT_1 - d_{31}\TT_2 + d_{21}\TT_3 = 0 \\
\TT_1^2 + \TT_2^2 - 2 a \TT_1\TT_2 - d_{21}^2 = 0
\end{array} \right.\ ,
\end{equation}
\begin{equation}\label{eq:el3}
\bs{\TT_3}(\Gamma_1):\ \left\{ \begin{array}{l}
-d_{32}\TT_1 + d_{31}\TT_2 + d_{21}\TT_3 = 0 \\
\TT_1^2 + \TT_2^2 - 2 a \TT_1\TT_2 - d_{21}^2 = 0
\end{array} \right.\ .
\end{equation}
\end{rem}
\begin{rem}
From a differential point of view, it is well known that the $16$ conics contained in Kummer's surfaces are asymptotic curves and their union is the locus where the Gaussian curvature vanishes. The knowledge of the Gaussian curvature is not only important for the local description of the shape of the surface but it is strongly related to the properties of the TOA statistical model (see \cite{Amari1982,Cheng2013}).

The formula of the Gaussian curvature $K(\x)$ of $\text{Im}(\bs{\TT_3}),$ computed using the parametrization $\bs\TT=\bs{\TT_3(\x)}$ is
$$
K(\x)=\frac{h_{1}(\x)h_{2}(\x)h_{3}(\x)
(d_1(\x)^2h_{1}(\x)+d_2(\x)^2h_{2}(\x)+d_3(\x)^2h_{3}(\x))}
{(d_1(\x)^2h_{1}(\x)^2+d_2(\x)^2h_{2}(\x)^2+d_3(\x)^2h_{3}(\x)^2)^2}\ ,
$$
where
$$
h_{1}(\x)=\ast(\mb{d_2}(\x)\wedge\mb{d_3}(\x)),\qquad
h_{2}(\x)=\ast(\mb{d_3}(\x)\wedge\mb{d_1}(\x)),\qquad
h_{3}(\x)=\ast(\mb{d_1}(\x)\wedge\mb{d_2}(\x)).
$$
With straightforward computations, it is possible to verify that $h_i(\x)=0$ is the defining equation of $r_i,\ i=1,2,3,$ while $d_1(\x)^2h_{1}(\x)+d_2(\x)^2h_{2}(\x)+d_3(\x)^2h_{3}(\x)=0$ is the defining equation of $\Gamma.$ This confirms that the $12$ arcs of conics of Proposition \ref{th:retttecerchio} are the only curves in $\text{Im}(\bs{\TT_3})$ where the Gaussian curvature vanishes. It is easy to check that $K(\x)>0$ at each $\x$ inside the triangle of the sensors, and its sign changes every time $\x$ crosses $r_i,\ i=1,2,3,$ or $\Gamma.$ $K(\x)$ is not defined only at the receivers and it assumes every real value in any neighborhood of $\m{i},\ i=1,2,3.$

When $\x$ is far from the sensors, we can assume $d_1(\x)=d_2(\x)=d_3(\x),$ therefore $K(\x)$ goes to zero as $d_1(\x)^{-2}.$ Hence, far from the singular points $\text{Im}(\bs{\TT_3})$ is well approximated by a cylinder. We further discuss this topic, in particular its connection to Maximum Likelihood Estimation, in Section \ref{sec:conclusion}.
\end{rem}

In Theorem \ref{th:1oct} we proved that $\text{Im}(\bs{\TT_3})$ is contained in the first octant $W$ of the $\TT$--space. Thanks to Proposition \ref{th:retttecerchio}, we can now show that the tropes in Remark \ref{rm:tropi} define a smaller region containing the image, that we represent in Figure \ref{fig:Q3}.
\begin{defn}\label{def:Q3}
In the $\TT$--space, let us take the region
\begin{equation}\label{eq:ineq12}
Q_3:\,\left\{
\begin{array}{l}
-d_{ji}\leq\TT_i-\TT_j\leq d_{ji}\\
\TT_i+\TT_j\geq d_{ji}\\
d_{32}\TT_1 + d_{31}\TT_2 - d_{21}\TT_3\geq0\\
d_{32}\TT_1 - d_{31}\TT_2 + d_{21}\TT_3\geq0\\
-d_{32}\TT_1 + d_{31}\TT_2 + d_{21}\TT_3\geq0
\end{array}
\right.\ ,
\end{equation}
for $1\leq i<j\leq 3.$
\end{defn}
\begin{thm}\label{th:Q3}
$Q_3$ is a convex polyhedron contained in $W$. It has $12$ facets and it is unbounded along the direction given by vector $(1,1,1).$ Furthermore, $\text{Im}(\bs{\TT_3})\subset Q_3.$
\end{thm}
\begin{proof}
$Q_3$ is the intersection of $12$ half--spaces, therefore it is a convex polyhedron. If we project $Q_3$ on the plane $\TT_3=0,$ we get a subset of the polyhedron $Q_2$ of Theorem \ref{prop:caso2mic} and so $\TT_1,\TT_2\geq 0.$ The same happens if we project on the other coordinate planes, therefore we have $Q_3\subset W.$

From Proposition \ref{th:retttecerchio} follows that none of the inequalities \eqref{eq:ineq12} is redundant, hence $Q_3$ has $12$ facets. Indeed, every point on the tangent conics to $\text{Im}(\bs{\TT_3})$ satisfies only one in \eqref{eq:ineq12} as an equality and all the others as strict inequalities. Moreover, it is easy to check that points $\bs{\TT}=(\TT,\TT,\TT)\in Q_3$ for $\TT\geq \frac{1}{2}\min(d_{21},d_{31},d_{32}),$ where the last three inequalities in \eqref{eq:ineq12} follow from the triangular inequalities associated to the triplet of sensors .
For this last statement we observe that a source $ \x $ and two receivers $ \mathbf{m_i}, \mathbf{m_j},$ with $ 1 \leq i < j \leq 3,$ form a triangle, therefore any point $\bs\TT=(\TT_1,\TT_2,\TT_3)\in\text{Im}(\bs{\TT_3})$ satisfies the first $9$ inequalities defining $Q_3:$
\begin{equation} \label{q3-cond}
\vert \TT_i - \TT_j \vert \leq d_{ji} \leq \TT_i + \TT_j, \quad (i,j) = (1,2), (1,3), (2,3).
\end{equation}
Let us now consider the inequality $d_{32}\TT_1 + d_{31}\TT_2 - d_{21}\TT_3\geq0,$ as the remaining two are similar. We have that $\bs{\TT_3}(\x)$ belongs to the trope $d_{32}\TT_1 + d_{31}\TT_2 - d_{21}\TT_3=0$ if, and only if, $\x\in\Gamma_3.$ Since $\Gamma_3$ does not disconnect the $x$--plane and the function $f(\x)=d_{32}\TT_1(\x) + d_{31}\TT_2(\x) - d_{21}\TT_3(\x)$ is continuous, the sign of $f(\x)$ is constant on $\RR^2\setminus\Gamma_3.$ Moreover $f(\m{3})=2d_{31}d_{32}\geq 0$, therefore $d_{32}\TT_1 + d_{31}\TT_2 - d_{21}\TT_3\geq0$ for every points in $\text{Im}(\bs{\TT_3})$. This completes the proof that $\text{Im}(\bs{\TT_3})\subset Q_3.$
\end{proof}
\begin{figure}[htb!]
\centering
 \includegraphics[width=8cm,bb=168 261 443 530]{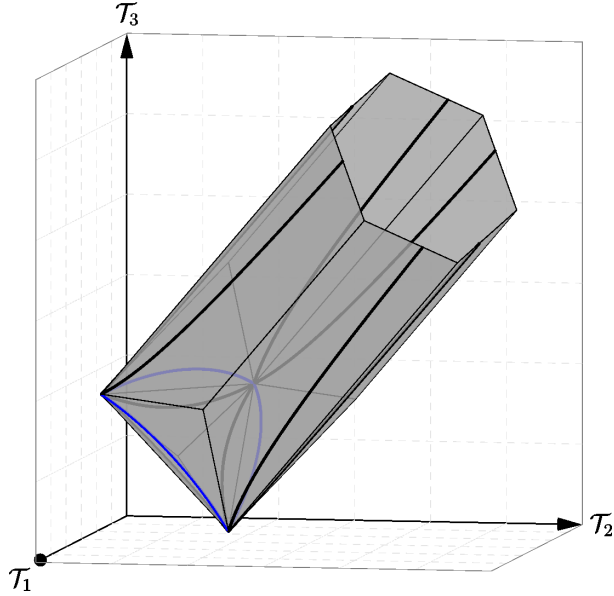}
 \caption{The polyhedron $Q_3$ containing the TOA surface in Figure \ref{fig:TOAsurf}. The bold lines are the asymptotic curves of $\text{Im}(\bs{\TT_3})$. Each facet of the polyhedron $Q_3$ is tangent to $\text{Im}(\bs{\TT_3})$ along one of these curves.}\label{fig:Q3}
 \end{figure}

As observed in the proof of the last Theorem, the first $9$ inequalities defining $Q_3$ are nothing more that the triangular inequalities involving the sensors and the source. Quite surprisingly, we have $3$ other linear inequalities related to $\Gamma_i,\ i=1,2,3$. This added information can turn out to be helpful in various applications.
For example, many algorithms designed for removing the TOA outliers are based on the detection of the TOA measurements not respecting the triangular inequalities \cite{Jian2010}. Therefore, the new inequalities in \eqref{eq:ineq12} could add a boost in the resulting performance.

The polyhedron $Q_3$ is strictly related to the convex hull $\mathcal{E}$ of $\text{Im}(\bs{\TT_3})$, i.e. the smallest convex set in the $\TT$--space that contains $\text{Im}(\bs{\TT_3})$. Recent studies showed the importance of the knowledge of the convex hall of a real semialgebraic variety $X$ for the solution of optimization problems over $X$ (see \cite{Blekherman2012} for an overview on the subject). In the context of localization problems, this could prove crucial for the study of MLE algorithms. We depict the boundary of $\mathcal{E}$ in Figure \ref{fig:CH}.
\begin{figure}[htb!]
\centering
 \includegraphics[width=8cm,bb=168 261 443 530]{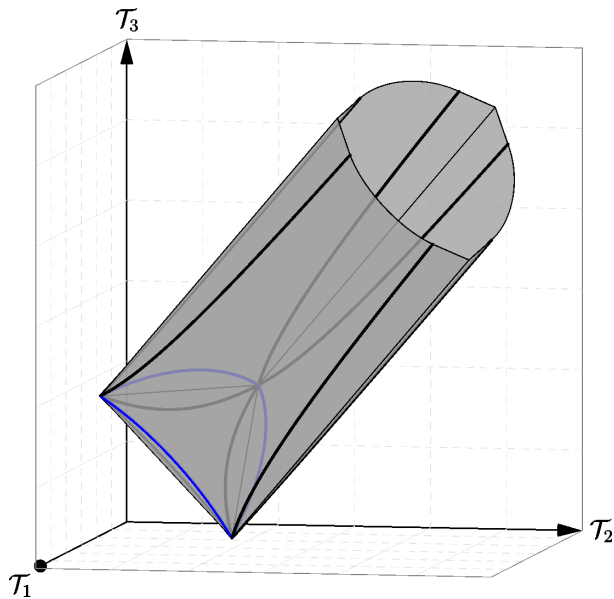}
 \caption{The grey surface is the boundary $\partial\mathcal{E}$ of the convex hull of $\text{Im}(\bs{\TT_3})$. It is the union of $12$ subsets of the facets of $Q_3$ and the $4$ regions of $\text{Im}(\bs{\TT_3})$ having positive Gaussian curvature. They meet smoothly along the asymptotic curves of $\text{Im}(\bs{\TT_3})$, that are depicted in bold.}\label{fig:CH}
 \end{figure}

In order to simplify the description of $\partial\mathcal{E},$ let us introduce the following notation:
\begin{itemize}[leftmargin=*]
\item
$\mathcal{V}_0,\mathcal{V}_1,\mathcal{V}_2,\mathcal{V}_3$ are the $4$ regions in $\text{Im}(\bs{\TT_3})$ where the Gaussian curvature is non negative;
\item
$\mathcal{F}_{ijk}$ is the subset of the facet of $Q_3$ whose boundary is given by the union of $\bs{\TT_3}(\Gamma_i)$ and the edge of $Q_3$ with endpoints $\bs{\TT_3}(\m{j}),\bs{\TT_3}(\m{k}),$ where $(i,j,k)=(1,2,3),(2,1,3),(3,1,2);$
\item
$\mathcal{G}_{ijk}$ is the subset of the facet of $Q_3$ whose boundary is given by the union of $\bs{\TT_3}(r_i^0)$ and the edge of $Q_3$ with endpoints $\bs{\TT_3}(\m{j}),\bs{\TT_3}(\m{k}),$ where $(i,j,k)=(1,2,3),(2,1,3),(3,1,2);$
\item
$\mathcal{L}_i^+$ (resp. $\mathcal{L}_i^-$) is the unbounded subset of the facet of $Q_3$ having boundary given by the union of $\bs{\TT_3}(r_i^+)$ (resp. $\bs{\TT_3}(r_i^-)$) and the unbounded edge of $Q_3$ with endpoint $\bs{\TT_3}(\partial r_i^+)$ (resp. $\bs{\TT_3}(\partial r_i^-)$), where $i=1,2,3.$
\end{itemize}
\begin{thm}\label{th:CH}
The boundary of $\mathcal{E}$ is
$$
\partial\mathcal{E}=\bigcup_{i=0}^{3}\mathcal{V}_i
\cup\mathcal{F}_{1,2,3}\cup\dots\cup\mathcal{G}_{3,1,2}\cup
\left(\bigcup_{i=1}^{3}(\mathcal{L}_i^+\cup\mathcal{L}_i^-)\right).
$$
The convex hull $\mathcal{E}$ is the domain contained into $\partial\mathcal{E}$ minus the three unbounded edges of $Q_3\cap\partial\mathcal{E}.$
\end{thm}
\begin{proof}
By definition, $\mathcal{E}$ is a convex set if for any two points $p,q\in\mathcal{E}$ the segment $s_{pq}$ joining them is contained in $\mathcal{E}$. Clearly, it is sufficient to check this property for $p,q\in\partial\mathcal{E}.$ But this is a direct consequence of the convexity of $Q_3$ and the positive Gaussian curvature of $\mathcal{V}_i,\ i=0,\dots,3.$

We now need to show that, among the convex sets of the $\TT$--space, $\mathcal{E}$ is the smallest one containing $\text{Im}(\bs{\TT_3}).$ First of all, it is obvious that $\text{Im}(\bs{\TT_3})\subset\mathcal{E}.$ Then, using the previous notation, we have to prove that for any $\tilde{p}\in\mathcal{E}$ there exist $p,q\in\text{Im}(\bs{\TT_3})$ such that $\tilde{p}\in s_{pq}.$ As above, we focus on $\tilde{p}\in\partial\mathcal{E}.$ The statement is trivially true for every $\tilde{p}\in\mathcal{V}_i,\ i=0,\dots,3.$ Moreover, we have $\mathcal{F}_{123}=\bigcup_{q}s_{pq}$ where $p=\bs{\TT_3}(\m{2})$ and $q\in\bs{\TT_3}(\Gamma_1),$ and similar relations hold true for the other components of $\partial\mathcal{E}\cap\partial Q_3.$ Finally, the unbounded edges of $Q_3\cap\partial\mathcal{E}$ do not belong to $\mathcal{E},$ indeed they connect the singular points of $\text{Im}(\bs{\TT_3})$ to the ideal singular point $(0:1:1:1)$ of $\bar{S},$ that is outside the image of $\bs{\TT_3}$.
\end{proof}
\noindent Theorem \ref{th:CH} accurately describes the set $\mathcal{E}.$ However, we notice that the convex polyhedron $Q_3$ represents an approximation of the convex hull that is a great deal simpler, as it is defined by linear inequalities only.

\subsection{The parameter space of the range surfaces}\label{sec:Cayley}
It is well known that the accuracy of range--based localization depends on the configuration of the sensors. In our geometrical perspective, by moving the sensors on the $x$--plane, we change the shape of $\text{Im}(\bs{\TT_3}).$
Let us recall equation \eqref{Kum-proj} defining the Kummer's surface $\bar{S}:$
\begin{equation*}
t_0^4 + t_1^4 + t_2^4 + t_3^4 -2 a (t_1^2 t_2^2 + t_0^2 t_3^2) +2 b (t_1^2 t_3^2 + t_0^2t_2^2) -2 c (t_2^2 t_3^2 + t_0^2 t_1^2) = 0.
\end{equation*}
$\bar{S}$ depends on the parameters $ a = \mb{\tilde{d}_{32}} \cdot \mb{\tilde{d}_{31}}, b = \mb{\tilde{d}_{32}} \cdot \mb{\tilde{d}_{21}}, c = \mb{\tilde{d}_{31}} \cdot \mb{\tilde{d}_{21}}.$ These are scalar products of unit vectors, therefore they are functions of the angles $ \alpha, \beta, \gamma $ between the receivers (see Figure \ref{fig:parametri-Kummer}).
\begin{figure}[htb]
\centering
  \includegraphics[width=4cm,bb=234 345 377 446]{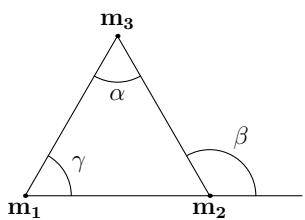}\\
  \caption{The angles $ \alpha $ and $ \gamma $ are inside the triangle of the receivers, 
                while $ \beta $ is exterior to the same triangle.}\label{fig:parametri-Kummer}
\end{figure}

Now we study more in depth the relation between such angles and $\text{Im}(\bs{\TT_3}).$ From the elementary properties of triangles follows that $ \beta = \alpha + \gamma,$ therefore $ a,b,c $ are not independent. In fact, by using some trigonometric identities, we obtain
\begin{equation*} \begin{split} b = & \cos(\beta) = \cos(\alpha) \cos(\gamma) - \sqrt{(1-\cos^2(\alpha))(1-\cos^2(\gamma))} \\ & = ac - \sqrt{(1-a^2)(1-c^2)}. \end{split} \end{equation*}
In the previous computation we used the fact that $ 0 \leq \alpha, \beta, \gamma \leq \pi.$ We therefore have
\begin{equation} \label{Cayley-sur}
2abc -a^2-b^2-c^2 + 1 = 0.
\end{equation}
Notice that, after setting the parameter $d=0,$ this equation coincides with \eqref{eq:paramgK}, and for $a,b,c\in\CC$ it defines the parameter space of a general tetrahedroid. On the other hand, we are interested in describing the parameter space of the range surfaces, which constrains $a,b,c\in\RR$, more precisely $-1\leq a,b,c\leq 1.$

The zero locus $S'$ of \eqref{Cayley-sur} is the Cayley's surface, i.e. the unique cubic surface with four isolated nodes.\footnote{For a picture of the surface, see for example \href{en.wikipedia.org/wiki/Cayley's_nodal_cubic_surface}{Wikipedia.}}
For equation \eqref{Cayley-sur}, the nodes have coordinates $(1,1,1),(1,-1,-1),(-1,1,-1),(-1,-1,1).$ These points are the vertices of a tetrahedron $T$ whose edges are contained in $S'$, and so the Cayley's surface $S'$ contains $6$ lines. Moreover, as it can be seen from its picture, the real part of $S'$ contains a topological deformation of the faces of $T$.
\begin{thm} \label{par-TOA-surf}
The parameter space of the range surfaces is $$ \mathcal{P} = \{ (a,c) \in (-1,1) \times (-1,1) \ \vert \ a+c > 0 \} \times \RR^+,$$ up to rototranslations and reflections of the sensors triangle.
\end{thm}

\begin{proof} Let us choose three non collinear points as receivers. We can resize the triangle, up to fix a positive scale factor. Hence, we can work on the angles of the triangle.
From the previous computations, we know that $$ b = ac - \sqrt{(1-a^2)(1-c^2)}.$$ As $ \alpha $ and $ \gamma $ are interior angles of the same triangle, their sum cannot be larger than $ \pi,$ thus $ \alpha \leq \pi - \gamma.$ The cosine function is decreasing in $ [0, \pi],$ and so $ \cos(\alpha) \geq \cos(\pi - \gamma) = -\cos(\gamma)$, that is to say, $ a \geq -c $, or $ a+c \geq 0.$ If $ a = \pm 1,$ then $ \alpha = 0$ or $ \alpha = \pi,$ hence the three receivers lie on the same line. The same happens if $c= \pm 1$. Finally, if $a+c=0$, then $ \beta = \alpha + \gamma = \pi $ and so again the receivers lie on the same line.

Conversely, if $ (a,c) $ is in the first factor of $\mathcal{P},$ then there exist $ \alpha, \gamma \in (0,\pi) $ such that $ \cos(\alpha) = a, \cos(\gamma) = c$, and $ \beta = \alpha + \gamma < \pi.$ Therefore, we can construct the triangle, up to choosing the length of one of its sides, and the proof is complete.
\end{proof}

The precise description of the TOA parameter space is instrumental for the study of applied problems such as the searching of the optimal sensors configuration for localizing a source under given constraints. This is not completely novel in the literature. For example, an analysis of the optimal sensor placement in TOA--based localization is proposed in \cite{Moran2012} for the case of two sensors, using Information Geometry. In Figure \ref{fig:spazio-param} we give a picture of the first factor $\tilde{\mathcal{P}} = \{ (a,c) \in (-1,1) \times (-1,1) \ \vert \ a+c > 0 \}$ of the parameter space $\mathcal{P}.$
\begin{figure}[htb!]
\centering
 \includegraphics[width=5.5cm,bb=208 304 403 487]{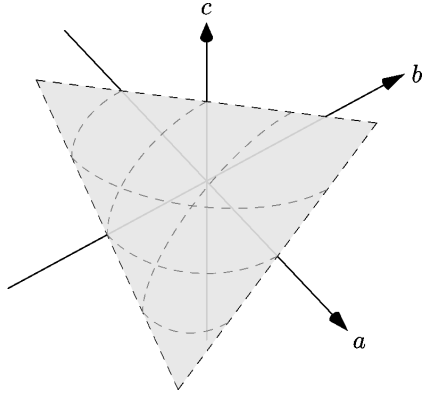}
 \caption{The surface $\tilde{\mathcal{P}}.$ The dashed black lines are the segments joining the singular points of the Cayley's surface containig $\tilde{\mathcal{P}}$. To every point of $\tilde{\mathcal{P}}$ uniquely corresponds a configuration of the angles $\alpha,\beta,\gamma$ in Figure \ref{fig:parametri-Kummer}. To get the triangle of the receivers, we have to fix the length of one of its side.}\label{fig:spazio-param}
 \end{figure}


\section{The image of $\bs{\TT_{3}}$ when the receivers are collinear}
\label{sec:collinear-microphones}

In this section, we consider the range--based localization when
$\m{1}$, $\m{2}$ and $\m{3}$ are collinear. As done with Definition \ref{def:rettecerchio}, we set
\begin{defn}\label{def:allineati}
The line containing $\m{1},\m{2},\m{3}$ is $r.$ The segment with end point $\m{j},\m{k}$ is $r_i^0,$ where $1\leq i,j,k\leq 3$ are all different from each other.
\end{defn}
From Section \ref{sec:MK} we recall that, in order to locate the source, we have
to intersect the three planes $\Lambda(\bs\TT) =
\Pi\left(\bs{\mathcal{T}}\right) \cap
\Pi^{1}_{2}\left(\bs{\mathcal{T}}\right) \cap
\Pi^{1}_{3}\left(\bs{\mathcal{T}}\right)$ and the cone
$C_{1}(\bs{\mathcal{T}}).$ As assumed in
Proposition \ref{Sol}, we can pick $\rho \in \mathbb{R}$ in such a way that
$\mb{d_{31}} = \rho \mb{d_{21}}.$ The intersection
$\Lambda(\bs\TT)$ is not empty if, and only if, equation
\eqref{eq:comp} is satisfied
$$
\mathcal{T}_{3}^{2} = (1 - \rho) \mathcal{T}_{1}^{2} + \rho \mathcal{T}_{2}^{2} - \rho (1-\rho) d_{21}^{2}\,.
$$
In this case, $\Lambda(\bs\TT)$ is a line of parametric equation
\begin{equation}
\mb{X}\left(\lambda; \bs{\TT}\right) = \mb{L_0}\left(\bs{\TT}\right) + \lambda \mb{v}\left(\bs{\TT}\right),
\end{equation}
where we followed the notations of Lemma \ref{4am}. By
replacing $\mb{X}\left(\lambda; \bs{\TT}\right)$ into the
equation of the cone
$\left\|\mb{D_1}\left(\mb{X},\bs{\TT}\right)\right\|^{2} = 0,$
we obtain the following quadratic equation in $\lambda:$
\begin{equation}\label{eq:norm}
\left\|\mb{D_1}\left(\mb{L_0}\left(\bs{\TT}\right) + \lambda \mb{v}\left(\bs{\TT}\right) \right)\right\|^{2} = 0.
\end{equation}
Since $\left\|\mb{v}\left(\bs{\TT}\right)\right\|^{2} =
d_{21}^{2}$ and $\left \langle
\mb{D_1}\left(\mb{L_0}\left(\bs{\TT}\right)\right),
\mb{v}\left(\bs{\TT}\right)\right \rangle = 0$, equation
\eqref{eq:norm} can be rewritten as
\begin{equation}\label{eq:norm2}
d_{21}^{2} \lambda^{2} + \left\|\mb{D_1}\left(\mb{L_0}\left(\bs{\TT}\right)\right)\right\|^{2} = 0.
\end{equation}
Given two real solutions $\lambda_\pm$ of \eqref{eq:norm2}, we
have two admissible source locations given by
$\mb{X}\left(\lambda_\pm; \bs{\TT}\right).$ The condition for
$\lambda_\pm$ to be real is
$\left\|\mb{D_1}\left(\mb{L_0}\left(\bs{\TT}\right)\right)\right\|^{2}
\leq 0$. After straightforward computations, this inequality
results equivalent to
\begin{equation}\label{eq:Sigma}
\left\|\mb{D_{21}}\left(\bs{\TT}\right)\right\|^{2} \left({d}_{21}^{2} - \left(\TT_{2} + \TT_{1}\right)^{2}\right) \leq 0,
\end{equation}
which in turn corresponds exactly to the triangular inequalities
involving $\m{1}$ and $\m{2}.$ Obviously, equation
\eqref{eq:norm2} has two coinciding solutions if, and only if,
$\left\|\mb{D_1}\left(\mb{L_0}\left(\bs{\TT}\right)\right)\right\|^{2}=0.$
This means that the localization is unique if, and only if,
$\mb{x}$ lies on the line $r.$ On the other hand, there are two distinct solutions
$\lambda_\pm$ if, and only if, inequality \eqref{eq:Sigma} is
strictly satisfied. Equivalently, there are two solutions
(symmetric with respect to $r$) for the localization problem if,
and only if, $\m{1}$, $\m{2}$ and $\mb{x}$ are not
collinear.

Now we are interested in describing the image of
$\bs{\TT_3}.$ From the previous discussion, the feasible set of the
TOAs is described by equation \eqref{eq:comp} and inequality
\eqref{eq:Sigma}. We begin our analysis by studying the geometrical
properties of the locus $\sigma$ in the $\TT$--space defined by
\eqref{eq:comp}.

\begin{prop} \label{lemma:sigma}
$ \sigma $ is a one--sheet hyperboloid for every $ \rho \not=
0,1,$ thus $ \sigma $ is ruled. Its principal axes are the coordinate ones.
Moreover, it is a hyperboloid of revolution if, and only if, $ \rho = -1,1/2,$ or
$2.$
\end{prop}
\begin{proof}
The defining equation of $ \sigma $ is $$ (1-\rho) \TT_{1}^2 +
\rho \TT_{2}^2 - \TT_{3}^2 - \rho(1-\rho) d_{21}^2 = 0,$$ that is
the canonical form of a quadric surface. The $ 4 \times 4 $
associated symmetric matrix has determinant $ d_{21}^2\rho^2
(1-\rho)^2 > 0 $ for every $ \rho \not= 0,1.$ Thus, $ \sigma $ is
a ruled smooth quadric. The eigenvalues of the $ 3 \times 3 $
symmetric matrix associated to the quadratic part of the equation
are $ 1-\rho, \rho, -1.$ It follows that they all are non zero for
$ \rho \not= 0,1$ and they are not all positive or negative,
therefore $ \sigma $ is a one--sheet hyperboloid. In particular,
it is circular if, and only if, an eigenvalue has multiplicity at
least $ 2,$ i.e. only when $ \rho = -1,1/2$ or $ 2.$
\end{proof}

For the sake of shortening the mathematical details, from now on we assume that $\m{3}\in r_3^0$ and so $\rho\in(0,1).$ This causes no loss of generality, as we can go back to this situation by simply relabeling the sensors.

As said above, Im$(\bs{\TT_3})$ is constrained in the region of the $\TT$--space defined by inequality \eqref{eq:Sigma}. However, we can provide an even smaller set containing the image by considering the polyhedron $Q_3$ given in Definition \ref{def:Q3}. In the current scenario, $Q_3$ is a simpler region with respect to the one we described in Theorem \ref{th:Q3}.
\begin{thm}
In the described configuration of the sensors, the polyhedron $Q_3$ is defined by the following non redundant inequalities:
\begin{equation}\label{eq:Q3s}
\left\{\begin{array}{l}
\TT_{1} - \TT_{3} \leq d_{31}\\
\TT_{2} - \TT_{3} \leq d_{32}\\
\TT_1+\TT_2\geq d_{21}\\
d_{32}\TT_1 + d_{31}\TT_2 - d_{21}\TT_3\geq0
\end{array}\right.\,.
\end{equation}
$Q_3$ is contained in the first octant $W$ of the $\TT$--space and it has $3$ unbounded facets plus a fourth bounded one, $3$ edges that are half lines parallel to the vector $(1,1,1)$ and $3$ edges that are the line segments connecting the $3$ vertices $(d_{31},d_{32},0),$ $(d_{21},0,d_{32}),$ $(0,d_{21},d_{31}).$
\end{thm}
\begin{proof}
Firstly, we prove that in the aligned scenario the system \eqref{eq:ineq12} is equivalent to \eqref{eq:Q3s}. Since $d_{31}=\rho d_{21}$ and $d_{32}=(1-\rho)d_{21},$ the last inequality in \eqref{eq:Q3s} can be rewritten as
\begin{equation}\label{eq:ineqarcos}
\TT_3\leq(1-\rho)\TT_1+\rho\TT_2.
\end{equation}
By combining \eqref{eq:ineqarcos} with the first inequality, we get
$$
d_{31}\geq \TT_1-\TT_3\geq \TT_1-(1-\rho)\TT_1-\rho\TT_2=\rho(\TT_1-\TT_2),
$$
thus $\TT_1-\TT_2\leq d_{21}.$ With similar computations, by using the second inequality we get $\TT_1-\TT_2\geq -d_{21}.$ On the other hand, the second inequality in \eqref{eq:Q3s} can also be rewritten as
$$
\TT_2\leq\TT_3+(1-\rho)d_{21}\quad \Longrightarrow\quad \TT_3\leq(1-\rho)\TT_1+\rho\TT_3+\rho(1-\rho)d_{21}
$$
and, after simplifications, we obtain $\TT_1-\TT_3\geq -d_{31}.$ Similarly, starting from the first inequality we arrive to $\TT_2-\TT_3\geq -d_{32}.$ This means that \eqref{eq:Q3s} implies all the $6$ inequalities in the first row of \eqref{eq:ineq12}.

The inequality $\TT_2+\TT_3\geq d_{32}$ in the second row of \eqref{eq:ineq12} can be obtained by subtracting in \eqref{eq:Q3s} the first inequality from the third one. Similarly, $\TT_1+\TT_3\geq d_{31}$ is equivalent to the subtraction of the second inequality from the third one and this completes the second row of \eqref{eq:ineq12}. Let us observe that the $9$ inequalities we just proved imply $\TT_1,\TT_2,\TT_3\geq 0.$

What remains to be verified are the last two inequalities in \eqref{eq:ineq12}. By inequalities \eqref{eq:ineqarcos} and $0<\rho<1,$ we have
$$
d_{32}\TT_1 - d_{31}\TT_2 + d_{21}\TT_3 =
d_{21}((1-\rho)\TT_1 - \rho\TT_2 + \TT_3)\geq
2d_{21}(\TT_3-\rho\TT_2)\geq 2d_{21}(\TT_3-\TT_2)\geq 2d_{21}d_{32}>0,
$$
where at the end we used the second row of \eqref{eq:Q3s}. This proves that system \eqref{eq:Q3s} implies the second last inequality in \eqref{eq:ineq12}, and the last one follows similarly.

To show that inequalities in \eqref{eq:Q3s} are non redundant, we observe that the system
\begin{equation*}
\left\{\begin{array}{l}
\TT_{1} - \TT_{3} \leq d_{31}\\
\TT_{2} - \TT_{3} \leq d_{32}\\
d_{32}\TT_1 + d_{31}\TT_2 - d_{21}\TT_3\geq0
\end{array}\right.
\end{equation*}
defines a polyhedron with $3$ facets and $3$ edges parallel to the vector $(1,1,1).$ The remaining inequality $\TT_1+\TT_2\geq d_{21}$ cuts such polyhedron and pick out the part completely contained in the first octant of $\RR^3,$ with vertices $(d_{31},d_{32},0),(d_{21},0,d_{32}),(0,d_{21},d_{31}).$
\end{proof}

In the following, we are going to prove that Im$(\bs{\TT_3})$ coincides with $\sigma\cap Q_3.$ Preliminarily, in the next Proposition we give a geometric description of such intersection.
\begin{prop}\label{headache}
Let us set the map
\begin{equation}\label{eq:T3map}
\begin{array}{cccc}
\varphi: & Q_{2} & \longrightarrow & [0,+\infty)\\
& (\TT_{1},\TT_{2}) & \longmapsto & \sqrt{(1-\rho)\TT_1^2+\rho\TT_2^2-\rho(1-\rho)d_{21}^2}
\end{array}.
\end{equation}
Then, $\sigma \cap Q_{3}$ is the graph $G_\varphi$ of the function $\varphi.$
The boundary $\partial(\sigma \cap Q_{3})=\sigma\cap\partial Q_3$ is the connected union of $2$ bounded and $2$ unbounded edges of $Q_3,$ given by $\varphi(\partial Q_2).$
\end{prop}

\begin{proof}
First of all, we have to prove that $\varphi$ is well--defined. This is true if, and only if,
\begin{equation}\label{eq:ineqboh}
(1 - \rho) \mathcal{T}_{1}^{2} + \rho \mathcal{T}_{2}^{2} - \rho (1-\rho) d_{21}^{2}\geq 0
\end{equation}
for every $\left(\TT_{1},\TT_{2}\right) \in Q_{2}$.
In the $\TT$--plane, the equality
$(1 - \rho) \mathcal{T}_{1}^{2} + \rho \mathcal{T}_{2}^{2} - \rho (1-\rho) d_{21}^{2}=0$
defines a conic $C$ with center at the origin and tangent to each line supporting the facets of $Q_2.$ Therefore, $C$ has to be either contained in $Q_2$ or in its complement $\RR^2\setminus Q_2.$ On the other hand, as $(0,0)\notin Q_2$ it follows that $C\not\subset Q_2$ and so the left side of \eqref{eq:ineqboh} has constant sign on $Q_2.$ Through a direct check, we have that the point $(d_{21},d_{21})\in Q_2$ satisfies inequality \eqref{eq:ineqboh}, hence the same is true for any point in $Q_2.$

It is trivial to verify that the point $(\TT_1,\TT_2,\varphi(\TT_1,\TT_2))\in\sigma$ for every $(\TT_1,\TT_2)\in Q_2.$ Indeed, the map $\varphi$ can be extended to the real plane minus the region internal to $C$ and the part of the hyperboloid $\sigma$ in the half--space $\TT_3\geq 0$ is exactly the graph of $\varphi.$ This also means that $\sigma\cap Q_3\subseteq G_\varphi,$ because $Q_{2}$ contains the projection of $Q_{3}$ onto the $\left(\TT_{1},\TT_{2}\right)$--plane.

What it still to be proven is that $G_\varphi\subset Q_3.$ Since $\TT_1+\TT_2\geq d_{21}$ is a defining inequality of $Q_2,$ it is satisfied for any point in $G_\varphi.$ We have to check
\begin{equation*}
\TT_{1} - \TT_{3} \leq d_{31},\qquad
\TT_{2} - \TT_{3} \leq d_{32},\qquad
d_{32}\TT_1 + d_{31}\TT_2 - d_{21}\TT_3\geq0,
\end{equation*}
for $(\TT_1,\TT_2)\in Q_2$ and $\TT_{3}=\varphi(\TT_1,\TT_2).$ The first inequality can be rewritten as
$$
\TT_1-\rho d_{21}\leq\sqrt{(1-\rho)\TT_1^2+\rho\TT_2^2-\rho(1-\rho)d_{21}^2}.
$$
With straightforward computation, we arrive to the system of inequalities $\TT_1+\TT_2\geq d_{21}$ and $\TT_1-\TT_2\leq d_{21},$ that are verified for every point in $Q_2.$ In similar way, one can check also the second inequality.
Instead, the last one can be rewritten as
$$
((1-\rho)\TT_1+\rho\TT_2)^2\geq (1-\rho)\TT_1^2+\rho\TT_2^2-\rho(1-\rho)d_{21}^2.
$$
With easy computations, we get
$\TT_1^2+\TT_2^2-2\TT_1\TT_2\leq d_{21}^2,$
i.e. $|\TT_1-\TT_2|\leq d_{21},$ that again is verified for every $(\TT_1,\TT_2)\in Q_2.$

Finally, since $\varphi$ is continuous on $\mathring Q_2,$ we have that $\partial G_\varphi=\varphi(\partial Q_2).$ In the $\TT$--space, the boundary $\partial Q_{2}$ is given by two half--lines $l_\pm$ and one line segment $s:$
\begin{equation*}
l_+:\ \TT_{1} = \TT_{2} + d_{21},\ \TT_2\geq 0,\qquad\
l_-:\ \TT_{1} = \TT_{2} - d_{21},\ \TT_2\geq d_{21},\qquad\
s:\ \TT_{1} = d_{21} - \TT_{2},\ 0\leq\TT_2\leq d_{21}.
\end{equation*}
The evaluations of $\varphi$ on $l_+,l_-$ and $s$ are, respectively:
\begin{equation*}
\varphi(\TT_{2}+d_{21},\TT_{2}) = \TT_{2} + (1-\rho) d_{21},\quad
\varphi(\TT_{2}-d_{21},\TT_{2}) = \TT_{2} - (1-\rho) d_{21},\quad
\varphi(d_{21}-\TT_{2},\TT_{2}) = \vert \TT_{2} - (1-\rho) d_{21}\vert.
\end{equation*}
It follows that
\begin{equation*}
\varphi(l_+):\ \left\{
\begin{array}{l}
\TT_{1} = t + d_{21}\\
\TT_{2} = t\\
\TT_{3} = t + d_{32}
\end{array}\right.,\quad t \geq 0,\qquad\qquad
\varphi(l_-):\ \left\{
\begin{array}{l}
\TT_{1} = t - d_{21}\\
\TT_{2} = t\\
\TT_{3} = t - d_{32}
\end{array}\right.,\quad t \geq d_{21},
\end{equation*}
while $\varphi(s)$ is the union of two line segments $\varphi(s)_+\cup\varphi(s)_-$
\begin{equation*}
\varphi(s)_+:\ \left\{
\begin{array}{l}
\TT_{1} = d_{21} - t\\
\TT_{2} = t\\
\TT_{3} = d_{32} - t
\end{array}\right.,\quad 0 \leq t \leq d_{32},\qquad\qquad
\varphi(s)_-:\ \left\{
\begin{array}{l}
\TT_{1} = d_{21} - t\\
\TT_{2} = t\\
\TT_{3} = t - d_{32}
\end{array}\right.,\quad d_{32} \leq t \leq d_{21}.
\end{equation*}
The parallel half--lines $\varphi(l_+)$ and $\varphi(l_-)$ are the two unbounded edges of $Q_3$ with end points $(d_{21},0,d_{32})$ and $(0,d_{21},d_{31}),$ respectively. They are connected via the two bounded edges $\varphi(s)_+$ and $\varphi(s)_-,$ that meet at $ (d_{31}, d_{32}, 0).$
\end{proof}
Now, we are ready to state the main result of this Section.
\begin{thm}\label{headache2}
$Im\left(\bs{\TT_3}\right) = \sigma \cap Q_{3}$. In particular, given $\bs{\TT} \in Im\left(\bs{\TT_3}\right)$ then
\begin{align*}
\left| \bs{\TT_3}^{-1}\left(\bs{\TT}\right)\right| = \begin{cases}
2, \quad \text{if}\quad \bs{\TT} \in \mathring{Q}_{3} \cr
1, \quad \text{if}\quad \bs{\TT} \in \partial Q_{3} \end{cases} .
\end{align*}
\end{thm}
\begin{proof}
Let $\bs\TT=(\TT_1,\TT_2,\TT_3)$ be a point in the $\TT$--space. By the previous Propositions, if $\bs\TT\in\sigma\cap Q_3$ then equation \eqref{eq:comp} and inequality \eqref{eq:Sigma} holds, therefore $\bs\TT\in\text{Im}(\bs{\TT_3}).$ Conversely, if $\bs\TT\in\text{Im}(\bs{\TT_3}),$ then $(\TT_1,\TT_2)\in\text{Im}(\bs{\TT_2})=Q_2$ and $\TT_3$ is a positive number satisfying equation \eqref{eq:comp}. It follows that $\TT_3=\varphi(\TT_1,\TT_2)$ and, by Proposition \ref{headache}, $\bs\TT\in\sigma\cap Q_3.$

Furthermore, from the discussion at the beginning of this Section, we know that the range map is $1$--to--$1$ if, and only if, $\x\in r.$ Since $\bs{\TT_2}(r)=\partial Q_2,$ Proposition \ref{headache} implies that $\bs{\TT_3}(r)=(\bs{\TT_2}(r),\varphi(\bs{\TT_2}(r)))=\text{Im}(\bs{\TT_3})\cap\partial Q_3,$ and the second claim follows.
\end{proof}

\begin{figure}[htb!]
\centering
 \includegraphics[width=6cm,bb=208 297 403 494]{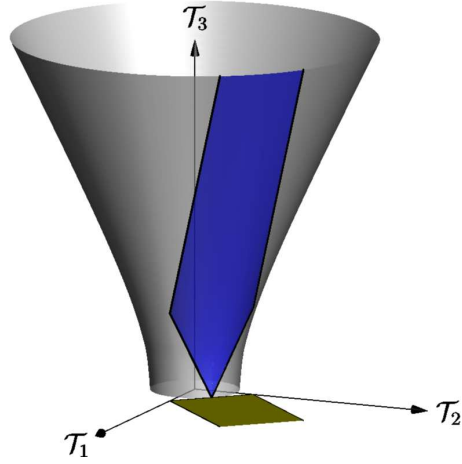}
 \caption{The range surface for the aligned configuration of sensors, where $\rho=0.5.$ The half of the hyperboloid $\sigma$ in $\TT_3\geq 0$ is the graph of the (extension of the) function $\varphi$ defined in Proposition \ref{headache}. The image of $\bs{\TT_3}$ is the blue region, which corresponds to $\sigma\cap Q_3$. Its projection on the coordinate plane $(\TT_1,\TT_2)$ is the unbounded polyhedron $Q_2.$}\label{fig:TOAdegsurf}
 \end{figure}

\begin{rem} \label{rem:kummer-to-sigma}
If we substitute $\mb{d_{31}} = \rho \mb{d_{21}}$ and $\mb{d_{32}} = (1-\rho) \mb{d_{21}}$ into the defining equation \eqref{eq:imm2} of the Kummer's surface $S,$ we get
\begin{equation*}
\begin{split}
 & (1-\rho)^2d_{21}^2 \mathcal{T}_1^4 + \rho^2 d_{21}^2 \mathcal{T}_2^4 + d_{21}^2 \mathcal{T}_3^4 -\\
 & -2 \rho(1-\rho) d_{21}^2 \mathcal{T}_1^2 \mathcal{T}_2^2
+2 (1-\rho)d_{21}^2 \mathcal{T}_1^2 \mathcal{T}_3^2 -2 \rho d_{21}^2 \mathcal{T}_2^2 \mathcal{T}_3^2 -\\
& -2 \rho (1-\rho)^2 d_{21}^4 \mathcal{T}_1^2 + 2 \rho^2 (1-\rho) d_{31}^4 \mathcal{T}_2^2 - 2 \rho (1-\rho) d_{21}^4\mathcal{T}_3^2 +\\
& + \rho^2 (1-\rho)^2 d_{21}^6 = 0
\end{split}\quad ,
\end{equation*}
that can be simplified to
\begin{equation*}
((1-\rho)\TT_{1}^2 + \rho \TT_{2}^2 - \TT_{3}^2 - \rho(1-\rho)d_{21}^2)^2
= 0.
\end{equation*}
This means that, when the receivers degenerate to an aligned
configuration, the Kummer's surface $S$ degenerates as well to the
(double) quadric surface $\sigma$.

As we can expect, the range surfaces for the aligned sensors configurations correspond to particular points in the parameter space $\mathcal{P}.$ Following the notations of Section \ref{sec:Cayley}, in such situations the parameters $ a,b,c $ assume the values
$$
(a,b,c) = \left\{
\begin{array}{ccl}
(1,-1,-1) & \text{if} & \m{1} \in r_1^0 \\
(1,1,1) & \quad\text{if}\quad & \m{2} \in r_2^0 \\
(-1,-1,1) & \text{if} & \m{3} \in r_3^0
\end{array} \right..
$$
These are the coordinates of the $3$ singular points of the Cayley's surface lying on the border of $\tilde{\mathcal{P}}$. Differently from the general case, at these points the fibers in $\mathcal{P}$ have dimension $2.$ Indeed, we need both the value of $ \rho $ and the length of $ d_{21} $ to reconstruct the position of the receivers, up to the choice of a reference frame in the Euclidean $x$--plane.
\end{rem}


\section{A look at the range--based localization in the 3D space}
\label{sec:TOA-3D}

In the previous sections of the manuscript we worked under the assumption of coplanarity of the source and the receivers. Here we remove this restriction. Let us start by adjusting the definition of the range model.
\begin{defn}\label{def:TOAmodel3D}
The statistical range model in the Euclidean 3D space is
\begin{align}
\bs{\widehat{\TT}_{3,r}}(\x) = \left(\widehat{\mathcal{T}}_{1}\left(\mb{x}\right), \dots, \widehat{\mathcal{T}}_{r}\left(\mb{x}\right)\right).
\end{align}
The deterministic part of this model is obtained by setting $\varepsilon_{i} = 0$ in $\bs{\widehat{\TT}_r}$, which gives us the range map:
\begin{center}
$\begin{matrix} \bs{\TT_{3,r}}&: \mathbb{R}^{3}& \rightarrow & \mathbb{R}^{r}
\\ &\mb{x} &\mapsto& \left(d_{1}\left(\mb{x}\right), \dots, d_{r}\left(\mb{x}\right)\right).
\end{matrix}$
\end{center}
The target set is referred to as the $\mathcal{T}$--space and we indicate its points with $\bs{\mathcal{T}}=\left(\mathcal{T}_{1}, \dots,\mathcal{T}_{r}\right)$.
\end{defn}

As for the planar case, the range--based localization is geometrically equivalent to computing the intersection of the level sets of the range functions.
\begin{defn}\label{levelsets3D}
Let $\mathcal{T} \in \mathbb{R}$. The set
\begin{align}
A_{i}\left(\mathcal{T}\right) = \left\{\mb{x} \in \mathbb{R}^{3} \vert\, d_{i}(\x) = \mathcal{T} \right\}
\end{align}
is the level set of $d_{i}(\x)$ in the $x$--space.
\end{defn}
\begin{rem}
$A_i(\TT)$ is a sphere centered at $\m{i}$ and with radius $\TT$, when $\TT > 0$. If $\ \TT < 0$ then $A_i(\TT) = \varnothing$. Finally, if $\TT = 0$ we have $A_i(0) = \{\m{i}\}$.
\end{rem}
\noindent In the noiseless scenario, given $\bs{\TT}=(\TT_1,\dots,\TT_r)$ the position of the source is one of the point in the intersection $\bigcap_{i=1}^r A_i(\TT_i)$ (see Figure \ref{fig:spheresintersections}).
\begin{figure}[htb]
\centering
   {\includegraphics[width=10cm,bb=163 330 448 461]{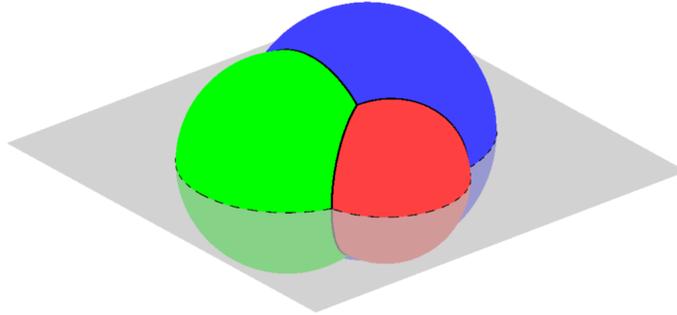}}
 \caption{The level set $A_i(\TT)$ of $\bs{\TT_{3,r}}$ is the sphere with center $\m{i}$ and radius $\TT.$ If $r=3,$ the source is placed at one of the intersection points of the three spheres. There exist another admissible source position, symmetric with respect to the grey plane $H$ containing the sensors.}\label{fig:spheresintersections}
\end{figure}

In the following we study separately the case with $r=2$ and $r=3$ sensors.

\subsection{r=2}\label{sec:3Dr2}
The component function $d_{i} (\x)$ of $\bs{\TT_{3,2}}$ is differentiable in $\RR^3 \setminus \m{i},$ for $i=1,2.$ Therefore, $\bs{\TT_{3,2}}$ is differentiable in $D=\RR^{3} \setminus \{\m{1},\m{2}\}.$ The $i$--th row of $J(\x)$ is equal to
\begin{equation}\label{eq:Jac3D}
\nabla d_{i} (\x)= \left(\frac{x-x_{i}}{d_{i}(\x)}, \frac{y-y_{i}}{d_{i}(\x)}, \frac{z-z_{i}}{d_{i}(\x)}\right) = \bs{\tilde{d}_i}(\x), \quad \ i = 1,2.
\end{equation}
\begin{thm}\label{rank3Dr2}
Let $r$ be the line through $\m{1}$, $\m{2}$ and $J(\x)$ the Jacobian matrix of $\bs{\TT_{3,2}}$ at $\mb{x} \in D$. Then
\begin{equation*}
\rank \ J(\x) = \begin{cases}
1\quad \text{if}\ \,\mb{x} \in r \cap D, \cr
2\quad \text{otherwise}. \end{cases}
\end{equation*}
\end{thm}
\begin{proof}
Assume $\mb{x} \neq \m{i}$, for $i = 1,2$. The rank of $J(\x)$ drops if, and only if, $\bs{\tilde{d}_1}(\x)$ and $\bs{\tilde{d}_2}(\x)$ have the same direction or, equivalently, are linearly dependent. This happens if, and only if, $\mb{x}$ lies on $r.$
\end{proof}
From the above theorem we have that $\bs{\TT_{3,2}}$ is neither injective nor locally injective, because the rank of the Jacobian matrix is strictly smaller than $3$ at every point in $D.$ Now we describe the set of feasible range measurements and we give the localization formula.
\begin{thm}\label{thm:caso2mic3D}
The image of $\bs{\TT_{3,2}}$ is the unbounded polyhedron $Q_{2}$ in the $\mathcal{T}$--plane defined by the triangular inequalities ~\eqref{eq:TI}. If $\bs{\mathcal{T}}= \left(\mathcal{T}_{1}, \mathcal{T}_{2}\right) \in Q_{2}$, then
\begin{align}
\left| \bs{\TT_{3,2}}^{-1}\left(\bs{\mathcal{T}}\right) \right| = \begin{cases}
\infty \quad\text{if}\ \,\bs{\mathcal{T}} \in \mathring{Q}_{2}, \\
1 \quad\text{if}\ \,\bs{\mathcal{T}} \in \partial Q_{2}. \end{cases}
\end{align}
Moreover, the points $\x$ in the fiber $\bs{\TT_{3,2}}^{-1}(\TT_1,\TT_2)$ are
\begin{align}\label{eq:solT23D}
\mb{d_1} (\x) = \frac{d_{21}^{2}+\mathcal{T}_{1}^{2}-\mathcal{T}_{2}^{2}}{2d_{21}}\, \mb{\tilde{d}_{21}} + \sqrt{\mathcal{T}_{1}^{2} - \left(\frac{d_{21}^{2}+\mathcal{T}_{1}^{2}-\mathcal{T}_{2}^{2}}{2d_{21}} \right )^{2}} (\mb{\tilde{d}_{21}'}\cos\varphi+\mb{\tilde{d}_{21}''}\sin\varphi),
\end{align}
where $\varphi\in [0,2\pi)$ and $\{\mb{\tilde{d}_{21}},\mb{\tilde{d}_{21}'},\mb{\tilde{d}_{21}''}\}$ is an orthonormal basis of $\RR^3.$
\end{thm}
\begin{proof}
The solutions of the localization problem in two dimensions are given by formula \eqref{eq:solT2}. In that case, for $\bs{\TT}\in\mathring{Q}_2$ we have two solutions symmetric with respect to the line $r$. In the three dimensional space, the symmetry of the problem implies that the fiber $\bs{\TT_{3,2}}^{-1}(\bs{\TT})$ is the circle described by formula \eqref{eq:solT23D}, having center on $r$ and contained in a plane orthogonal to $r.$ We have uniqueness of localization if, and only if, $\bs{\TT}\in\partial Q_2,$ or equivalently if the source $\x$ lies on $r$.
\end{proof}
\begin{rem}\label{T32toT2}
Let us assume that $\m{1},\m{2}$ are on the $x$--axis and let $\RR^2_+$ be the half--plane with $y\geq 0.$ Then we can think of $\RR^3$ as obtained by rotating $\RR^2_+$ around the $x$--axis. In particular, given a point $\mb{y}\in\RR^2_+$ we set $\Upsilon(\mb{y})$ the circle described by the rotation of $\mb{y}.$ Obviously, every $\x\in\RR^3$ belongs to only one $\Upsilon(\mb{y}),$ for a suitable $\mb{y}.$ This way, we can define the map
$$
\begin{array}{cccc}
\psi: & \RR^3 & \rightarrow & \RR^2_+\\
& \x & \mapsto & \mb{y}
\end{array}
$$
if $\x\in\Upsilon(\mb{y}).$ From Theorem \ref{thm:caso2mic3D} it follows that $\bs{\TT_{3,2}}=\bs{\TT_2}\circ\psi,$ which gives a geometric interpretation of the results contained in the theorem.
\end{rem}

\subsection{r=3}\label{sec:3Dr3}
The localization depends on the relative positions of the three sensors, namely whether they are collinear or not.
\begin{defn}
We set $H$ the smallest affine subspace of $\RR^3$ containing the receivers $\m{1},\m{2},\m{3}.$
\end{defn}
\noindent Clearly, $H$ is a line if the sensors are collinear, otherwise it is a plane. As usual, we begin with the local analysis of $\bs{\TT_{3,3}}.$
\begin{thm}\label{rank3Dr3}
The function $\bs{\TT_{3,3}}$ is differentiable in $D=\RR^{3} \setminus \{\m{1},\m{2},\m{3}\}.$ Let $J(\x)$ be the Jacobian matrix of $\bs{\TT_{3,3}}$ at $\mb{x} \in D$. We have the following:
\begin{enumerate}[label=(\roman{*}), ref=(\roman{*})]
\item
if $\m{1}$, $\m{2}$ and $\m{3}$ are not collinear, then
\begin{equation*}
\rank \ J(\x) = \begin{cases}
2\quad \text{if}\ \,\mb{x} \in  H\cap D, \cr
3\quad \text{otherwise}; \end{cases}
\end{equation*}
\item
if $\m{1}$, $\m{2}$ and $\m{3}$ are collinear, then
\begin{equation*}
\rank \ J(\x) = \begin{cases}
1\quad \text{if}\ \,\mb{x} \in H \cap D, \cr
2\quad \text{otherwise}. \end{cases}
\end{equation*}
\end{enumerate}
\end{thm}
\begin{proof}
The rank of $J(\x)$ is equal to the number of its linearly independent rows, which have the form \eqref{eq:Jac3D}. Thus, we have to study the number of linearly independent vectors $\bs{\tilde{d}_i}(\x),\ i=1,2,3.$ We have the following cases.
\begin{enumerate}[label=(\roman{*}), ref=(\roman{*})]
\item
Let $\m{1}$, $\m{2}$ and $\m{3}$ be not collinear. If $\x$ lies on $H,$ the three displacement vector are coplanar, therefore $\rank \ J(\x) = 2.$ Otherwise, the three vectors are independent and the rank of $J(\x)$ does not drop.
\item
Let $\m{1}$, $\m{2}$ and $\m{3}$ be on the line $H.$ If also $\x\in H,$ the three displacements vectors are parallel and $\rank \ J(\x) = 1.$ Otherwise, the three vectors are coplanar and $\rank \ J(\x) = 2.$
\end{enumerate}
\end{proof}
From the above theorem, we have that if the receivers are not collinear, then the map $\bs{\TT_{3,3}}$ is locally injective at every $\x$ not lying on the plane $H.$ On the other hand, $\bs{\TT_{3,3}}$ is injective only on $H$, because the localization is symmetric with respect to $H$ (see Figure \ref{fig:spheresintersections}). Finally, in the collinear configuration of the receivers the localization behaves very similar to the case with $r=2.$ Now we give the main results of the section, i.e. the description of $\mbox{Im}(\bs{\TT_{3,3}})$ and the inverse range map.
\begin{thm}\label{thm:caso3mic3D}
Let $\m{1},\m{2},\m{3} $ be non collinear. Then $\mbox{Im}(\bs{\TT_{3,3}})$ is the domain in the $\TT$--space with boundary $\mbox{Im}(\bs{\TT_{3}})$ and containing all the points that verify the inequality
\begin{equation}\label{solido-t3,3}
\begin{split}
 & d_{32}^2 \mathcal{T}_1^4 + d_{31}^2 \mathcal{T}_2^4 + d_{21}^2 \mathcal{T}_3^4- 2 \mb{d_{32}} \cdot \mb{d_{31}} \mathcal{T}_1^2 \mathcal{T}_2^2
+2 \mb{d_{32}} \cdot \mb{d_{21}} \mathcal{T}_1^2 \mathcal{T}_3^2 -2 \mb{d_{31}} \cdot \mb{d_{21}} \mathcal{T}_2^2 \mathcal{T}_3^2 -\\
& -2 \mb{d_{21}} \cdot \mb{d_{31}} d_{32}^2 \mathcal{T}_1^2 + 2 \mb{d_{32}} \cdot \mb{d_{21}} d_{31}^2\mathcal{T}_2^2 - 2 \mb{d_{32}} \cdot \mb{d_{31}} d_{21}^2\mathcal{T}_3^2+ d_{21}^2 d_{31}^2 d_{32}^2 \leq 0.
\end{split}
\end{equation}
If $\bs{\TT} = (\TT_1,\TT_2,\TT_3) \in \mbox{Im}(\bs{\TT_{3,3}})$, then
$$
\vert \bs{\TT_{3,3}}^{-1}(\bs{\TT}) \vert =
\left\{ \begin{array}{cl}
2 & \mbox{ if }\ \bs{\TT} \in \mbox{Im}(\bs{\TT_{3,3}}) \setminus  \mbox{Im}(\bs{\TT_{3}}) \\
1 & \mbox{ if }\ \bs{\TT} \in  \mbox{Im}(\bs{\TT_{3}}) \end{array} \right. .
$$
\end{thm}
\begin{proof}
In this proof, we use the exterior algebra formalism over the four dimensional Minkowski space $\RR^{3,1}.$ Therefore, we preliminarily adapt the definitions and notation set in Section \ref{sec:MK} for $\RR^{2,1}.$
Let $ \mb{e_1}, \mb{e_2}, \mb{e_3}$ be an orthonormal basis of the Euclidean space containing source and receivers, where we choose $\mb{e_1}$ and $\mb{e_2}$ parallel to $H.$
By adding the unit vector $\mb{e_4}$ parallel to the time direction, we get the orthonormal basis $ \mb{e_1}, \mb{e_2}, \mb{e_3}, \mb{e_4}$ of $\mathbb{R}^{3,1}.$
For any two vectors $ \mb{u} = u_1 \mb{e_1} + \dots + u_4 \mb{e_4}$ and $\mb{v} = v_1 \mb{e_1} + \dots + v_4 \mb{e_4},$ their inner product in $\mathbb{R}^{3,1}$ is
$$
\langle \mb{u} , \mb{v} \rangle = u_1 v_1 + u_2 v_2 + u_3 v_3 - u_4 v_4.
$$
Consequently, the associated norm is $ \Vert \mb{v} \Vert^2 = v_1^2 + v_2^2 + v_3^3 - v_4^2.$ As usual, we denote with the symbol $\cdot$ the product of two space vectors. If we fix the volume form $ \bs{\omega} = \mb{e_1} \wedge \mb{e_2} \wedge \mb{e_3} \wedge \mb{e_4},$ then the Hodge operator is defined on the natural basis of the exterior algebra as
\begin{equation*}
\begin{array}{c}
\ast 1=\bs\omega,\qquad\ast\bs\omega=-1,\\
\ast\mb{e_1}=\mb{e_2}\wedge\mb{e_3}\wedge\mb{e_4},\quad
\ast\mb{e_2}=\mb{e_1}\wedge\mb{e_3}\wedge\mb{e_4},\quad
\ast\mb{e_3}=\mb{e_1}\wedge\mb{e_2}\wedge\mb{e_4},\quad
\ast\mb{e_4}=-\mb{e_1}\wedge\mb{e_2}\wedge\mb{e_3},\\
\ast(\mb{e_1}\wedge\mb{e_2})=\mb{e_3}\wedge\mb{e_4},\quad
\ast(\mb{e_1}\wedge\mb{e_3})=-\mb{e_2}\wedge\mb{e_4},\quad
\ast(\mb{e_1}\wedge\mb{e_4})=-\mb{e_2}\wedge\mb{e_3},\\
\ast(\mb{e_2}\wedge\mb{e_3})=\mb{e_1}\wedge\mb{e_4},\quad
\ast(\mb{e_2}\wedge\mb{e_4})=\mb{e_1}\wedge\mb{e_3},\quad
\ast(\mb{e_3}\wedge\mb{e_4})=-\mb{e_1}\wedge\mb{e_2},\\
\ast(\mb{e_1}\wedge\mb{e_2}\wedge\mb{e_3})=-\mb{e_4},\quad
\ast(\mb{e_1}\wedge\mb{e_2}\wedge\mb{e_4})=\mb{e_3},\quad
\ast(\mb{e_1}\wedge\mb{e_3}\wedge\mb{e_4})=\mb{e_2},\quad
\ast(\mb{e_2}\wedge\mb{e_3}\wedge\mb{e_4})=\mb{e_1}.
\end{array}
\end{equation*}

A point in $\mathbb{R}^{3,1}$ has four coordinates: the first three ones specify the spatial position of the point, while the last one is the time coordinate.
Given $\bs{\mathcal{T}}=(\mathcal{T}_1,\mathcal{T}_2,\mathcal{T}_3)$ in the $\mathcal{T}$--space, we define the spacetime position of the receivers as
$
\mb{M_i}\left(\bs{\mathcal{T}}\right) = \left(x_{i},y_{i},0,\mathcal{T}_{i}\right),\ i = 1,2,3,
$
while $\mb{X} = \left(x,y,z,\mathcal{T}\right)$ is the generic point of $\mathbb{R}^{3,1}.$ The displacement vectors are then
$$
\mb{D_i}(\X,\bs{\TT}) = \X - \mb{M_i}(\bs{\TT})
\qquad\text{and}\qquad
\mb{D_{ji}}(\X,\bs{\TT}) = \mb{M_j}(\bs{\TT}) - \mb{M_i}(\bs{\TT}),
\qquad\text{for}\ i,j = 1,2,3\ \text{and}\ i \neq j.
$$

In this setting, to localize the source we have to solve the system
\begin{equation}\label{eq:sys8-1}
\left\{\begin{array}{l}
\Vert\mb{D_1}(\X,\bs{\TT})\Vert^2 = 0 \\
\Vert\mb{D_2}(\X,\bs{\TT})\Vert^2 = 0 \\
\Vert\mb{D_3}(\X,\bs{\TT})\Vert^2 = 0 \\
\langle\mb{D_1}(\X,\bs{\TT}), \mb{e_4}\rangle = \TT_1
\end{array}\right..
\end{equation}
Through the same computations made in Section \ref{sec:MK}, we arrive to
\begin{equation}\label{eq:sys8-2}
\left\{\begin{array}{l}
\Vert\mb{D_1}(\X,\bs{\TT})\Vert^2 = 0 \\
i_{\mb{D_1}(\X,\bs{\TT})}\mb{D_{21}}(\bs{\TT}) =
\frac 12 \Vert \mb{D_{21}}(\bs{\TT}) \Vert^2 \\
i_{\mb{D_1}(\X,\bs{\TT})}\mb{D_{31}}(\bs{\TT}) =
\frac 12 \Vert \mb{D_{31}}(\bs{\TT}) \Vert^2 \\
i_{\mb{D_1}(\X,\bs{\TT})}\mb{e_4} = \TT_1
\end{array}\right..
\end{equation}
Since the sensors are not aligned, in the last rows of system \eqref{eq:sys8-2} we have three independent linear equations. Therefore, they define three hyperplanes in $\RR^{3,1}$ intersecting along the line $l(\bs{\TT})$ of equation:
\begin{equation*}
\begin{array}{lcl}
i_{\mb{D_1}(\mb{X},\bs{\TT})}(\mb{D_{21}}(\bs{\TT})^\flat\wedge\mb{D_{31}}(\bs{\TT})^\flat\wedge\mb{e_4}^\flat)
& = & \frac{1}{2}\left(\Vert\mb{D_{21}}(\bs{\TT})\Vert^2\mb{D_{31}}(\bs{\TT})^\flat -\Vert\mb{D_{31}}(\bs{\TT})\Vert^2\mb{D_{21}}(\bs{\TT})^\flat\right)\wedge\mb{e_4}^\flat +\\[1mm]
&& +\TT_1\, \mb{D_{21}}(\bs{\TT})^\flat\wedge\mb{D_{31}}(\bs{\TT})^\flat.
\end{array}
\end{equation*}

A vector $\mb{v}$ is parallel to $l(\bs{\TT})$ if it satisfies
\begin{align*}
i_{\mb{v}}(\mb{D_{21}}(\bs{\TT})^\flat\wedge\mb{D_{31}}(\bs{\TT})^\flat\wedge\mb{e_4}^\flat) = 0.
\end{align*}
A solution is
\begin{align*}
\mb{v} = \ast \left( \mb{D_{21}}(\bs{\TT}) \wedge \mb{D_{31}}(\bs{\TT}) \wedge \mb{e_4} \right) = \ast ( \mb{d_{21}} \wedge \mb{d_{31}} \wedge \mb{e_4} ).
\end{align*}

In order to obtain a point $\mb{l_0}(\bs{\TT})\in l(\bs{\TT}),$ we take the intersection between $l(\bs{\TT})$ and the hyperplane $H$, whose equation is $i_{\mb{D_1}(\X,\bs{\TT})}\mb{e_3} = 0$. If we set
$$
\bs{\Omega} =\mb{D_{21}}(\bs{\TT}) \wedge \mb{D_{31}}(\bs{\TT}) \wedge \mb{e_3} \wedge \mb{e_4}=
\mb{d_{21}} \wedge \mb{d_{31}}  \wedge \mb{e_3} \wedge \mb{e_4},
$$
then the equation for $\mb{D_1}(\mb{l_0}(\bs{\TT}))$ is
\begin{equation*}
\begin{array}{lcl}
i_{\mb{D_1}(\mb{l_0}(\bs{\TT}))} \bs{\Omega}^\flat & = &
\frac{1}{2}\left(\Vert\mb{D_{21}}(\bs{\TT})\Vert^2\mb{D_{31}}(\bs{\TT})^\flat -\Vert\mb{D_{31}}(\bs{\TT})\Vert^2\mb{D_{21}}(\bs{\TT})^\flat\right)\wedge\mb{e_3}^\flat\wedge\mb{e_4}^\flat -\\[1mm]
&&-\TT_1\, \mb{D_{21}}(\bs{\TT})^\flat\wedge\mb{D_{31}}(\bs{\TT})^\flat\wedge\mb{e_3}^\flat.
\end{array}
\end{equation*}
Hence
\begin{equation*}
\begin{array}{lcll}
\mb{D_1}(\mb{l_0}(\bs{\TT})) &=& {\displaystyle\frac 1{2 \ast \bs\Omega}} 
& \left\{ \Vert \mb{D_{21}}(\bs{\TT}) \Vert^2 \ast\left( \mb{d_{31}} \wedge \mb{e_3} \wedge \mb{e_4} \right) - \Vert \mb{D_{31}}(\bs{\TT}) \Vert^2 \ast\left( \mb{d_{21}} \wedge \mb{e_3} \wedge \mb{e_4} \right) - \right. \\[1mm]
&&& \left. - 2 \mathcal{T}_1 \ast\left( \mb{D_{21}}(\bs{\TT}) \wedge \mb{D_{31}}(\bs{\TT}) \wedge \mb{e_3} \right) \right\}.
\end{array}
\end{equation*}

With straightforward computations, we have
$$
\mb{D_1}(\mb{l_0}(\bs{\TT})) = -\frac{\ast[((d_{31}^2 + \mathcal{T}_1^2 - \mathcal{T}_3^2) \mb{d_{21}} - (d_{21}^2 + \mathcal{T}_1^2 - \mathcal{T}_2^2) \mb{d_{31}})\wedge\mb{e_3}\wedge\mb{e_4}]}{2\ast\bs\Omega}
- \mathcal{T}_1 \mb{e_4},
$$
By substituting the parametric description of $l(\bs\TT)$ in the first equation of system (\ref{eq:sys8-2}), we get the quadratic equation $\Vert \mb{D_1}(\mb{l_0}(\bs{\TT})) + t \mb{v} \Vert^2 = 0$ in the real parameter $t:$
\begin{equation*} \begin{split}
t^2 =& -\frac{1}{4 \Vert \mb{d_{21}} \wedge \mb{d_{31}} \Vert^4} \left( d_{32}^2 \mathcal{T}_1^4 + d_{31}^2 \mathcal{T}_2^4 + d_{21}^2 \mathcal{T}_3^4- 2 \mb{d_{32}} \cdot \mb{d_{31}} \mathcal{T}_1^2 \mathcal{T}_2^2
+2 \mb{d_{32}} \cdot \mb{d_{21}} \mathcal{T}_1^2 \mathcal{T}_3^2- \right. \\ & \left. -2 \mb{d_{31}} \cdot \mb{d_{21}} \mathcal{T}_2^2 \mathcal{T}_3^2
 -2 \mb{d_{21}} \cdot \mb{d_{31}} d_{32}^2 \mathcal{T}_1^2 + 2 \mb{d_{32}} \cdot \mb{d_{21}} d_{31}^2\mathcal{T}_2^2 - 2 \mb{d_{32}} \cdot \mb{d_{31}} d_{21}^2\mathcal{T}_3^2+ d_{21}^2 d_{31}^2 d_{32}^2 \right).
\end{split} \end{equation*}
If inequality \eqref{solido-t3,3} is strictly satisfied, the above equation has two real opposite solutions $t^*$ and $-t^*.$ If \eqref{solido-t3,3} is satisfied as an equality, the point $\bs\TT$ lies on Im($\bs{\TT_3}$) and the two solutions of the equation are coincident.
\end{proof}

\begin{rem}
As shown in Figure \ref{fig:spheresintersections}, the two ranges $ \TT_1, \TT_2$ define two spheres in the $x$--space that meet along a circle with center on the line $r_3$ through $\m{1},\m{2},$ suitable radius, and contained in a plane orthogonal to $r_3$. The third range $ \TT_3 $ fixes another sphere with center at $ \m{3}.$ There exist exactly two values $\{\TT_m,\TT_M\}$ of $ \TT_3 $ such that the circle and the sphere meet at only a point. In these cases $t^*=0,$ i.e the source lies on the plane $H$ and its position is given by $\mb{D_1}(\mb{l_0}(\bs{\TT})).$ Moreover, $\bs\TT=(\TT_1,\TT_2,\TT_3)\in\text{Im}(\bs{\TT_3}).$
For $\TT_m<\TT_3<\TT_M,$ there are two intersection points at $\mb{D_1}(\mb{l_0}(\bs{\TT})) \pm t^* \mb{v}$, symmetric with respect to $H,$ and $\bs\TT\in\mbox{Im}(\bs{\TT_{3,3}}) \setminus  \mbox{Im}(\bs{\TT_{3}}).$ The explained construction can be interpreted in terms of Figure \ref{fig:TOAsurf}: the lines parallel to the $ \TT_3$--axis cut $ \mbox{Im}(\bs{\TT_3}) $ at two points, and the segment they identify is contained in $ \mbox{Im}(\bs{\TT_{3,3}}).$
\end{rem}

Now, we assume that $ \m{1},\m{2},\m{3} $ are collinear, that is to say, $ H $ is a line. For geometrical reasons, the localization problem is symmetric around $H.$ The three TOAs have to satisfy equation \eqref{eq:ineqboh}. This condition means that the three spheres have a common circle. In order to characterize the TOA map, the crucial observation is $\bs{\TT_{3,3}}(\x)=(\bs{\TT_{3,2}}(\x),\varphi(\bs{\TT_{3,2}}(\x)),$ where $\varphi$ was defined in Proposition \ref{headache}. Then, Theorem \ref{thm:caso2mic3D} applies with the straightforward changes.


\section{The range difference model: a brief description}
\label{sec:TDOA-summary}

As said in the Introduction, the last two sections of this manuscript are devoted to the comparison of the two deterministic models behind range and range difference based source localization. For the convenience of the reader, in this section we give a brief summary of the analysis contained in \cite{Compagnoni2013a}.
More specifically, we summarize the main properties of the model describing the range difference--based localization in the noiseless scenario with
three receivers and a coplanar source. In particular, we emphasize the similarities and the differences with respect to the range--based localization.

As for $\bs{\TT_3}$, we work with three receivers $\m{1}$, $\m{2}$ and $\m{3}$ placed at known positions in the $x$--plane. Unlike the range case, this is the minimal number of sensors necessary for uniquely locating the source, at least locally. After choosing $\m{3}$ as the reference receiver\footnote{According to our choice of the reference receiver, we change to $3$ every subscript $0$ appearing in \cite{Compagnoni2013a}. As a consequence, we have to correct signs in formulas whenever necessary.}, we define the range differences with respect to $\m{3}$ as the pseudoranges (i.e. the range
differences)
\begin{equation}\label{eq:TDOAdef}
\tau_{i}(\x) = d_{i}\left(\mb{x}\right) - d_{3}\left(\mb{x}\right), \
\ i=1,2.
\end{equation}
The goal of the range difference--based localization is to infer the source position $\x$ from the two measured range differences $(\tau_1,\tau_2).$ As explained in \cite{Compagnoni2013a}, we do not consider the third range difference between $\m{1}$ and $\m{2},$ because in the noiseless scenario this one is linearly dependent on the other two.

As for the range case, the range difference localization can be nicely interpreted in terms of intersections of real conics. As done for Definition \ref{levelsets}, we have the following construction (see
Definitions \ref{def:rettecerchio} and \ref{def:allineati} for the notation).\footnote{In order to avoid confusion with the range notation, we call $B_i(\bt)$ the set that in \cite{Compagnoni2013a} is called $A_i(\bt)$.}
\begin{defn}[$\!\!$\cite{Compagnoni2013a}, Definition 3.3]\label{TDOA-lev-set}
Given $ \boldsymbol{\tau} = (\tau_1,\tau_2)\in\RR^2,$ the level sets of the pseudoranges \eqref{eq:TDOAdef} are
$$
B_i(\boldsymbol{\tau}) = \{ \mb{x} \ \vert \ d_{i}\left(\mb{x}\right) -
d_{3}\left(\mb{x}\right) = \tau_{i} \},\ i=1,2.
$$
\end{defn}

\begin{lemma}[$\!\!$\cite{Compagnoni2013a}, Lemma 3.4]\label{LemmaTriang}
If $ \vert \tau_i \vert > d_{3i},$ then $ B_{i}(\boldsymbol{\tau})
= \emptyset.$ Moreover, if $ 0 < \vert \tau_i \vert < d_{3i},$
then $ B_{i}(\boldsymbol{\tau}) $ is the branch of hyperbola with
foci $ \m{i},\m{3} $ and parameter $ \tau_i,$ while
$$
B_{i}(\bt) = \left\{
\begin{array}{lcl}
r_{j}^+ & \text{if} & \tau_i = d_{3i}\\
r_{j}^- & \text{if} & \tau_i = -d_{3i}\\
a_{j} & \text{if} & \tau_i = 0
\end{array}\right.,
$$
where $ j \not= i, \{ i, j \} = \{ 1, 2 \},$ and $ a_{j} $ is the
line that bisects the line segment $ r_{j}^0.$
\end{lemma}
\noindent From a geometric standpoint, the psuedorange localization problem is equivalent to find the intersection between the two level sets $B_{1}(\bt),B_{2}(\bt).$ By Lemma \ref{LemmaTriang}, for the general case we have to consider the intersection of two branches of hyperbolas. This fact makes the range difference analysis more complicated than the one for the ranges, which is based on the intersection of circles.

As usual, the relevant data can be encoded into a unique vector map.
\begin{defn}[$\!\!$\cite{Compagnoni2013a}, Definition 2.4] \label{def:TDOA-map}
The map from the position of the source in the $x$--plane to the
defined pseudoranges
\begin{center}
$\begin{matrix} \bs{\tau_2}&: \mathbb{R}^{2}& \rightarrow &
\mathbb{R}^{2}
\\ &\mb{x} &\mapsto& \left(\tau_{1}\left(\mb{x}\right), \tau_{2}\left(\mb{x}\right)\right)
\end{matrix}$
\end{center}
is called the range difference map. The target set is referred to as the $\tau$--plane.
\end{defn}
\noindent The range difference map is the main object of the analysis contained in \cite{Compagnoni2013a}. In particular, its local and global invertibility and its image have been studied in full details .

About the local invertibility, we have the following result.
\begin{thm}[$\!\!$\cite{Compagnoni2013a}, Theorem 3.2]
Let $ J(\mb{x}) $ be the Jacobian matrix of $ \bs{\tau_2} $ at $
\mb{x} \not= \m{1},\m{2}, \m{3}.$ Then,
\begin{enumerate}
\item if $ \m{1},\m{2}, \m{3} $ are not collinear,
then
$$
\rank(J(\mb{x})) = \left\{ \begin{array}{cl} 1 & \mbox{if }
\mb{x} \in \left( \cup_{i=1}^3 (r_i^- \cup r_i^+) \right)
\\ 2 & \mbox{otherwise};
\end{array} \right.\ ;
$$
\item if $ \m{1},\m{2}, \m{3} $ are contained in a
line $r$, with $ \mb{m}_{3}\in r_3^0,$
then
$$
\rank(J(\mb{x})) = \left\{ \begin{array}{cl} 0 & \mbox{if }
\mb{x} \in r\setminus r_3^0 \\ 1 &
\mbox{if } \mb{x}
\in r_3^0 \\
2 & \mbox{otherwise} \end{array} \right.\ .
$$
\end{enumerate}
\end{thm}
\noindent
In the second item, without loss of generality we assume $\m{3} $ between $ \m{1},\m{2}$, similarly to what we have done in Section \ref{sec:collinear-microphones}. For a generic sensors configuration, the set $D=\cup_{i=1}^3 (r_i^- \cup r_i^+)$  (respectively, $D=r$ for an aligned configuration) is the degeneracy locus of $\bs{\tau_2},$ where the map is not locally invertible.

The image of $D$ in the $\tau$--plane is a subset of the boundary of $\text{Im}(\bs{\tau_2}).$ In particular, $\bs{\tau_2}(D)$ is contained in the facets of the bounded convex polytope $P_2$ defined by the triangular inequalities involving the range differences.
\begin{defn}[$\!\!$\cite{Compagnoni2013a}, Section 5]\label{def:P2}
Let $ P_2 $ be the region in the $ \tau$--plane containing the
points whose coordinates satisfy
\begin{align*}
P_{2}: \left\{\begin{matrix} -d_{31} \leq \tau_{1} \leq d_{31}
\\ - d_{32} \leq \tau_{2} \leq d_{32}
\\ - d_{21} \leq \tau_{2} - \tau_{1} \leq d_{21}
\end{matrix}\right. .
\end{align*}
\end{defn}
\noindent
Clearly, any admissible set of range difference measurements satisfies the triangular inequalities, therefore $\text{Im}(\bs{\tau_2})\subseteq P_2.$ The geometric description of $P_2$ is given in the following theorem.
\begin{thm}[$\!\!$\cite{Compagnoni2013a}, Theorem 5.3]\label{thm:P2}
$P_{2}$ is a polygon (a 2--dimensional convex polytope). Moreover,
if the points $\m{1}$, $\m{2}$ and $\m{3}$ are not
collinear, then $P_{2}$ has exactly 6 facets, named $F_{k}^{\pm}$,
$k = 1,2,3$ in \cite{Compagnoni2013a}, which drop to $ 4 $ if the receivers
$\m{1}$, $\m{2}$ and $\m{3}$ are collinear.
\end{thm}
\begin{figure}[htb]
\includegraphics[bb=107 327 504 464]{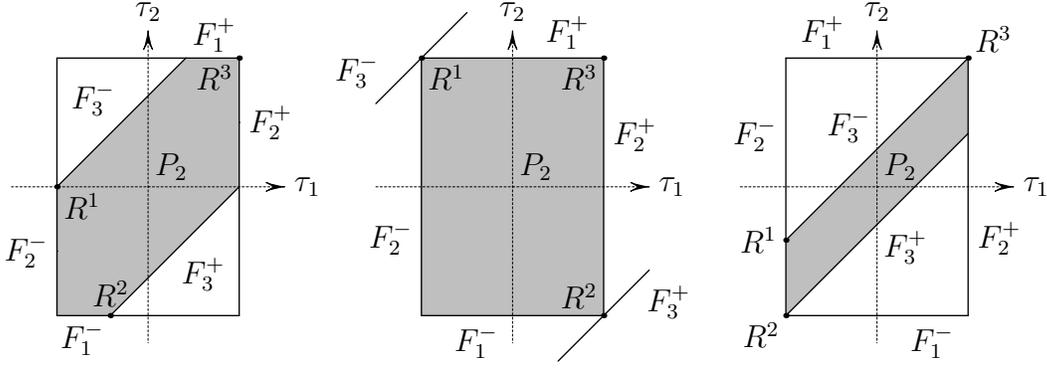}\\
\caption{Left--hand side: polygon $P_{2}$ (in shaded gray) under
the assumption that the points $\m{1}$, $\m{2}$ and
$\m{3}$ are not collinear. Center: polygon $P_{2}$ (in shaded
gray) in the case of three collinear points with $\m{3}\in r_3^0$.
Right--hand side: polygon
$P_{2}$ (in shaded gray) when the sensors lie on a line, but with
$\m{1}$ between $\m{2}$ and
$\m{3}$. The image of $\bs{\tau_2}$ is a subset of $P_2.$}\label{sez9-fig1}
\end{figure}
\noindent
In Figure \ref{sez9-fig1}, we set the image of the sensor as $ R^i = \bs{\tau_2}(\m{i}),\ i=1,2,3.$ As we said above, we have that $\bs{\tau_2}^{-1}(F_i^{\pm})=r_i^{\mp},\ i=1,2,3.$

Now, let us focus on the system
\begin{equation}\label{eq:systTDOA}
\left\{
\begin{array}{l}
\tau_{1}=d_1(\x)-d_3(\x)\\
\tau_{2}=d_2(\x)-d_3(\x)
\end{array}\right.\ .
\end{equation}
The Minkowski space formalism that we presented in Section \ref{sec:MK} is particularly suitable for the investigation of \eqref{eq:systTDOA}. We only have to adapt the notations in order to handle the range differences insted of the ranges.
We set the points $\mb{M_i}(\boldsymbol{\tau}) = (x_i, y_i, \tau_i),\ i=1, 2,$ and $\mb{M_3} = (x_3, y_3, 0)$. If $ \mb{X} = (x,y,\tau) $ is a point in $\RR^{2,1}$, the displacement vectors are then $ \mb{D_i}(\mb{X},
\boldsymbol{\tau}) = \mb{X} - \mb{M_i}(\boldsymbol{\tau}).$
Furthermore, we set $ \mb{D_{ji}}(\boldsymbol{\tau}) =
\mb{M_j}(\boldsymbol{\tau}) - \mb{M_i}(\boldsymbol{\tau}),$ for $
1 \leq i < j \leq 3.$
Then, system \eqref{eq:systTDOA} can be readily rewritten as
\begin{equation} \label{eq:mink-cone-planes}
\left\{
\begin{array}{l}
i_{\mb{D_3}\left(\mb{X},\bt\right)} \mb{D_{31}} \left(\bt\right)^{\flat}=
-\frac 12 \parallel \mb{D_{31}}(\boldsymbol{\tau}) \parallel^2\\
i_{\mb{D_3}\left(\mb{X},\bt\right)} \mb{D_{32}} \left(\bt\right)^{\flat}=
-\frac 12 \parallel \mb{D_{32}}(\boldsymbol{\tau}) \parallel^2\\[1mm]
\parallel \mb{D_3}(\mb{X},\boldsymbol{\tau}) \parallel^2=0\\[1mm]
i_{\mb{D_3}\left(\mb{X},\bt\right)}\, \mb{e_3}^{\flat}\geq
\rm{min}(\tau_1,\tau_2,0)
\end{array}\right.\, .
\end{equation}
System \eqref{eq:mink-cone-planes} is very like to \eqref{eq:Formulation_i} and it can just as well be geometrically interpreted in terms of intersection of cones and planes.
\begin{defn}[$\!\!$\cite{Compagnoni2013a}, Definition 4.1]\label{cone-plane}
We set
\begin{enumerate}
\item $ C_3(\boldsymbol{\tau}) = \{ \mb{X} \in \RR^{2,1} \ \vert \
\parallel \mb{D_3}(\mb{X}, \boldsymbol{\tau}) \parallel^2 = 0 \};$
\item $ C_3(\boldsymbol{\tau})^- = \{ \mb{X} \in
C_3(\boldsymbol{\tau}) \ \vert \
i_{\mb{D_3}\left(\mb{X},\bt\right)}\, \mb{e_3}^{\flat}\geq 0 \}.$
\end{enumerate}
Moreover, for $ i=1,2,$ we set $$ \Pi_i(\boldsymbol{\tau}) = \{
\mb{X} \in \RR^{2,1} \ \vert \
i_{\mb{D_3}\left(\mb{X},\bt\right)} \mb{D_{3i}} \left(\bt\right)^{\flat} =
-\frac 12 \parallel\mb{D_{3i}}(\boldsymbol{\tau})
\parallel^2 \} $$ and $ L_{21}(\boldsymbol{\tau}) =
\Pi_1(\boldsymbol{\tau}) \cap \Pi_2(\boldsymbol{\tau}).$
\end{defn}
\noindent As in Section \ref{sec:MK}, $ C_3(\bs{\tau}) $ is a right circular cone with $
\mb{M_3}(\bs{\tau}) $ as vertex and $ \Pi_i(\bs{\tau}) $ is a plane. The set $ C_3(\bs{\tau})^- $ is the half--cone contained in the half--space $\TT\leq\TT_3.$ The source is placed at the projection on the $x$--plane of the intersection $ C_3(\bs{\tau})^-\cap L_{21}(\boldsymbol{\tau}).$ Now, we carry on our analysis separately for the case of a general configuration of the sensors and for the aligned one.

\subsection{The general case: $\m{1},\m{2},\m{3}$ are not contained in a line}
\label{sec:TDOA-summary-general}

If $\m{1}$, $\m{2}$ and $\m{3}$ are not
collinear, then $ \mb{D_{31}}(\bs{\tau}) $ and $ \mb{D_{32}}(\bs{\tau})
$ are linearly independent. In this scenario, $ L_{21}(\bs{\tau}) $ is the line described in the following lemma.

\begin{lemma}[$\!\!$\cite{Compagnoni2013a}, Lemma 6.1]\label{lm:paramL21}
For any $\bt\in\RR^2$, $L_{21}(\bt)=\Pi_1(\bt)
\cap \Pi_2(\bt)$ is a line. A parametric representation of
$L_{21}(\bt)$ is $\X(\lambda;\bt)=\mb{L_0}(\bs{\tau}) + \lambda
\mb{v}(\bs{\tau})$, where
\begin{equation*}
\mb{v}(\bs{\tau})= \ast( (\mb{d_{31}} \wedge \mb{d_{32}}) +
(\tau_1 \mb{d_{32}} - \tau_2 \mb{d_{31}}) \wedge \mb{e_3})
\end{equation*}
and the displacement vector of $\mb{L_0}(\bs{\tau})$ is
\begin{eqnarray*}
\mb{D_3}(\mb{L_0}(\bs{\tau}))= -\frac{\ast \left[ \left(
\Vert\mb{D_{32}}(\bs{\tau})\Vert^2\mb{d_{31}} -
\Vert\mb{D_{31}}(\bs{\tau})\Vert^2 \mb{d_{32}} \right) \wedge
\mb{e_3} \right]} {2\Vert\mb{d_{31}} \wedge \mb{d_{32}} \Vert}\,.
\end{eqnarray*}
\end{lemma}

The intersection $ C_3^- \cap L_{21}(\bs{\tau}) $ can be computed
by finding the non--positive real roots of the following quadratic equation in $\lambda$ (see \cite{Compagnoni2013a}, Equation (18)):
\begin{equation} \label{eq-deg-2}
\Vert \mb{v}(\bs{\tau}) \Vert^2 \lambda^2 + 2 \lambda \langle
\mb{D_3}(\mb{L_0}(\bs{\tau})) , \mb{v}(\bs{\tau}) \rangle + \Vert
\mb{D_3}(\mb{L_0}(\bs{\tau})) \Vert^2 = 0.
\end{equation}
The number of such solutions depends on the coefficients
$$
a(\bt)=\Vert \mb{v}(\bs{\tau}) \Vert^2,\qquad
b(\bt)=\langle\mb{D_3}(\mb{L_0}(\bs{\tau})),\mb{v}(\bs{\tau})\rangle,\qquad
c(\bt)=\Vert\mb{D_3}(\mb{L_0}(\bs{\tau})) \Vert^2
$$
of the equation, which are polynomial in $\bt$ (see \cite{Compagnoni2013a}, Definition 6.3). In particular, let us consider the following (semi)--algebraic subsets of the $\tau$--plane (see Figure \ref{fig:slide19}).
\begin{defn}[$\!\!$\cite{Compagnoni2013a}, Definitions 6.5, 6.9]\label{def:atau}
We set
$$
\begin{array}{lll}
E=\{\bt\in\RR^2 \ \vert\ a(\bs{\tau})=0\}, & \qquad &
E^\pm=\{\bt\in\RR^2 \ \vert\ \pm a(\bs{\tau}) > 0\};\\[2mm]
C=\{\bt\in\RR^2 \ \vert\ b(\bs{\tau})=0\}, & &
C^\pm=\{\bt\in\RR^2 \ \vert\ \pm b(\bs{\tau}) > 0\}.
\end{array}
$$
\end{defn}
\begin{prop}[$\!\!$\cite{Compagnoni2013a}, Propositions 6.6, 6.10--6.12]\label{atau}
$ E \subset P_2 $ is an ellipse centered at $\mb{0}= (0,0) $, and
it represents the only conic that is tangent to each side of the
hexagon $P_2$. The curve $C$ is the only cubic passing through $\mb{0},$ the points $\pm(d_{31},d_{32}),\pm(d_{31},-d_{32})$ and the intersection $E\cap \partial P_2.$
\end{prop}
\begin{figure}[htb]
\begin{center}
\resizebox{12.5cm}{!}{
  \includegraphics[bb=107 317 504 474]{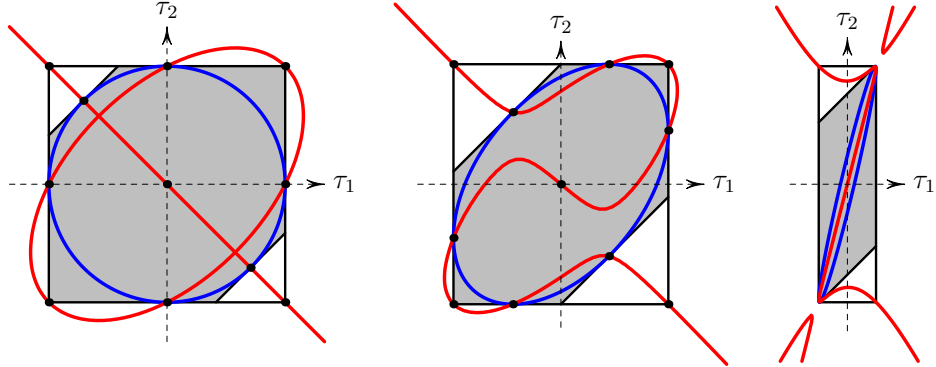}}
  \caption{\label{fig:slide19}Examples of ellipses $E$ in blue and cubics $C$ in red. The $ 11 $ distinguished points are marked in the first two pictures, but not in the last one because $ 4 $ of them are very close to each other on the upper--right vertex of the rectangle, and similarly for the $ 4 $ ones, which are close to the opposite vertex. In all the three cases $ P_2 $ is an hexagon, but it exhibits two very short sides in the right-hand picture.}
\end{center}
\end{figure}
\noindent
We give some examples of the curves $E$ and $C$ in Figure \ref{fig:slide19}. The tangency points in $E\cap\partial P_2$ are
$$
T_i^+ = \left( \langle \mb{d_{31}}, \mb{\tilde{d_{jk}}} \rangle,
\langle \mb{d_{32}}, \mb{\tilde{d_{jk}}} \rangle \right) \quad
\mbox{ and } \quad T_i^- = \left( -\langle \mb{d_{31}},
\mb{\tilde{d_{jk}}} \rangle, -\langle \mb{d_{32}},
\mb{\tilde{d_{jk}}} \rangle \right),
$$
where $1 \leq i,j,k \leq 3 $ with $k < j$ and $k \neq j$ (see \cite{Compagnoni2013a}, Definition 6.7).

We are now ready to describe the image and the global invertibility of the map $ \bs{\tau_2}$ (see Figure \ref{fig:tauimage}).
\begin{defn}[$\!\!$\cite{Compagnoni2013a}, Definition 6.15]
The set $ \mathring{P_2} \cap E^+ \cap C^+ $ is the union of three
disjoint connected components that we name $ U_1, U_2, U_3,$ where
$ R^i \in \bar{U}_i $ for $ i = 1, 2, 3$.
\end{defn}
\begin{thm}[$\!\!$\cite{Compagnoni2013a}, Theorem 6.16]\label{th:imtau2}
$\mbox{Im}(\bs{\tau_2}) = E^- \cup \bar{U}_1 \cup \bar{U}_2 \cup
\bar{U}_3 \setminus \{ T_1^\pm, T_2^\pm, T_3^\pm \}.$ Moreover,
$$
\vert\bs{\tau_2}^{-1}(\bt)\vert=\left\{
\begin{array}{lcl}
2 & \text{if} & \bt \in U_1 \cup U_2 \cup U_3\\
1 & \text{if} & \bt \in \mbox{Im}(\bs{\tau_2}) \setminus U_1 \cup U_2 \cup
U_3
\end{array}\right..
$$
\end{thm}

\begin{figure}[htb]
\begin{center}
\resizebox{5cm}{!}{
  \includegraphics[bb=227 296 384 495]{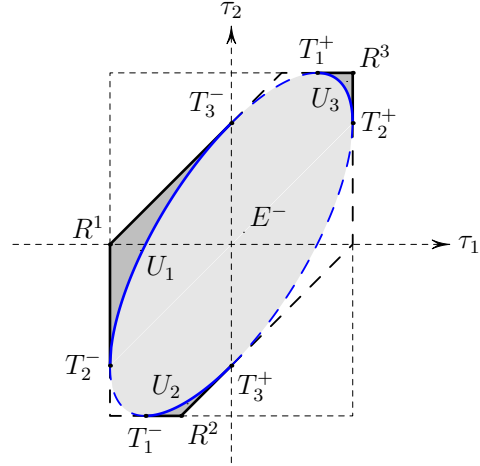}}
  \caption{\label{fig:tauimage}The image of $ \bs{\tau_2}$ is the
  gray subset of $P_2$. In the light gray region marked as $E^-$ the map
  $ \bs{\tau_2} $ is $ 1$--to--$1$, while in the medium gray
  regions $U_0 \cup U_1 \cup U_2$ the map $ \bs{\tau_2} $ is
  $1$--to--$2$. The continuous part of $\partial P_2$ and $E$,
  and the vertices $R^i$, are in the image, where $ \bs{\tau_2}$
  is $ 1$--to--$1$. The dashed part of $\partial P_2$ and $E$,
  and the tangency points $T_i^\pm$, do not belong to
  Im($\bs{\tau_2}$).}
\end{center}
\end{figure}
\noindent
Finally, for any given $ \bs{\tau} \in \mbox{Im}(\bs{\tau_2}) $ and a
negative solution $ \lambda $ of equation \eqref{eq-deg-2}, we
have the corresponding preimage in the $x$--plane (i.e. the solution of the localization problem, see Equation (20) in \cite{Compagnoni2013a}):
\begin{equation}
\label{eq:inv-image} \mb{x}(\bs{\tau}) = \mb{L_0}(\bs{\tau}) +
\lambda \ast((\tau_1 \mb{d_{32}} - \tau_2 \mb{d_{31}}) \wedge
\mb{e_3}).
\end{equation}

\subsection{The special case: $\m{1},\m{2},\m{3}$
collinear}

As assumed in Section \ref{sec:collinear-microphones} and without loss of generality, we take $ \m{3}\in r_3^0. $ This is equivalent to $\mb{d_{31}} = \rho\,\mb{d_{21}},$ where $0<\rho<1.$
By Proposition \ref{thm:P2}, in this configuration the polygon $P_2$ has only four sides. The image of $\bs{\tau_2}$ can be studied by using similar arguments as in the general case and it is depicted in Figure \ref{fig:tauimagespecial}.
\begin{thm}[$\!\!$\cite{Compagnoni2013a}, Theorem 7.5]\label{th:imspecial}
Let us assume that $\mb{d_{31}} = \rho\,\mb{d_{21}},$ where $0<\rho<1.$ Then, the image of $ \bs{\tau_2} $ is the triangle $ T $ with vertices $R^1,R^2,R^3$ minus the open segment with endpoints $R^1,R^2.$
Moreover, given $\bs{\tau} \in \mbox{Im}(\bs{\tau_2}),$ we have
$$ \vert \bs{\tau_2}^{-1}(\bs{\tau}) \vert =
\left\{
\begin{array}{cl}
\infty & \mbox{ if } \bs{\tau}=R^1\ \text{or}\ R^2\\
2 & \mbox{ if } \bs{\tau} \in \mathring{T}\\
1 & \mbox{ otherwise} \end{array} \right. .$$
\end{thm}

\begin{figure}[htb]
\begin{center}
\resizebox{5.3cm}{!}{
  \includegraphics[bb=220 311 391 480]{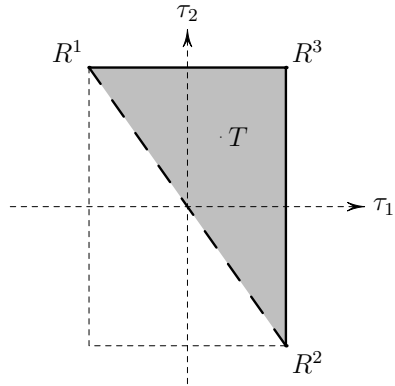}}
  \caption{\label{fig:tauimagespecial}The image of $ \bs{\tau_2}$
  under the assumption that $\m{3}$ lies on the segment between
  $\m{1}$ and $\m{2}$. In the gray region $T$ the map $ \bs{\tau_2} $
  is $ 2$--to--$1$. Along the horizontal and vertical sides of $T$ the
  map is $ 1$--to--$1,$ with the exception of the vertices $R^1,R^2$,
  where the fibers of $\bs{\tau_2}$ are not finite. Finally, the dashed
  side of $T$ is not in Im($\bs{\tau_2}$).  }
\end{center}
\end{figure}


\section{The comparison of the two localization models}\label{sec:TDOAvsTDOA}

In the previous section we found many similarities between the source localization problems involving TOA and TDOA measurements, respectively. Here, we investigate deeper into their relationship. In particular, we are going to define a natural map between the feasible sets of the two models, that allows us to recover all the information on the range differences from the knowledge of the ranges. The importance of the analysis contained in this section is twofold. First of all, from a theoretical perspective it gives an alternative way to the methodology proposed in \cite{Compagnoni2013a} for studying the range difference--based localization model. Moreover, this is the first step toward the comparison of the two localization models in presence of noisy measurements.

We start by recalling the definitions of the deterministic models:
$$
\text{TOA}: \left\{\begin{array}{l}
\TT_1(\x) = d_1(\x) \\
\TT_2(\x) = d_2(\x) \\
\TT_3(\x) = d_3(\x)
\end{array} \right. \qquad\qquad
\text{TDOA}: \left\{ \begin{array}{l}
\tau_1(\x) = d_1(\x) - d_3(\x) \\
\tau_2(\x) = d_2(\x) - d_3(\x)
\end{array} \right. .
$$
It follows easily that
\begin{equation} \label{T-tau-rel}
\left\{ \begin{array}{l}
\TT_1(\x) - \TT_3(\x) = \tau_1(\x) \\
\TT_2(\x) - \TT_3(\x) = \tau_2(\x)
\end{array} \right. .
\end{equation}
This proves the following:
\begin{thm}\label{th:projTOATDOA}
Let us consider the map
\begin{equation*}
\begin{array}{cccc}
\pi: & \RR^3 & \longrightarrow & \RR^2\\
& (\TT_1,\TT_2,\TT_3) & \longmapsto & (\TT_1-\TT_3,\TT_2-\TT_3)
\end{array}.
\end{equation*}
Then, $\bs{\tau_2}=\pi\circ\bs{\TT_3}$ and so $ \mbox{Im}(\bs{\tau_2}) = \pi(\mbox{Im}(\bs{\TT_3})). $
\end{thm}
For any given point $ \bt=(\tau_1,\tau_2) $ in the $ \tau$--plane, the fiber $\pi^{-1}(\bt)$ in the $\TT$--space is the line
\begin{equation}\label{eq:rettaTT}
r_{\bt}: \left\{ \begin{array}{l}
\TT_1 - \TT_3 = \tau_1 \\
\TT_2 - \TT_3 = \tau_2
\end{array} \right.
\qquad\Rightarrow\qquad
r_{\bt}: \left\{ \begin{array}{l}
\TT_1(t) = \tau_1 + t\\
\TT_2(t) = \tau_2 + t\\
\TT_3(t) = t
\end{array} \right.,\quad t \in \RR.
\end{equation}
In geometric terms, this means that $\pi$ is the projection map along the lines $r_{\bt}$ parallel to the vector $(1,1,1).$ Equivalently, if we embed the $\TT$--space in $\PP_\RR^3$, the map $\pi$ is the projection from the ideal point $q=(0:1:1:1).$ Moreover, from \eqref{eq:rettaTT} we have a natural correspondence between the coordinates $(\TT_1,\TT_2)$ and $(\tau_1,\tau_2)$ by fixing $t=0.$ This allows us to identify the $\tau$--plane with the coordinate plane $ \TT_3=0 $ in the $ \TT$--space.

In the next subsections we look into the consequences of Theorem \ref{th:projTOATDOA} separately for the general and the aligned configurations of the receivers. We are interested in the full exploration of the relationship that incurs between the sets $\text{Im}(\bs{\TT_3})$ and $\text{Im}(\bs{\tau_2}),$ as well as the one between the polyhedrons $Q_3$ and $P_2$ containing them.

\subsection{$\m{1},\m{2},\m{3}$ not collinear}

In this scenario, the polyhedron $Q_3$ is defined by the $12$ inequalities in \eqref{eq:ineq12}. We note that the facets of $Q_3$ given by the first $ 6 $ inequalities are parallel to $(1,1,1),$ hence $\pi$ projects them onto $6$ line segments in the $\tau$--plane that delimit a convex polytope. Indeed, under the map $\pi,$ the $6$ inequalities become
$$
-d_{21} \leq \tau_1 - \tau_2 \leq d_{21},\qquad
-d_{31} \leq \tau_1 \leq d_{31},\qquad
-d_{32} \leq \tau_2 \leq d_{32},
$$
which define the polygon $ P_2 $ (see Definition \ref{def:P2}). Any other inequality in \eqref{eq:ineq12} defines an half--space bounded by a plane that is not parallel to $(1,1,1),$ whose projection is the entire $\tau$--plane. This proves that the projection of $ Q_3 $ is $ P_2.$

$\text{Im}(\bs{\TT_3})$ is contained in the algebraic surface $ S $ defined by equation \eqref{eq:imm2}. The projective closure of $S$ is the Kummer's surface $\bar{S}$ having equation \eqref{Kum-proj}, that is written using the rescaled homogeneous coordinates \eqref{eq:rescaledcoord}. The $16$ singular points of $ \bar{S}$ are given in \eqref{eq:puntisingolari}. In particular, one of the ideal singular points is $ (0: \sqrt{d_{32}} : \sqrt{d_{31}} : \sqrt{d_{21}}).$ In the original (homogeneous) range coordinates $(0:\TT_1:\TT_2:\TT_3),$ this is exactly the point $q$ from which we have defined the projection $\pi$.

This fact has important consequences on the projection of $S.$ Indeed, every line $r_{\bs{\tau}}$ meets $\bar{S}$ at $q$ with multiplicity at least $2.$ Since $\bar{S}$ is a quartic surface, it follows that each $r_{\bs{\tau}}$ intersects $\bar{S}$ in at most two other points. More precisely, over the complex field we have the following situation:
\begin{itemize}[leftmargin=*]
\item
The tangent cone $V$ to $\bar{S}$ at $q$ has equation
\begin{equation} \label{tang-cone}
d_{32}^2 \TT_1^2 - 2 \mb{d_{31}} \cdot \mb{d_{32}} \TT_1 \TT_2
+ 2 \mb{d_{21}} \cdot \mb{d_{32}} \TT_1 \TT_3 + d_{31}^2 \TT_2^2 -
2 \mb{d_{21}} \cdot \mb{d_{31}} \TT_2 \TT_3 + d_{21}^2 \TT_3^2 -
\Vert \mb{d_{31}} \wedge \mb{d_{32}} \Vert^2 = 0.
\end{equation}
In the affine space, $V$ is an elliptic cylinder. Its axis contains the origin and it is parallel to $(1,1,1)$. By construction, the lines $ r_{\bs{\tau}} $ on $V$ meet $\bar{S}$ at $q$ with multiplicity $3,$ therefore they intersect $S$ at exactly one point.
\item
The lines $ r_{\bs{\tau}} $ on the tropes $\TT_i - \TT_j = \pm d_{ji},\ 1\leq i< j\leq 3,$ are tangent to $\bar{S},$ therefore they intersect $S$ at exactly one point with multiplicity $2.$
\item
In any other case, the line $ r_{\bs{\tau}} $ meets $S$ at two distinct point.
\end{itemize}

In the rest of our analysis, we use the identification of the plane $\TT_3=0$ and the $\tau$--plane. This way, the projection of $V$ can be calculated by substituting $\TT_3=0$ in Equation \ref{tang-cone}. We get
$$
d_{32}^2 \tau_1^2 - 2\, \mb{d_{31}} \cdot \mb{d_{32}}\, \tau_1 \tau_2
+ d_{31}^2 \tau_2^2 - \Vert \mb{d_{31}} \wedge \mb{d_{32}} \Vert^2
= 0,
$$
that defines the ellipse $ E $ tangent to every facet of $ P_2 $ (see Definition \ref{def:atau} and Proposition \ref{atau}). The projection of the tropes parallel to $(1,1,1)$ are the lines supporting $\partial P_2.$

Theorem \ref{th:projTOATDOA} states that the structure of $\text{Im}(\bs{\tau_2})$ depends on the intersection $ r_{\bs{\tau}}\cap\text{Im}(\bs{\TT_3}).$ This fact can be seen directly by substituting the parametric description \eqref{eq:rettaTT} of $ r_{\bs{\tau}} $ into Equation \eqref{eq:imm2} that defines $S.$ We obtain the following quadratic equation in $t:$
\begin{equation}\label{eq:TDOAeqfromTOA}
\begin{array}{c}
t^2 (\Vert \tau_1 \mb{d_{32}} - \tau_2 \mb{d_{31}} \Vert^2 -
\Vert \mb{d_{31}} \wedge \mb{d_{32}} \Vert^2) -
4 t \langle \tau_1\mb{d_{32}} - \tau_2 \mb{d_{31}}, \Vert D_{31}(\bs{\tau}) \Vert^2\mb{d_{32}} - \Vert D_{32}(\bs{\tau}) \Vert^2 \mb{d_{31}} \rangle+\\
+ \left\Vert \Vert D_{31}(\bs{\tau}) \Vert^2 \mb{d_{32}} - \Vert
D_{32}(\bs{\tau}) \Vert^2 \mb{d_{31}} \right\Vert^2 = 0.
\end{array}
\end{equation}
If we set the parameter $ t = - \lambda \Vert \mb{d_{31}} \wedge \mb{d_{32}}\Vert,$ Equation \eqref{eq:TDOAeqfromTOA} matches Equation \eqref{eq-deg-2}. Moreover, since $\text{Im}(\bs{\TT_3})$ is contained in the first octant of the $\TT$--space, we have to search for the non negative solutions $t$ of \eqref{eq:TDOAeqfromTOA}, or equivalently, for the non positive solutions $\lambda.$ This way we come back to the same problem addressed in \cite{Compagnoni2013a} in order to compute the image of $ \bs{\tau_2}$ (see Section \ref{sec:TDOA-summary-general}).

Now, we can interpret the results given in Theorem \ref{th:imtau2} and depicted in Figure \ref{fig:tauimage} from the point of view of the projection map $\pi.$ A visual illustration of this discussion is given in Figure \ref{fig:TOA-TDOA}. Preliminary, let us observe that the global injectivity of $\bs{\TT_3}$ implies that
$$
|\bs{\tau_2}^{-1}(\bt)|=|(\pi\circ\bs{\TT_3})^{-1}(\bt)|=|\pi^{-1}(\bt)|.
$$
\begin{itemize}[leftmargin=*]
\item
First of all, since $\text{Im}(\bs{\TT_3})=S\cap Q_3,$ the intersection is not empty only if $\bt\in P_2.$ Thus, $\text{Im}(\bs{\tau_2})\subseteq P_2.$
\item
Let $\bt$ be on the facet $F_1^+$ of $P_2,$ as the other cases are similar. If $\bt$ lies between $T_1^+$ and $R^3,$ then $r_{\bt}$ intersects once $\text{Im}(\bs{\TT_3})$ along the unbounded arc of hyperbola $\bs{\TT_3}(r_1^+).$ In this case the localization is unique and the preimage $\bs{\tau_2}^{-1}(\bt)$ lies on the degeneracy locus $D$ of the range difference map. On the other hand, for $\bt$ lying on the complementary segment in $F_1^+,$ the line $r_{\bt}$ intersects $S$ in the negative octant and so $\bt\notin\text{Im}(\bs{\tau_2}).$
\item
If $\bt\in U,$ then $ r_{\bt} $ meets $\text{Im}(\bs{\TT_3})$ at two distinct points and so the TDOA map is $2$--to--$1$. As $\bt$ approaches $E,$ the line $ r_{\bt} $ comes close to the cylinder $V$ and one of the two intersection points goes to infinity. This confirms the fact that the ellipse $E$ is the transition curve that bounds the set of range differences having unique localization.
\item
If $\bt\in E^-,$ then $ r_{\bs{\tau}}$ stays in the interior of the cylinder $ V$ and meets $ \mbox{Im}(\bs{\TT_3}) $ at one point, the other intersection point with $ S $ being in the negative octant. In this situation, the range difference map is $1$--to--$1.$
\item
The remaining points $\bt\in P_2$ define lines $ r_{\bs{\tau}}$ in $Q_3$ that intersect $S$ at two points having negative coordinates, therefore $\bt\notin\mbox{Im}(\bs{\tau_2}).$
\end{itemize}
\begin{figure}[htb!]
\centering
 \includegraphics[width=8cm,bb=178 270 433 521]{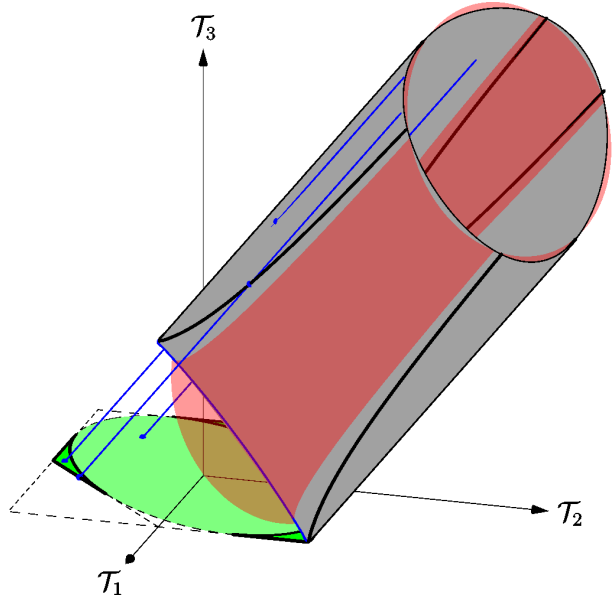}
 \caption{$\text{Im}(\bs{\tau_2})$ can be identified with the green region on the plane $\TT_3=0.$ It is the projection of $\text{Im}(\bs{\TT_3}),$ depicted in gray, along the blue lines $r_{\bt}.$ The light red surface is the tangent cone $V$ to the Kummer's surface $\bar{S}$ at the ideal point $q.$ We represent in transaparency $V,$ but not $\text{Im}(\bs{\TT_3}).$ The projection of $V$ is the ellipse $E.$ For $\bt\in E^-,$ the line $r_{\bt}$ stays in the interior of $V$ and intersects $\text{Im}(\bs{\TT_3})$ at one point. For $\bt\in U_i,\ i=1,2,3,$ the line $r_{\bt}$ stays in the exterior of $V$ and intersects $\text{Im}(\bs{\TT_3})$ at two distinct points. If $\bt\in \partial P_2,$ the line $r_{\bt}$ meets $\text{Im}(\bs{\TT_3})$ at one point on the black hyperbola, with multiplicity two.}\label{fig:TOA-TDOA}
 \end{figure}
\begin{rem}
From the above analysis follows that in presence of range measurements, we have to switch toward the range difference--based localization when there is no information along the direction $(1,1,1)$ of the $\TT$--space. This is the case if the source is not synchronized with the receivers (e.g in GPS localization, where the clock of the GPS receiver is not synchronized with the ones of the satellites). By supposing that the measurements are affected by the same bias $b,$ we have that $\widehat{\TT}_i=\TT_i+b,\ i=1,2,3.$ If $b$ is completely unknown, then $\widehat{\bs\TT}$ could be anywhere on $r_{\bt},$ where $\bt=\pi(\widehat{\bs\TT}).$ However, if there are some information on $b,$ then the probability density function of $\widehat{\bs\TT}$ along $r_{\bt}$ is not uniform and one should use this fact to improve the accuracy of the source estimation. In geometric terms, this corresponds to find the best line $r_{\bt}$ for $\text{Im}(\bs{\TT_3}),$ with respect to a suitable metric.
\end{rem}
\begin{rem}
The geometric construction we used above is a classical tool in the study of algebraic varieties. In particular, also Kummer's surfaces have been investigated through their projection from a singular point onto a projective plane $\mathcal{H}$ (see \cite{hudson}, Chapter XVIII). Over $\CC$ it is known that this projection identifies any Kummer's surface $\mathcal{S}\subset\PP^3_\CC$ to a double cover of $\mathcal{H}$, ramified along a sextic plane curve $\mathcal{C}$. This means that the line through the node and every point $p\in\mathcal{H}$ meets $\mathcal{S}$ at two other points, coinciding if, and only if, $p$ belongs to $\mathcal{C}.$

In the special case of a tetrahedroid surface one recovers some of the above results. In particular, $\mathcal{C}$ is the union of six lines tangent to a conic, which is obtained as the intersection of the tangent cone to $\mathcal{S}$ at the node and $\mathcal{H}$.

The novelty in our analysis is that we work over the field $\RR$ and we project only the subset $\text{Im}\bs(\TT_3)$ of the surface $S$ given by equation \eqref{eq:imm2}. As a consequence, the projection behaves in a different way, for example it is constrained by the bounded set $P_2$ and it is $1$--to--$1$ on the open set $E^-.$
\end{rem}
\begin{rem}
As done in Section \ref{sec:TOA-3D}, we can consider the TDOA--based localization for a source not contained in the plane of the receivers. Also this problem can be encoded into the TDOA map:
\begin{center}
$\begin{matrix} \bs{\tau_{3,2}}&: \mathbb{R}^{3}& \rightarrow & \mathbb{R}^{2}
\\ &\mb{x} &\mapsto& \left(\tau_1\left(\mb{x}\right), \tau_2\left(\mb{x}\right)\right)
\end{matrix}\ .$
\end{center}
As in Theorem \ref{th:projTOATDOA}, $\bs{\tau_{3,2}}=\pi\circ\bs{\TT_{3,3}}$ and so $\text{Im}(\bs{\tau_{3,2}})=\pi(\text{Im}(\bs{\TT_{3,3}}))=\text{Im}(\bs{\tau_2}).$ Given a feasible $\bt=(\tau_1,\tau_2),$ its inverse image $\bs{\tau_{3,2}}^{-1}(\bt)$ contains the points of the sensors plane $\x(\bt)$ given by equations \eqref{eq:inv-image}.

The admissible source position outside the plane can be computed from system \eqref{eq:mink-cone-planes} interpreted in the Minkowski space $\RR^{3,1}.$ The first two equations of the system define a two dimensional plane orthogonal to the receivers plane. It meets the three dimensional half--cone given by the third equation and the last inequality along a conic curve.
\end{rem}

\subsection{$\m{1},\m{2},\m{3}$ collinear}

As usual, we assume that $\m{3}\in r_3^0,$ therefore $\mb{d_{31}} = \rho\, \mb{d_{21}} $ with $ 0 < \rho < 1.$ The polyhedron $Q_3$ is defined by \eqref{eq:Q3s}. It has $4$ facets, three of which are parallel to $(1,1,1).$ The projection of $Q_3$ through $\pi$ is the triangle $P_2$ depicted in Figure \ref{fig:tauimagespecial}, with vertices $R^1,R^2,R^3$ and defined by inequalities
$$
\tau_1 \leq d_{31}, \qquad \tau_2 \leq d_{32}
\qquad (1-\rho)\tau_1 + \rho \tau_2 \geq 0.
$$

$\text{Im}(\bs{\TT_3})$ is contained in the one--sheet hyperboloid $ \sigma $ whose defining equation is
$$
(1-\rho) \TT_1^2 + \rho \TT_2^2 - \TT_3^2 - \rho(1-\rho) d_{21}^2 = 0.
$$
The point $q$ is in the projective closure $\bar\sigma$ of the quadric surface $\sigma$. The tangent plane $H$ to $\bar\sigma$ at $q$ has equation
$$
(1-\rho)\TT_1+\rho\TT_2-\TT_3=0.
$$
$H$ is parallel to $(1,1,1)$ and it supports the facet of $Q_3$ given by the last inequality in \eqref{eq:Q3s}. From Section \ref{sec:collinear-microphones}, it follows that $H$ intersects $\sigma$ along the two lines $l_1,l_2$ supporting the unbounded edges of $Q_3$ and containing the vertices $(0,d_{21},d_{31}),(d_{21},0,d_{32})$, respectively (see Figure \ref{fig:TOAdegsurf}).
Using these arguments and the property that every line $r_{\bt}$ meets $\bar\sigma$ at $q,$ it is simple to check the following facts about the projection of $H$ and $\sigma$.
\begin{itemize}[leftmargin=*]
\item
The lines $r_{\bt}$ containing $R^1$ and $R^2$ are $l_1$ and $l_2.$ As said above, they lie entirely on $\sigma$ and their projections are $\pi(l_i)=R^i,\ i=1,2.$ The projection $\pi(H)$ is the line in the $\tau$--plane passing through the two vertices $R^1,R^2.$
\item
The other lines $r_{\bt}$ on $H$ intersect $\bar\sigma$ only at $q$ with multiplicity $2,$ hence they do not meet $\sigma.$
\item
In any other case, the line $r_{\bt}$ meets $\sigma$ at one point.
\end{itemize}

We are now ready to recover all the results of Theorem \ref{th:imspecial} on the image of $\bs{\tau_2}$. At this purpose, we use again $\bs{\tau_2}^{-1}(\bt)=(\pi\circ\bs{\TT_3})^{-1}(\bt),$ together with Theorem \ref{headache2}:
\begin{align*}
\left| \bs{\TT}_{3}^{-1}\left(\bs{\TT}\right)\right| = \begin{cases}
2, \quad \text{if}\quad \bs{\TT} \in \mathring{Q}_{3} \cr
1, \quad \text{if}\quad \bs{\TT} \in \partial Q_{3} \end{cases} .
\end{align*}
\begin{itemize}[leftmargin=*]
\item
First of all, as $\text{Im}(\bs{\TT_3})$ is a subset of $Q_3,$ the points outside $P_2=\pi(Q_3)$ are not in $\text{Im}(\bs{\tau_2}).$
\item
On $\bt=R^1,R^2$ we have $|\pi^{-1}(\bt)|=\infty$ and so also $|\bs{\tau_2}^{-1}(\bt)|=\infty.$ The other points on the facet of $P_2$ containing $R^1,R^2$ are not in $\pi(\sigma),$ therefore they are not in $\text{Im}(\bs{\tau_2}).$
\item
Let $\bt$ lie on another facet of $P_2.$ Then, the line $r_{\bt}$ intersects $\sigma$ at one point on $\partial Q_3,$ thus $|\bs{\tau_2}^{-1}(q)|=1.$
\item
Finally, for any other $\bt\in P_2,$ the line $r_{\bt}$ intersects $\sigma$ at one point on $\mathring{Q}_3,$ thus $|\bs{\tau_2}^{-1}(\bt)|=2.$
\end{itemize}


\section{Conclusions and perspectives}\label{sec:conclusion}

In this manuscript we offered an exhaustive mathematical characterization of the deterministic models behind the range--based source localization, in the cases of two and three sensors. Our work is based on the analysis of the range maps $\bs{\TT_2}$ and $\bs{\TT_3}$. By assuming that the source is coplanar to the receivers, we studied the identifiability of the two models in both cases of an aligned and a generic configuration of the sensors. Then, we derived a complete characterization of Im$(\bs{\TT_2})$ and Im$(\bs{\TT_3}),$ that are the sets of feasible range measurements.
We found that Im$(\bs{\TT_2})$ is the unbounded convex polyhedron $Q_2$ of the $\TT$--plane. Assuming the sensors are not collinear, Im$(\bs{\TT_3})$ is the real portion of a Kummer's surface contained into the first octant of the $\TT$--space. By using some classical results on the Kummer's, we showed that Im$(\bs{\TT_3})$ is contained in the unbounded convex polyhedron $Q_3,$ strongly related to the convex hull $\mathcal{E}$ of Im$(\bs{\TT_3}).$ On the other hand, in the aligned sensors configuration we have that Im$(\bs{\TT_3})$ is a subset of a one sheet hyperboloid. In the last sections of the manuscript, we extended our treatment to the case of 3D localization and we used the above results to compare the deterministic range and range difference based localization models.

Our analysis followed the line drawn in \cite{Compagnoni2013a}. In particular, we used similar mathematical tools, that include multilinear algebra, the Minkowski space, algebraic and differential geometry. All along the paper we gave many comments about the relevance of our work for the deeper understanding and the resolution of several concrete problems concerning source localization. Here we come back to some of the relevant issues highlighted in the previous sections, that are currently under investigation by the authors. For the sake of simplicity, we consider only the localization with three non collinear receivers.

\emph{(1) Error propagation in presence of noisy range measurements.}
The description of the range model is instrumental for the analysis of the error propagation of whatever estimation algorithm of the source position. As observed in Remark \ref{rm:singmodels}, in the near field regime the behavior of the localization is particularly affected by the singular points of Im$(\bs{\TT_3})$. As a consequence, the usual investigation tools such as the asymptotic normal expansion could fail. To overcome such difficulties, we suggest to locally approximate Im$(\bs{\TT_3})$ at the singular points with their respective tangent cones, that are defined by equations similar to \eqref{eq:tgcone}. On the other hand, in the far field scenario the model is well approximated by the cylinder \eqref{tang-cone}. From a geometric standpoint, this is considerably simpler than the Kummer's quartic surface, thus it can be used to reduce the complexity of the problem.

\emph{(2) The estimation of the source position.}
In a noisy scenario, a vector of three range measurements $\widehat{\bs{\TT}}$ corresponds to a point in the $\TT$--space that lies close to the set Im($\bs{\TT_3}$).
The estimation of the source position is a two-step procedure: at first one finds the best $\overline{\bs{\TT}}\in\rm{Im}(\bs{\TT_3})$ from $\widehat{\bs{\TT}}$, then the source stays at $\bar\x=\bs{\TT_3}^{-1}(\overline{\bs{\TT}})$.
The point $\overline{\bs{\TT}}$ depends on the choice of the statistical estimator. In presence of Gaussian errors, the Fisher Information Matrix defines a Euclidean structure on the $\TT$--space. In this case, the Maximum Likelihood Estimation gives the point $\overline{\bs{\TT}}\in\rm{Im}(\bs{\TT_3})$ that is the orthogonal projection of $\widehat{\bs{\TT}}$ on the noiseless measurements set.
The resolution of the MLE for semialgebraic statistical models is a leading research problem in the area of Algebraic Statistics and, more in general, Applied Algebraic Geometry. With respect to the range based localization, some results on the Euclidean Degree of Cayley--Menger varieties are contained in \cite{Draisma2013}.
However, we observe that a CM variety is not the real set of noiseless range measurements, since it is defined in terms of the squared distances between the source and the receivers. We believe that the precise characterization of Im($\bs{\TT_3}$) and the several observations given in this manuscript (e.g. the description of the convex hull $\mathcal{E}$ and the comments contained at point $(1)$) will be surely helpful in the development of novel effective methods for solving MLE.

\emph{(3) Extension to general configurations of sources and receivers.}
The authors are currently working on the extension to arbitrary distributions of $r>3$ sensors and $n>1$ sources in the 3D Euclidean space, using similar techniques and notations. In particular, the general model can be encoded as well in a range map $\bs{\TT_{r,n}}$, whose image is again a real semialgebraic variety.
The description of $\bs{\TT_{r,n}}$ is needed also for the study of the localization with partially synchronized and calibrated receivers, that for example is the typical situation in wireless sensor network localization. In this scenario not all TOAs are available and, from the geometric point of view, this is equivalent to considering a projection of Im($\bs{\TT_{r,n}}$) to a smaller subspace of the measurements space (at this respect, see also the discussion in \cite{Compagnoni2013a}). On the basis of the results in Section \ref{sec:TDOAvsTDOA}, we finally deem that the analysis of $\bs{\TT_{r,n}}$ will simplify also the study of range difference--based localization model.


\section*{Acknowledgements}
The authors would like to thank R. Sacco for useful discussions
and suggestions during the preparation of this work.
\bigskip


\bibliographystyle{plain}   
\bibliography{biblio}       

\end{document}